\UseRawInputEncoding
\documentclass[12pt]{article}
\usepackage[centertags]{amsmath}
\usepackage{amsfonts}
\usepackage{amssymb}
\usepackage{latexsym}
\usepackage{amsthm}
\usepackage{newlfont}
\usepackage{graphicx}
\usepackage{listings}
\usepackage{booktabs}
\usepackage{abstract}
\usepackage{enumerate}
\usepackage{xcolor}
\RequirePackage{srcltx}
\date{}

\newlength{\defbaselineskip}
\newcommand{\setlinespacing}[1]%
           {\setlength{\baselineskip}{#1 \defbaselineskip}}

\newcommand{\actaqed}{\hfill $\actabox$}
{\medskip\noindent \textit{Proof of #1. }}%
{\actaqed \medskip}

\def\cA{{\mathcal A}}

\def\cC{{\mathcal C}}
\def\cD{{\mathcal D}}
\def\cE{{\mathcal E}}

\def\cH{{\mathcal H}}

\def\cM{{\mathcal M}}

\def\cR{{\mathcal R}}

\def\cT{{\mathcal T}}

\def\cV{{\mathcal V}}

\def\cX{{\mathcal X}}

\def\bbC{{\mathbb C}}

\def\bbN{{\mathbb N}}

\def\bbR{{\mathbb R}}

\def\bbT{{\mathbb T}}

\def\bbZ{{\mathbb Z}}

\def\bF{{\mathbf F}}

\def\bH{{\mathbf H}}

\def\bN{{\mathbf N}}

\def\bR{{\mathbf R}}

\def\bW{{\mathbf W}}

\def\ba{\mathbf a}
\def\bb{\mathbf b}
\def\bc{\mathbf c}

\def\bk{\mathbf k}

\def\bp{\mathbf p}
\def\bq{\mathbf q}
\def\br{\mathbf r}

\def\btt{\mathbf t}

\def\bx{\mathbf x}
\def\by{\mathbf y}
\def\bz{\mathbf z}
 
 \def \<{\langle}
\def\>{\rangle}
\def \La{\Lambda}

\def \e{\varepsilon}

\def \de{\delta}

\def \ff{\varphi}
\def\al{\alpha}
\def\bt{\beta}

\def\ga{\gamma}
\def\la{\lambda}

\def \conv{\operatorname{conv}}

\def \sp{\operatorname{span}}

\def \sign{\operatorname{sign}}

\def\bt{\beta}

\newtheorem{Theorem}{Theorem}[section]
\newtheorem{Lemma}{Lemma}[section]
\newtheorem{Definition}{Definition}[section]
\newtheorem{Proposition}{Proposition}[section]
\newtheorem{Remark}{Remark}[section]

\newtheorem{Corollary}{Corollary}[section]
\numberwithin{equation}{section}

\newcommand{\be}{\begin{equation}}
\newcommand{\ee}{\end{equation}}

\begin{document}

\title{Brief introduction in greedy approximation}

\author{  V. Temlyakov}

\newcommand{\Addresses}{{
  \bigskip
  \footnotesize


 \medskip
  V.N. Temlyakov, \textsc{University of South Carolina, USA,\\ Steklov Mathematical Institute of Russian Academy of Sciences, Russia;\\ Lomonosov Moscow State University, Russia; \\ Moscow Center of Fundamental and Applied Mathematics, Russia.\\  
  \\
E-mail:} \texttt{temlyakovv@gmail.com}

}}
\maketitle

\begin{abstract}{Sparse approximation is important in many applications because of concise form of an approximant and 
good accuracy guarantees. The theory of compressed sensing, which proved to be very useful in the image processing and data sciences, is based on the concept of sparsity. A fundamental issue of sparse approximation is the problem of construction of efficient algorithms, which provide good approximation. It turns out that greedy algorithms with respect to dictionaries are very good from this point of view. They are simple in implementation and there are well developed theoretical guarantees of their efficiency. This survey/tutorial paper contains brief description of different kinds of greedy algorithms and results on their convergence and rate of convergence. Also, Chapter IV gives some typical proofs of convergence and rate of convergence results for important greedy algorithms and Chapter V gives some open problems.	}
\end{abstract}

\newpage

 \medskip
 {\bf \Large Chapter I : Greedy approximation with respect to arbitrary dictionaries.}
  \medskip

\section{Introduction}
\label{In}

Sparse and greedy approximation theory is important in applications. This theory is actively developing for about 40 years. In the beginning researchers were interested in approximation in Hilbert spaces (see, for instance,  \cite{FS}, \cite{H}, \cite{J1}, \cite{J2}, \cite{B}, \cite{DT1}, \cite{DMA}, \cite{BCDD}; for detailed history see \cite{VTbook}). Later, this theory was extended to the case of real Banach spaces (see, for instance, \cite{DGDS}, \cite{VT80}, \cite{GK}, \cite{VT165}; for detailed history see \cite{VTbook} and \cite{VTbookMA}) and to the case of convex optimization (see, for instance, \cite{Cl}, \cite{FNW}, \cite{Ja2}, \cite{SSZ}, \cite{TRD}, \cite{Z}, \cite{DT}, \cite{GP}, \cite{DeTe}). The reader can find some open problems in the book \cite{VTbook}. 

We study greedy approximation with respect to arbitrary dictionaries in Banach spaces. 
It is known that in many numerical problems users are satisfied with a Hilbert space setting and do not consider a more general setting in a Banach space. There are known arguments (see \cite{VTbook}, p. xiii) that justify interest in Banach spaces. The first argument is an {\it a priori} argument that the spaces $L_p$ are very natural and should be studied along with the $L_2$ space. The second argument is an {\it a posteriori} argument. 
The study of greedy approximation in Banach spaces has discovered that the  characteristic of a Banach space $X$ that governs the behavior of greedy approximation is the {\it modulus of smoothness} $\rho(u)$ of $X$. It is known that the spaces $L_p$, $2\le p<\infty$ have modulo of smoothness of the same order: $u^2$. Thus, many results that are known for the Hilbert space $L_2$ and proved using some special structure of a Hilbert space can be generalized to Banach spaces $L_p$, $2\le p<\infty$. The new proofs use only the geometry of the unit sphere of the space expressed in the form $\rho(u) \le \gamma u^2$ (see \cite{VTbook} and \cite{VTbookMA}). The third and the most important argument is that the study of approximation problems in Banach spaces is a step in the direction of studying the convex optimization problems.
 It is pointed out in \cite{Z} that in many engineering applications researchers are interested in 
an approximate solution of an optimization problem  as a linear combination of a few elements from a given system $\cD$ of elements. There is an increasing interest in building such sparse approximate solutions using different greedy-type algorithms (see, for instance, \cite{VT140}, \cite{DT} and references therein).
 We refer the reader to the papers \cite{Ja2} and \cite{BRB} for concise surveys of some results on greedy algorithms from the point of view of convex optimization and signal processing.   At first glance the settings of 
approximation and optimization problems are very different. In the approximation problem an element   is given and our task is to find a sparse approximation of it. In optimization theory an energy function (loss function)  is given and we should find an approximate sparse solution to the minimization problem. It turns out that the same technique can be used for solving both problems (see, for instance, \cite{VT140} and \cite{DT}).
 For some more specific arguments in favor of studying greedy approximation in Banach spaces see \cite{ST}.
 
 {\bf Some basic concepts of Banach spaces.} We list here some concepts, which we use in this paper. For a nonzero element $h\in X$ we let $F_h$ denote a norming (peak) functional for $h$: 
$$
\|F_h\| =1,\qquad F_h(h) =\|h\|.
$$
The existence of such a functional is guaranteed by Hahn-Banach theorem. We denote the dual to $X$ space by $X^*$.

We consider here approximation in uniformly smooth Banach spaces. For a Banach space $X$ we define the modulus of smoothness
$$
\rho(u) :=\rho(u,X) := \sup_{\|x\|=\|y\|=1}\left(\frac{1}{2}(\|x+uy\|+\|x-uy\|)-1\right).
$$
The uniformly smooth Banach space is the one with the property
$$
\lim_{u\to 0}\rho(u)/u =0.
$$
It is easy to see that for any Banach space $X$ its modulus of smoothness $\rho(u)$ is an even convex function satisfying the inequalities
$$
\max(0,u-1)\le \rho(u)\le u,\quad u\in (0,\infty).
$$
It is well known (see, for instance, \cite{DGDS})  that in the case $X=L_p$, 
 $1\le p < \infty$ we have
\be\label{Lprho}
 \rho(u,L_p) \le \begin{cases} u^p/p & \text{if}\quad 1\le p\le 2 ,\\
(p-1)u^2/2 & \text{if}\quad 2\le p<\infty. \end{cases} 
\ee

 Let $X$ be a Banach space (real or complex) with norm $\|\cdot\|$. We say that a set of elements (functions) $\cD$ from $X$ is a dictionary if each $g\in \cD$ has norm bounded by one ($\|g\|\le1$) and the closure of $\sp \cD$ is $X$. We introduce a new norm, associated with a dictionary $\cD$, in the dual space $X^*$.
 
 \begin{Definition}\label{Dnorm} For $F\in X^*$ define
 $$
 \|F\|_\cD:=\sup_{g\in\cD}|F(g)|.
 $$
 \end{Definition}

 It is convenient for us to consider along with the dictionary $\cD$ its symmetrization. In the case of real Banach spaces we 
denote 
$$
\cD^{\pm} := \{\pm g \, : \, g\in \cD\}.
$$
In the case of complex Banach spaces we denote 
$$
\cD^{\circ} := \{e^{i\theta} g \, : \, g\in \cD,\quad \theta \in [0,2\pi)\}.
$$
In the above notation $\cD^{\circ}$ symbol $\circ$ stands for the unit circle. 

Following standard notations denote 
$$
A_1^o(\cD) := \left\{f \in X\,:\, f =\sum_{j=1}^\infty a_j g_j,\quad \sum_{j=1}^\infty |a_j| \le 1, \quad g_j \in \cD,\quad j=1,2,\dots \right\}
$$
and by $A_1(\cD)$ denote the closure (in $X$) of $A_1^o(\cD)$. Then it is clear that $A_1(\cD)$ is the closure (in $X$) of the convex hull of $\cD^\pm$ in the real case and of $\cD^{\circ}$ in the complex case. Also, denote 
by $\conv(\cD)$ the closure (in $X$) of the convex hull of $\cD$.  

Note, that clearly in the real case $\|F\|_\cD = \sup_{\phi\in\cD^\pm}F(\phi)$ and in the complex case $\|F\|_\cD = \sup_{\phi\in\cD^{\circ}}Re(F(\phi))$.

  In this paper we study  greedy algorithms with respect to $\cD$. 
  We begin with a general description of a greedy algorithm. For brevity we write GA for greedy algorithm.

{\bf Generic GA.} GA is an iterative process. Each iteration of it consists of two steps -- 
a greedy step and an approximation step. GA works in a Banach space $X$ with respect to a given 
system (dictionary) $\cD$ of elements. For a given element $f\in X$ it builds a sequence of elements $\ff_1,\ff_2,\dots$ of 
the system $\cD$, a sequence of approximations $G_1,G_2,\dots$, and a sequence of residuals $f_k:= f-G_k$, $k=1,2,\dots$. For convenience we set $G_0=0$ and $f_0=f$. Suppose we have performed $m-1$ iterations of GA. 
Then, at the $m$th iteration we apply the greedy step of our specific GA. The greedy step provides a new element 
$\ff_m\in \cD$ to be used for approximation. After that, using $\ff_m$ and the information from the previous $m-1$ iterations, we apply the approximation step of our specific GA and obtain the approximation $G_m$. 
Specifying the greedy steps and the approximation steps, we obtain different greedy algorithms. 
Below, we present some of the most basic examples of greedy and approximation steps. 

{\bf Weakness sequence.} This concept is useful in the definition of some greedy algorithms. Let $\tau := \{t_k\}_{k=1}^\infty$ be a given sequence of numbers $t_k\in [0,1]$, $k=1,2,\dots$. We call such a sequence the {\it weakness sequence}. In the case $\tau = \{t\}$, i.e. $t_k=t$, $k=1,2\dots$, we write $t$ instead of $\tau$ in the notation.

{\bf Relaxation sequence.} Let $\br := \{r_k\}_{k=1}^\infty$ be a given sequence of numbers $r_k\in [0,1)$, $k=1,2,\dots$. We call such a sequence the {\it relaxation sequence}. 

We use $C$, $C'$ and $c$, $c'$ to denote various positive constants. Their arguments indicate the parameters, which they may depend on. Normally, these constants do not depend on a function $f$ and running parameters $m$, $v$, $u$. We use the following symbols for brevity. For two nonnegative sequences $a=\{a_n\}_{n=1}^\infty$ and $b=\{b_n\}_{n=1}^\infty$ the relation $a_n\ll b_n$ means that there is  a number $C(a,b)$ such that for all $n$ we have $a_n\le C(a,b)b_n$. Relation $a_n\gg b_n$ means that 
 $b_n\ll a_n$ and $a_n\asymp b_n$ means that $a_n\ll b_n$ and $a_n \gg b_n$. 
 For a real number $x$ denote $[x]$ the integer part of $x$, $\lceil x\rceil$ -- the smallest integer, which is 
 greater than or equal to $x$.

 \subsection{Greedy steps}
 \label{gs}
 
 The main step of a greedy algorithm is its greedy step. At a greedy step after $m-1$ iterations of the algorithm we choose a new element $\ff_m \in \cD$ to be used at the $m$ iteration of the algorithm. We discuss three different types of greedy steps. 
 
 {\bf $X$-greedy step.} It is the most straight forward greedy step. We define the following mapping (we assume existence of such an element)
 \be\label{gs1}
 \ff(f):=\ff^X(f) := XGS(f) := \text{arg}\min_{\phi\in\cD} \left(\inf _\lambda\|f-\lambda \phi\|_X\right).
 \ee
 
 {\bf Dual greedy step.} This step is based on the concept of norming functional (we assume existence of such an element)
  \be\label{gs2}
 \ff(f):=\ff^D(f) := DGS(f) := \text{arg}\max_{\phi\in\cD} |F_f(\phi)|.
 \ee
 
 {\bf Weak dual greedy step with parameter $t\in [0,1]$.} Denote by 
 \be\label{gs3}
 \ff(f) := \ff^{W,t}(f) := W_tGS(f)
 \ee
 any element from $\cD$, which satisfies the inequality
 \be\label{gs4}
 |F_f(\ff(f))| \ge t\sup_{\phi \in \cD} |F_f(\phi)|.
 \ee 
 
 {\bf Thresholding step with parameter $\de$.} Denote by 
 \be\label{gs5}
 \ff(f) := \ff^{T,\de}(f) := T_\de GS(f)
 \ee
 any element from $\cD$, which satisfies the inequality
 \be\label{gs6}
 |F_f(\ff(f))| \ge \de.
 \ee 
 
 \subsection{Approximation steps}
 \label{as}
 
 At the second step (after the greedy step) at each iteration of a greedy algorithm we perform an approximation step. 
 We list the most important approximation steps. For a linear finite dimensional subspace $Y$ of $X$ denote 
 by $P_Y$ the Chebyshev projection mapping (it may not be unique):
 \be\label{as1}
 P_Y(f) := \text{arg}\min_{y\in Y} \|f-y\|_X.
 \ee
 
 {\bf Best one dimensional approximation step.}  In the case $Y=\{\la \ff, \la \in \bbR $ or $\la \in \bbC\}$ is the one dimensional subspace spanned by an element $\ff$ we write $P_\ff$ instead of $P_Y$. It is the best one dimensional approximation of $f$ 
 with respect to $\ff$.
 
 {\bf Free relaxation step.} For a given element $f$, its approximation $g$, and an element $\ff$ define
  \be\label{as2}
 FR(f,g,\ff) := (1-w^*)g +\la^* \ff,  
 \ee
 where
  \be\label{as3}
 \|f-((1-w^*)g +\la^* \ff)\|_X = \inf_{w,\la} \|f-((1-w)g +\la \ff)\|_X.
 \ee
 
 {\bf Fixed relaxation step.} Let $r\in [0,1)$ be a relaxation parameter. For a given element $f$, its approximation $g$, and an element $\ff$ define
  \be\label{as4}
 R_r(f,g,\ff) := (1-r)g +\la^* \ff,   
 \ee
 where
  \be\label{as5}
 \|f-((1-r)g +\la^* \ff)\|_X = \inf_{\la} \|f-((1-r)g +\la \ff))\|_X.
 \ee
 Clearly, $R_r(f,g,\ff) = (1-r)g + P_\ff(f-(1-r)g)$. 
 
 \subsection{Greedy algorithms}
 \label{ga}
 
 We now list some of the most basic greedy algorithms. 
 
 {\bf $X$-Greedy Algorithm (XGA).} This algorithm uses the $X$-greedy step (\ref{gs1}) and the best one dimensional approximation step. 
 
 {\bf Dual Greedy Algorithm (DGA).} This algorithm uses the dual greedy step (\ref{gs2}) and the best one dimensional approximation step. 
 
 {\bf Weak Dual Greedy Algorithm  (WDGA).} Let $\tau := \{t_k\}_{k=1}^\infty$ be a given sequence of numbers $t_k\in [0,1]$, $k=1,2,\dots$, which we call the {\it weakness sequence}. At the $m$th iteration ($m=1,2,\dots$) this algorithm uses the weak dual greedy step with parameter $t_m$ (see (\ref{gs3} and (\ref{gs4})) and the best one dimensional approximation step. For the reader's convenience we give the definition of this 
 important algorithm in detail. Suppose that after $m-1$ iterations we have built $f_{m-1}$. Then we look for an element 
 $\ff_m \in \cD$ satisfying the inequality (in the case $t_m=1$ we assume existence)
 \be\label{ga1}
 |F_{f_{m-1}}(\ff_m)| \ge t_m \sup_{\phi \in \cD} |F_{f_{m-1}}(\phi)|.
 \ee
We now define $\la_m$ as 
\be\label{ga2}
\|f_{m-1} -\la_m\ff_m\|_X = \min_{\la} \|f_{m-1} -\la\ff_m\|_X.
\ee
Finally, let
\be\label{ga3}
f_m := f_{m-1} -\la_m\ff_m.
\ee

Sometimes it is convenient for us to include the weakness sequence $\tau$ in the notation of WDGA. In such a case we write WDGA($\tau$).
Clearly, the DGA is a special case of the WDGA($\tau$), when $\tau = \{1\}$, i.e. $t_k=1$ for all $k$. 

The above defined algorithms XGA, DGA, and WDGA($\tau$) are natural and simple greedy algorithms. However,
it turns out that they have some drawbacks. First, these algorithms are difficult for studying. In many cases we do not 
have results on their convergence and on the right order of convergence on classes $A_1(\cD)$. Second, it is known that 
even in the case of $X$ being a Hilbert space in some cases these algorithms do not provide the optimal (in the sense of order) rate of decay of the error of best $m$-term approximation. These facts motivated researchers to study other greedy algorithms, which might be somewhat more complex than the above ones but are better in the sense of convergence and 
rate of convergence. We give definitions of some of them. 

 {\bf Weak Chebyshev Greedy Algorithm (WCGA).} Let $\tau$ be a weakness sequence. At the $m$th iteration ($m=1,2,\dots$) this algorithm uses the weak dual greedy step with parameter $t_m$ (see (\ref{gs3} and (\ref{gs4})) and the Chebyshev projection mapping (see (\ref{as1})). For the reader's convenience we give the definition of this 
 important algorithm in detail. Suppose that after $m-1$ iterations we have built $f_{m-1}$ and the elements $\ff_1,\ff_2,\dots, \ff_{m-1}$. Then we look for an element 
 $\ff_m \in \cD$ satisfying the inequality (\ref{ga1}) (in the case $t_m=1$ we assume existence).
  We now define $Y_m := \sp(\ff_1,\ff_2,\dots,  \ff_{m})$ and 
\be\label{ga4}
G_m := P_{Y_m}(f).
\ee
Finally, let
\be\label{ga5}
f_m := f-G_m.
\ee

{\bf  Weak Greedy Algorithm with Free Relaxation (WGAFR).} Let $\tau$ be a weakness sequence. At the $m$th iteration ($m=1,2,\dots$) this algorithm uses the weak dual greedy step with $f_{m-1}$ and parameter $t_m$ (see (\ref{gs3}) and (\ref{gs4})) and the free relaxation step (see (\ref{as2}) and (\ref{as3}) with $f_{m-1}$, $G_{m-1}$, and $\ff_m$.

Sometimes it is convenient for us to include the weakness sequence $\tau$ in the notation of WGAFR. In such a case we write WGAFR($\tau$).
Algorithms WGAFR, WCGA and the WDGA use the same greedy step (\ref{ga1}). This means that all of them fall in the category of dual greedy algorithms. The approximation step of the WGAFR is based on the idea of relaxation -- building the approximant  as a linear combination of the previous approximant $G_{m-1}$ and the new element $\ff_m$. Certainly, it makes the WGAFR more complex than the WDGA but much simpler than the WCGA.  

Let us discuss $X$-greedy type algorithms with free relaxation. The following one is known in the literature (see, for instance, \cite{VTbook}, p.379). 

 {\bf  $X$-Greedy Algorithm with Free Relaxation  (XGAFR). The first version.} 
  We define $f_0   :=f$ and $G_0  := 0$. Then for each $m\ge 1$ we have the following inductive definition.

 (1)  $\varphi_m\in\cD$ and $ \lambda_m $, $w_m$ are such that (we assume existence of such $\varphi_m$)
$$
\|f-((1-w_m)G_{m-1} + \la_m\varphi_m)\| = \inf_{g\in\cD, \la,w}\|f-((1-w)G_{m-1} + \la g)\|
$$
and 
$$
G_m:=   (1-w_m)G_{m-1} + \la_m\varphi_m.
$$

(2) Let
$$
f_m   := f-G_m.
$$

Sometimes we also use the notation XGAFR1 for the above algorithm in order to indicate that it is one of the 
modifications of the $X$-Greedy Algorithm with Free Relaxation. 
Here is a modification of the above greedy algorithm, which we denote by XGAFR2. It seems like this algorithm has not been defined and studied in the literature. 

{\bf  $X$-Greedy Algorithm with Free Relaxation  (XGAFR2). The second version.} This algorithm uses the $X$-greedy step (\ref{gs1}) and the free relaxation approximation step (see (\ref{as2}) and (\ref{as3}) with $f_{m-1}$, $G_{m-1}$, and $\ff_m$. 

Note that the XGAFR2 is an analog of the WGAFR -- they both use the same free relaxation approximation step. 
There is no dual greedy analogs of the XGAFR1. If we combine the ideas of the XGAFR1 and the WCGA then we 
obtain the algorithm, which provides at the $m$th iteration the best $m$-term approximant of $f$ under assumption that it exists. Namely,
$$
\sigma_m(f,\cD)_X =  \inf_{\phi_k\in\cD, \la_k, k=1,2,\dots,m}\|f- (\la_1 \phi_1+\cdots+\la_m\phi_m)\|_X.
$$

We now proceed to the greedy algorithms with fixed relaxation. The following algorithm is known and studied (see, for instance, \cite{VTbook}, p.354).

{\bf Greedy Algorithm with Weakness $\tau$ and Relaxation $\br$ (GAWR($\tau,\br$).} Let $\tau$ be a weakness sequence and $\br$ be a relaxation sequence. This algorithm uses the dual greedy step (\ref{gs2}) and the fixed relaxation approximation step (see (\ref{as4}) and (\ref{as5}). 

It is clear that in the case $\br = \mathbf 0$ the GAWR($\tau,\mathbf 0$) coincides with the WDGA($\tau$). 

{\bf $X$-Greedy Algorithm with Relaxation $\br$ (XGAR($\br$).} Let $\br$ be a relaxation sequence. This algorithm uses the  $X$-greedy step (\ref{gs1}) and the fixed relaxation approximation step (see (\ref{as4}) and (\ref{as5}) with $f_{m-1}$, $G_{m-1}$, and $\ff_m$. 

It is clear that in the case $\br = \mathbf 0$ the XGAR($\mathbf 0$) coincides with the XGA.

\section{Convergence of greedy algorithms}
\label{con}

We now formulate some known results on convergence of the above defined algorithms in real Banach spaces. We will use the following convenient terminology.

{\bf Convergence.} We say that an algorithm GA converges in a Banach space $X$ if for any dictionary $\cD$ and each $f\in X$ we have for any realization of the algorithm GA the following property 
$$
\lim_{m\to \infty} \|f_m\|_X =0.
$$

 We begin with the most studied algorithms WCGA and WGAFR. We  proceed to a theorem on convergence of WCGA and WGAFR. In the formulation of this theorem we need a special sequence, which is defined for a given modulus of smoothness $\rho(u)$ and a given weakness sequence $\tau = \{t_k\}_{k=1}^\infty$.
\begin{Definition}\label{xi} Let $\rho(u)$ be an even convex function on $(-\infty,\infty)$ with the property: $\rho(2) \ge 1$ and
$$
\lim_{u\to 0}\rho(u)/u =0.
$$
For any $\tau = \{t_k\}_{k=1}^\infty$, $0<t_k\le 1$, and $0<\theta\le 1/2$
 we define $\xi_m := \xi_m(\rho,\tau,\theta)$ as a number $u$ satisfying the equation
\be\label{2.1}
\rho(u) = \theta t_m u.  
\ee
\end{Definition}
\begin{Remark}\label{xim} Assumptions on $\rho(u)$ imply that the function
$$
s(u) := \rho(u)/u, \quad u\neq 0,\quad s(0) =0,
$$
is a continuous increasing function on $[0,\infty)$   with $s(2)\ge 1/2$. Thus (\ref{2.1}) has a unique solution $\xi_m=s^{-1}(\theta t_m)$ such that $0<\xi_m\le 2$.  
\end{Remark}
The following theorem (see, for instance, \cite{VTbook}, p.341 and p.352) gives a sufficient condition for convergence of  WCGA and WGAFR. 
\begin{Theorem}\label{conT1} Let $X$ be a uniformly smooth Banach space with modulus of smoothness $\rho(u)$. Assume that a sequence $\tau :=\{t_k\}_{k=1}^\infty$ satisfies the condition: for any $\theta >0$ we have
$$
\sum_{m=1}^\infty t_m \xi_m(\rho,\tau,\theta) =\infty.
$$
 Then, both the WCGA($\tau$) and the WGAFR($\tau$) converge in $X$. 
 \end{Theorem}
\begin{Corollary}\label{conC1} Let a Banach space $X$ have modulus of smoothness $\rho(u)$ of power type $1<q\le 2$, that is, $\rho(u) \le \gamma u^q$. Assume that 
\be\label{2.2}
\sum_{m=1}^\infty t_m^p =\infty, \quad p=\frac{q}{q-1}.  
\ee
Then, both the WCGA($\tau$) and the WGAFR($\tau$) converge in $X$. 
\end{Corollary}

\begin{Remark}\label{conR2} It is known that the condition (\ref{2.2}) is sharp in the class of Banach spaces with $\rho(u,X) \le \gamma u^q$ (see, for instance, \cite{VTbook}, p.346).
\end{Remark}

Here is a known result on convergence of the GAWR($t,\br$) (see, for instance, \cite{VTbook}, p.355). 

\begin{Theorem}\label{conT2} Let a sequence $\mathbf{r}$ satisfy the conditions
$$
\sum_{k=1}^\infty r_k =\infty,\quad r_k\to 0\quad\text{as}\quad k\to\infty.
$$
Then the GAWR($t,\mathbf{r}$)  converges in any uniformly smooth Banach space.  
\end{Theorem} 

Theorems \ref{conT1} and \ref{conT2} in the important special case $\tau =\{t\}$ provide convergence of those algorithms 
in any uniformly smooth Banach space. We do not have results of that kind for the simplest greedy algorithms WDGA and XGA. There is known result about convergence of the WDGA under an extra assumption on a Banach space (see \cite{GK} and \cite{VTbook}, p.376). 

\begin{Definition}\label{conD1} {\bf (Property $\Gamma$).} A uniformly smooth Banach space has property $\Gamma$ if there is a constant $\beta>0$ such that, for any $x,y\in X$ satisfying $F_x(y)=0$, we have
$$
\|x+y\|\ge\|x\|+\beta F_{x+y}(y).
$$
\end{Definition}
Property $\Gamma$ in the above form was introduced in \cite{GK}. This condition (formulated somewhat differently) was considered previously in the context of greedy approximation in \cite{Li1}.
\begin{Theorem}[{\cite{GK}}]\label{conT3}  Let $X$ be a uniformly smooth Banach space with property $\Gamma$. Then the WDGA($t$) with $t\in(0,1]$, converges in $X$.
\end{Theorem}
  
\begin{Proposition}[{\cite{GK}}]\label{conP1}  The $L_p$-space with $1<p<\infty$ has property $\Gamma$.
\end{Proposition}  

 We now give some known results on the $X$-type greedy algorithms (see, for instance, \cite{VTbook}, pp.378--381. 
 
 \begin{Theorem}\label{conT4} The XGAFR converges in any uniformly smooth Banach space.  
$$
\lim_{m\to \infty} \|f_m\| =0.
$$
\end{Theorem}

 \begin{Theorem}\label{conT5} Let a sequence $\mathbf{r}:=\{r_k\}_{k=1}^\infty$, $r_k\in[0,1)$, satisfy the conditions
$$
\sum_{k=1}^\infty r_k =\infty,\quad r_k\to 0\quad\text{as}\quad k\to\infty.
$$
Then the   XGAR($\mathbf{r}$) converges in any uniformly smooth Banach space.  
\end{Theorem}

\subsection{Convergence of the WGA in the Hilbert space}
\label{H}

We pointed out above that the problem of convergence of the simplest greedy algorithms WDGA and XGA is not satisfactory studied in the Banach space setting. In the special case of convergence in a Hilbert space the situation is much better. We now present the corresponding results.  
We begin with a discussion of the Pure Greedy Algorithm (PGA), which is the XGA working in a Hilbert space $H$. Then, the greedy step of the PGA is the following:   At the $m$th iteration we look for an element $\ff_m\in\cD$ and a number $\la_m$ satisfying (we assume existence) 
\be\label{1.1}
\|f_{m-1}-\la_m\ff_m\|_H=\inf_{\phi\in\cD,\la}\|f_{m-1}-\la \phi\|_H.  
\ee

We now explain that in the case of a Hilbert space the PGA coincides with the WDGA($\tau$), when $\tau = \{1\}$. 
Note that in the case of a Hilbert space the WDGA($\tau$) is called the Weak Greedy Algorithm with the weakness sequence $\tau$ (WGA($\tau$)). 
Indeed, in a Hilbert space a norming functional $F_f$ acts as follows
$$
F_f(g)=\<f/\|f\|,g\>.
$$
The weak dual greedy step with parameter $t_m$ is equivalent to the following step: We look for an element $\ff_m\in\cD$ such that
$$
|\<f_{m-1},\ff_m\>|\ge t_m\sup_{\phi\in\cD}\|<f_{m-1},\phi\>|,
$$
which means that in the case $t_m=1$ we have
\be\label{1.2}
|\<f_{m-1},\ff_m\>|=\sup_{\phi\in\cD}|\<f_{m-1},\phi\>|. 
\ee
Clearly, (\ref{1.1}) and (\ref{1.2}) give the same $\ff_m$, which is understood in the following way. If $\ff_m$ satisfies (\ref{1.1}) then it satisfies (\ref{1.2}) and vice versa. Thus,
in a Hilbert space both versions (\ref{1.1}) and (\ref{1.2}) result in the same PGA. 
 
 We now formulate some known results about convergence of the WGA($\tau$) in a Hilbert space. 
 
 We proved in \cite{VT82} a criterion on $\tau$ for convergence of the WGA($\tau$). Let us introduce some notation.
We define by $\cV$ the class of sequences $x=\{x_k\}_{k=1}^\infty$, $x_k\ge 0$, $k=1,2,\dots$, with the following property: there exists a sequence $0=q_0<q_1<\dots$ that may depend on $x$ such that
\be\label{2.9}
\sum_{s=1}^\infty \frac{2^s}{\Delta q_s} <\infty  
\ee
and
\be\label{2.10}
\sum_{s=1}^\infty 2^{-s}\sum_{k=1}^{q_s} x_k^2 <\infty,  
\ee
where $\Delta q_s:=q_s-q_{s-1}$.

\begin{Proposition}\label{conP1} The following two conditions are equivalent:
\be\label{2.11}
\tau \notin \cV,  
\ee
\be\label{2.12}
\forall \{a_j\}_{j=1}^\infty \in \ell_2,\quad a_j\ge 0,\quad \liminf_{n\to \infty} a_n\sum_{j=1}^na_j/t_n =0.
\ee
\end{Proposition}

\begin{Theorem}\label{conT6}The condition $\tau \notin \cV$ is necessary and sufficient for convergence of the  Weak Greedy Algorithm with the weakness sequence $\tau$.  
\end{Theorem}

The following theorem gives a criterion of convergence in a special case of monotone weakness sequences $\tau$. Sufficiency was proved in \cite{VT75} and  necessity in \cite{LTe1}. 

\begin{Theorem}\label{conT7} In the class of monotone sequences $\tau = \{t_k\}_{k=1}^\infty$, $1\ge t_1\ge t_2 \ge\dots\ge 0$, the condition 
\be\label{2.13}
\sum_{k=1}^\infty \frac{t_k}{k} = \infty  
\ee
 is necessary and sufficient for convergence of the Weak Greedy Algorithm with the weakness sequence $\tau$.
 \end{Theorem}

\begin{Remark} We note that the sufficiency part of Theorem \ref{conT7} (see \cite{VT75}) does not need the monotonicity of $\tau$.
\end{Remark}

\section{Rate of convergence of greedy algorithms}
\label{rc}

In this section we discuss the rate of convergence of some greedy algorithms with respect to a given dictionary $\cD$ for elements from the class $A_1(\cD)$. We will use the following convenient terminology.

{\bf Rate of convergence.} Let $rc$ be a sequence $\{rc(m)\}_{m=1}^\infty$ of nonnegative numbers.    We say that an algorithm GA has rate $rc$ of convergence in a Banach space $X$ if for any dictionary $\cD$ and each $f\in A_1(\cD)$ we have for any realization of the algorithm GA the following bound 
\be\label{rc1}
  \|f_m\|_X \le C(GA)rc(m),\quad m=1,2\dots,
\ee
where $C(GA)$ is a positive constant, which may depend on parameters of the Banach space $X$ and the algorithm GA
but does not depend on $f$, $m$, and $\cD$. 

 As in Section \ref{con} we begin with the most studied algorithms WCGA and WGAFR. We  proceed to a theorem on the rate of convergence of WCGA and WGAFR. For a weakness sequence $\tau$ and a parameter $p\in [2,\infty)$ denote the error sequence $e(\tau,p)$ as follows
 $$
 e(\tau,p)(m) := \left(1+\sum_{k=1}^m t_k^p\right)^{-1/p}.
  $$
  
  The following Theorem \ref{rcT1} is known (see, for instance, \cite{VTbook}, p.342 and p.353).
  
  \begin{Theorem}\label{rcT1} Let $X$ be a uniformly smooth Banach space with  modulus of smoothness $\rho(u) \le \gamma u^q$, $1<q\le 2$. Then, both the WCGA($\tau$) and the WGAFR($\tau$) have rate $e(\tau,p)$ of convergence with $p:= \frac{q}{q-1}$ being the dual to $q$ and the constant $C(GA)$, which may only depend on $q$ and $\gamma$. 
\end{Theorem}

The following Remark \ref{rcR1} in the case of WCGA($\tau$) is from \cite{VTbookMA}, p.421. One can check that it also holds for the WGAFR($\tau$).

\begin{Remark}\label{rcR1} In Theorem \ref{rcT1} the constant $C(GA)$ can be taken as $C\gamma^{1/q}$ with an absolute constant $C$.
\end{Remark}

We can formulate the rate of convergence problem as an extremal problem. Let $\cX$ be a collection of Banach spaces, for instance, for  fixed $1<q\le 2$ and $\ga >0$ define
$$
\cX(\ga,q) := \{X\,:\, \rho(u,X) \le \gamma u^q \}.
$$
Then for a specific greedy algorithm GA define 
$$
er_m(\cX,GA) := \sup_{X\in \cX} \sup_{\cD} \sup_{f\in A_1(\cD)} \sup_{realizations\, of\, GA} \|f_m\|_X.
$$
Note that the supremum over realizations of GA indicates that a realization of a specific GA may not be unique even in 
the case of algorithms without weakness in their greedy steps. Thus, Theorem \ref{rcT1} and Remark \ref{rcR1} 
give the following upper bounds
\be\label{rc2}
er_m(\cX(\ga,q),WCGA(\tau)) \le C\gamma^{1/q} \left(1+\sum_{k=1}^m t_k^p\right)^{-1/p},\quad p:= \frac{q}{q-1}
\ee
and
\be\label{rc3}
er_m(\cX(\ga,q),WGAFR(\tau)) \le C\gamma^{1/q} \left(1+\sum_{k=1}^m t_k^p\right)^{-1/p},\quad p:= \frac{q}{q-1}.
\ee

Remark \ref{conR2} shows that the condition (\ref{2.2}) is sharp for convergence in the class $\cX(\ga,q)$ of Banach spaces. However, we do not know if the upper bounds (\ref{rc2}) and (\ref{rc3}) are sharp.

We now present a rate of convergence result for the GAWR($t,\br$). 

\begin{Theorem}[{\cite{VT115}}]\label{rcT2} Let $X$ be a uniformly smooth Banach space with modulus of smoothness $\rho(u)\le \gamma u^q$, $1<q\le 2$. Let 
$\br:=\{2/(k+2)\}_{k=1}^\infty$.   Then, we have for the GAWR($t,\br$) for any $f\in A_1(\cD)$
$$
\|f_m\|\le C(t,q,\gamma) m^{-1+1/q}.
$$
\end{Theorem}

{\bf First conclusions.} Let us discuss the case $\tau =\{t\}$. In this case by Theorem \ref{rcT1} both the WCGA($t$) and the WGAFR($t$) have rate $(1+mt^p)^{-1/p}$ of convergence with $p:= \frac{q}{q-1}$. This means that, for instance,
\be\label{rc4}
er_m(\cX(\ga,q),WCGA(t)) \le C(t,q) m^{1/q-1}.
\ee
Therefore, the weakness parameter $t$ only affects the constant but not the rate of decay with respect to $m$. 
In particular, for the WOGA in a Hilbert space it guarantees the rate of decay $m^{-1/2}$ with respect to $m$. 
Let us compare this with the rate of convergence of the WGA in a Hilbert space. The following result is known 
(see, for instance, \cite{VTbook}, p.94). 

\begin{Theorem}\label{rcT3} Let $\cD$ be an arbitrary dictionary in $H$. Assume $\tau :=\{t_k\}_{k=1}^\infty$ is a non-increasing sequence. Then, for $f \in A_1(\cD)$ we have for the WGA($\tau$)
\be\label{3.17}
\|f_m\| \le \left(1+\sum_{k=1}^m t^2_k\right)^{-\frac{t_m}{2(2+t_m)}}.  
\ee
\end{Theorem}
In a particular case $\tau =\{t\}$ (\ref{3.17}) gives
$$
\|f_m\| \le (1+mt^2)^{-\frac{t}{4+2t}}, \quad 0<t\le 1. 
$$
This estimate implies the  inequality 
$$
\|f_m\| \le C_1(t)m^{-at},   
$$
with the exponent $at$ approaching $0$ linearly in $t$. It was proved in \cite{LTe2}   that this exponent cannot decrease to $0$ at a slower rate than linear.

\begin{Theorem}[{\cite{LTe2}}]\label{rcT4}   There exists an absolute constant $b>0$ such that, for any $t>0$, we can find a dictionary $\cD_t$ and a function $f_t \in A_1(\cD_t)$ such that, for some realization $G_m^t(f_t,\cD_t)$ of the Weak Greedy Algorithm with weakness parameter $t$, we have
$$
\liminf_{m\to \infty} \|f_t -G_m^t(f_t,\cD_t)\|m^{bt}  >0. 
$$
\end{Theorem}

We now compare three algorithms WDGA($t$), WCGA($t$), and WGAFR($t$). The advantages of the WDGA are the following: 

(A) It provides an expansion of an element $f$ with respect to the dictionary $\cD$;

(B) It has the simplest out of these three algorithms approximation step. 

The disadvantage of the WDGA($t$) (on the example of the WGA($t$)) is the following:

(C) Its rate of convergence depends heavily on the weakness parameter $t$ and is worse than the rate of convergence
of the other two algorithms.

The WGAFR($t$) is close to the WDGA($t$) in the sense of simplicity (B), but it does not provide an expansion. In the sense of rate of convergence both the WCGA($t$) and WGAFR($t$) are similar and better than the WDGA($t$). 
The WCGA($t$) has the most involved approximation step. The above comparison motivated researches to 
try to modify the WDGA($t$) in such a way that the modification keeps the advantages (A) and (B) and improves on 
the rate of convergence. We now discuss some of these modifications. 

\section{Some modifications of XGA and WDGA}
\label{mo}

\subsection{GA with prescribed in advance coefficients}

We begin with somewhat unexpected results. As we already mentioned above we would like to keep the property (A). 
From the definition of a dictionary it follows that any element $f\in X$ can be approximated arbitrarily well by finite linear combinations of the dictionary elements. In this section we study representations of an element $f\in X$ by a series
\be\label{7.1}
f\sim \sum_{j=1}^\infty c_j(f)g_j(f), \quad g_j(f) \in \cD^{\pm},\quad c_j(f)>0, \quad j=1,2,\dots.  
\ee
In building the representation (\ref{7.1}) we should construct two sequences: \newline
  $\{g_j(f)\}_{j=1}^\infty$ and $\{c_j(f)\}_{j=1}^\infty$.  In this section the construction of $\{g_j(f)\}_{j=1}^\infty$ is based on ideas used in greedy-type nonlinear approximation (greedy-type algorithms). This justifies the use of the term {\it greedy expansion} for (\ref{7.1}) considered in the section.  The construction of $\{g_j(f)\}_{j=1}^\infty$ is, clearly, the most important and difficult part in building the representation (\ref{7.1}). 
In the paper \cite{VT116} (see also \cite{VTbook}, S.6.7) we pushed to the extreme the flexibility choice of the coefficients $c_j(f)$ in (\ref{7.1}). We made these coefficients independent of an element $f\in X$. Surprisingly, for properly chosen coefficients we obtained results for the corresponding dual greedy expansion.   Even more surprisingly, we obtained similar results for the corresponding $X$-greedy expansions. We proceed to the formulation of these results. 
For convenience, let $\cC:=\{c_m\}_{m=1}^\infty$ be a fixed sequence of positive numbers.  We restrict ourselves to positive numbers and use the symmetric dictionary $\cD^{\pm}$ instead of the $\cD$.  

 {\bf $X$-Greedy Algorithm with coefficients $\cC$ (XGA($\cC$)).} We define $f_0:=f$, $G_0:=0$. Then, for each $m\ge 1$ we have the following inductive definition.

(1) $\ff_m\in\cD$ is such that (we assuming existence)
$$
 \|f_{m-1}-c_m\ff_m\|_X=\inf_{g\in\cD^{\pm}}\|f_{m-1}-c_m g\|_X. 
$$

(2) Let
$$
f_m:=f_{m-1}-c_m\ff_m,\qquad G_m:=G_{m-1}+c_m\ff_m.
$$

It will be convenient for us to use the following notation (see Definition \ref{Dnorm}). Denote 
$$
  \|F_f\|_\cD:= \sup_{g\in \cD}|F_f(g)| =  \sup_{g\in \cD^{\pm}}F_f(g) \quad \text{and} \quad r_\cD(f) := \sup_{F_f}\|F_f\|_\cD.
$$
We note that, in general, a norming functional $F_f$ is not unique. This is why we take $\sup_{F_f}$ over all norming functionals of $f$ in the definition of $r_\cD(f)$. It is known that in the case of uniformly smooth Banach spaces (our primary object here) the norming functional $F_f$ is unique. In such a case we do not need $\sup_{F_f}$ in the definition of $r_\cD(f)$, we have $r_\cD(f)=\|F_f\|_\cD$.

  {\bf Dual Greedy Algorithm with weakness $\tau$ and coefficients $\cC$\newline (DGA($\tau,\cC$)).} Let $\tau$ be a weakness sequence. 
We define $f_0 :=  f$, $G_0:=0$. Then, for each $m\ge 1$ we have the following inductive definition.

(1) $\varphi_m  \in \cD^{\pm}$ is any element satisfying
$$
F_{f_{m-1}}(\varphi_m) \ge t_m  \|F_{f_{m-1}}\|_\cD. 
$$

(2) Let  
$$
f_m :=   f_{m-1}-c_m\varphi_m,\qquad G_m:=G_{m-1}+c_m\ff_m.
$$

As above, in   the case $\tau=\{t\}$, $t\in(0,1]$, we write $t$ instead of $\tau$ in the notation.
The first result on convergence  properties of the DGA($t,\cC$) was obtained in \cite{VT111} (also, see \cite{VTbook}, pp.368-369).  

\begin{Theorem}[{\cite{VT111}}]\label{moT1} Let $X$ be a uniformly smooth Banach space with the modulus of smoothness $\rho(u)$. Assume $\cC=\{c_j\}_{j=1}^\infty$ is such that $c_j\ge 0$, $j=1,2,\dots$,
$$
\sum_{j=1}^\infty c_j =\infty,
$$
and for any $y>0$
\be\label{7.14}
\sum_{j=1}^\infty \rho(yc_j) <\infty.  
\ee
Then, for the  DGA($t,\cC$) we have
\be\label{7.15}
\liminf_{m\to \infty}\|f_m\| =0. 
\ee
\end{Theorem}

  In  \cite{VT116} we proved an analogue of Theorem \ref{moT1} for the XGA($\cC$)
 and improved upon the convergence in Theorem \ref{moT1} in the case of uniformly smooth Banach spaces with power-type modulus of smoothness.  Under an extra assumption on $\cC$ we replaced $\liminf$ by $\lim$. Here is the corresponding result from \cite{VT116}.
\begin{Theorem}[{\cite{VT116}}]\label{moT2} Let $\cC\in\ell_q\setminus\ell_1$ be a monotone sequence. Then both DGA($t,\cC$) and  XGA($\cC$) converge in any uniformly smooth Banach space $X$ with modulus of smoothness $\rho(u)\le \gamma u^q$, $q\in(1,2]$.
\end{Theorem}

In  \cite{VT116} we also addressed a question of the rate of convergence of these algorithms.   We proved the following theorem.
\begin{Theorem}[{\cite{VT116}}]\label{moT3} Let $X$ be a uniformly smooth Banach space with modulus of smoothness $\rho(u)\le \gamma u^q$, $q\in(1,2]$. 
We set $s:=(1+1/q)/2$ and $\cC_s:=\{k^{-s}\}_{k=1}^\infty$. Then both 
DGA($t,\cC_s$) and XGA($\cC_s$) (for this algorithm $t=1$) converge 
for $f\in A_1(\cD)$ with the following rate: For any $r\in(0,t(1-s))$
$$
\|f_m\|\le C(r,t,q,\gamma)m^{-r}.
$$
\end{Theorem}
In the case $t=1$, Theorem \ref{moT3} provides the rate of convergence $m^{-r}$   with $r$ arbitrarily close to $(1-1/q)/2$.  It would be interesting to know if the rate $m^{-(1-1/q)/2}$ is the best that can be achieved by algorithms DGA($t,\cC_s$) and XGA($\cC_s$) or, more generally, by algorithms DGA($t,\cC$) and XGA($\cC$). 

\begin{Remark}\label{moR1} The above results can be interpreted as a generalization of the following classical Riemann's theorem on convergence of sequences: Assume $\cC=\{c_j\}_{j=1}^\infty$ is such that $c_j\ge 0$, $j=1,2,\dots$, and 
\be\label{inf}
\sum_{j=1}^\infty c_j =\infty,\quad \lim_{j\to\infty} c_j =0.
\ee
Then for any real number $a$ we can find the signs $\epsilon_j =1$ or $=-1$ such that  $\sum_{j=1}^\infty \epsilon_j c_j = a$. In the Riemann's theorem the role of $X$ is played by $\bbR$ and the dictionary $\cD$ consists of one element $1$. 
Theorem \ref{moT2} generalizes the Riemann's theorem in the following way. Under an additional to (\ref{inf}) assumptions that $\cC$ is monotone and $\cC \in \ell_q$,  $q\in(1,2]$, it guaranties that for any $f\in X$ and any dictionary $\cD$ we can find the  signs $\epsilon_j$ and elements $\ff_j \in\cD$ such that $f = \sum_{j=1}^\infty \epsilon_j \ff_j c_j$. 
\end{Remark}

\subsection{Modification of WDGA}
   
 {\bf Dual Greedy Algorithm with parameters $(\tau,b,\mu)$ (DGA$(\tau,b,\mu)$).}
Let $X$ be a uniformly smooth Banach space with modulus of smoothness $\rho(u)$ and let $\mu(u)$ be a continuous majorant of $\rho(u)$: $\rho(u)\le\mu(u)$, $u\in[0,\infty)$. For a weakness sequence $\tau$ and a parameter $b\in (0,1]$ we define sequences
$\{f_m\}_{m=0}^\infty$, $\{\ff_m\}_{m=1}^\infty$, $\{c_m\}_{m=1}^\infty$ inductively. Let $f_0:=f$. If   $f_{m-1}=0$ for some $m\ge 1$, then we set $f_j=0$ for $j\ge m$ and stop. If $f_{m-1}\neq 0$ then we conduct the following three steps.

(1) Take any $\ff_m \in \cD^{\pm}$ such that
\be\label{7.21}
F_{f_{m-1}}(\ff_m) \ge t_m\|F_{f_{m-1}}\|_\cD.  
\ee

(2) Choose $c_m>0$ from the equation (we assume that $\mu$ is such that a solution exists)
\be\label{7.22}
\|f_{m-1}\|\mu(c_m/\|f_{m-1}\|) = \frac{t_mb}{2}c_m\|F_{f_{m-1}}\|_\cD.  
\ee

(3) Define
\be\label{7.23}
f_m:=f_{m-1}-c_m\ff_m.  
\ee
   
  For the reader's convenience let us figure out how the DGA$(1,b,u^2/2)$ works in Hilbert space. Consider its $m$th iteration. Let $\ff_m\in\cD^{\pm}$ be from (\ref{7.21}) with $t_m=1$. Then it is clear that $\ff_m$ 
maximizes the $\<f_{m-1},g\>$ over the dictionary $\cD^{\pm}$ and
$$
\<f_{m-1},\ff_m\> = \|f_{m-1}\|\|F_{f_{m-1}}\|_\cD.
$$
The PGA would use $\ff_m$ with the coefficient $\<f_{m-1},\ff_m\>$ at this step. 
The DGA$(1,b,u^2/2)$ uses the same $\ff_m$ and only a fraction of $\<f_{m-1},\ff_m\>$:
\be\label{7.43}
c_m = b\|f_{m-1}\|\|F_{f_{m-1}}\|_\cD.  
\ee
Thus, the choice $b=1$ in (\ref{7.43}) corresponds to the PGA. The choice $b\in (0,1)$ was introduced in 2003 in \cite{VT111} (see the preprint version of it from 2003) and studied there in the case of Hilbert space for $b\in (0,1]$ and in the case of Banach spaces for $b\in (0,1)$. We point out  that the technique developed in \cite{VT111} for the case of general Banach spaces, does not work for $b=1$. The idea of using a shorter step than the one provided by the PGA was called in the later follow up literature by "shrinkage".  

The following convergence and rate of convergence results are from \cite{VT111}.

 \begin{Theorem}[{\cite{VT111}}]\label{moT4} Let $X$ be a uniformly smooth Banach space with the modulus of smoothness $\rho(u)$ and let $\mu(u)$ be a continuous majorant of $\rho(u)$ with the property $\mu(u)/u \downarrow 0$ as $u\to +0$. Then, for any $t\in (0,1]$ and $b\in (0,1)$ the DGA$(t,b,\mu)$ converges in $X$.  
\end{Theorem}

The following result gives the rate of convergence in the case $\tau =\{t\}$.

\begin{Theorem}[{\cite{VT111}}]\label{moT5} Assume $X$ has a modulus of smoothness $\rho(u)\le \gamma u^q$, $q\in (1,2]$. Denote $\mu(u) := \gamma u^q$. Then, for any dictionary $\cD$ and any $f\in A_1(\cD)$, the rate of convergence of the DGA$(t,b,\mu)$ is given by 
$$
\|f_m\|\le C(t,b,\gamma,q)m^{-\frac{t(1-b)}{p(1+t(1-b))}}, \quad p:= \frac{q}{q-1}.
$$
\end{Theorem}

The following result gives the rate of convergence for general $\tau$.

\begin{Theorem}[{\cite{VT111}}]\label{moT6} Let $\tau :=\{t_k\}_{k=1}^\infty$ be a nonincreasing sequence $1\ge t_1\ge t_2 \dots >0$ and $b\in (0,1)$. Assume $X$ has a modulus of smoothness $\rho(u)\le \gamma u^q$, $q\in (1,2]$. Denote $\mu(u) := \gamma u^q$. Then, for any dictionary $\cD$ and any $f\in A_1(\cD)$, the rate of convergence of the DGA$(\tau,b,\mu)$ is given by 
$$
\|f_m\|\le C(b,\gamma,q)\left(1+\sum_{k=1}^mt_k^p\right)^{-\frac{t_m(1-b)}{p(1+t_m(1-b))}}, \quad p:= \frac{q}{q-1}.
$$
\end{Theorem}

\section{Algorithms designed for approximation of elements of $A_1(\cD)$}
\label{A1}

The discussed above algorithms XGA, WDGA, WCGA, GAWFR, and their modifications can be applied to any $f\in X$.
We formulated some convergence results for arbitrary elements of $X$. We also discussed the problem of the rate of convergence of these algorithms under an extra assumption that $f\in A_1(\cD)$. In this section we discuss algorithms, which are  designed for approximation of elements of $A_1(\cD)$. This means that in order to run these algorithms and in order to guarantee their convergence properties we need to know in advance that $f\in A_1(\cD)$. 

\subsection{Thresholding algorithms}

The above discussed algorithms may be classified by their greedy steps. The ones with the $X$-greedy step fall in the category of $X$-greedy algorithms and those, which use the weak dual greedy step fall in the category of dual greedy algorithms. It is clear that at each iteration of all the above algorithms the greedy step is the most difficult one to realize. 
Namely, calculation (or even estimating) the quantity $\sup_{\phi\in \cD}|F_{f_{m-1}}(\phi)|$ is indeed a very difficult task. 
In this subsection we will try to avoid calculating the quantity $\sup_{\phi\in \cD}|F_{f_{m-1}}(\phi)|$ at the greedy step. 
Formally, the weak dual greedy step is a thresholding step by its nature. However, as we have pointed out it uses the
quantity $\sup_{\phi\in \cD}|F_{f_{m-1}}(\phi)|$. In this subsection we discuss some dual greedy algorithms, whose greedy 
step at the $m$th iteration uses a threshold determined only by the following information from the previous iterations: $\|f_{m-1}\|_X$, the coefficients $c_1,\dots,c_{m-1}$ from $G_{m-1} = c_1\ff_1+\cdots+c_{m-1}\ff_{m-1}$, and the characteristics of the Banach space $X$. We also discuss the case of a fixed threshold $\delta>0$. We want to point out that the information $f\in A_1(\cD)$ allows us to substantially simplify   the greedy step of the algorithm. 

Let us begin with a known remark (see, for instance, \cite{VTbook}, p.346) on the WCGA. 

\begin{Remark}\label{A1R1} Theorem \ref{rcT1} holds for a slightly modified versions of the WCGA($\tau$) and the WGAFR($\tau$), for which at the greedy step we require
\be\label{2.9}
F_{f_{m-1}}(\varphi_m) \ge t_m\|f_{m-1}\|,\quad \varphi_m\in \cD^\pm.  
\ee
\end{Remark}
This statement follows from the fact that, in the proof of Theorem  \ref{rcT1}, the relation
$$
F_{f_{m-1}}(\varphi_m) \ge t_m\sup_{\phi\in \cD^\pm} F_{f_{m-1}}(\phi)
$$
was  only  used to get (\ref{2.9}).
  
 The following algorithm was introduced and studied in \cite{VT111} (also, see \cite{VTbook}, p.374).
 
  {\bf Modified Dual Greedy Algorithm $(\tau,b,\mu)$ (MDGA$(\tau,b,\mu)$).}
Let $X$ be a uniformly smooth Banach space with modulus of smoothness $\rho(u)$ and let $\mu(u)$ be a continuous majorant of $\rho(u)$: $\rho(u)\le\mu(u)$, $u\in[0,\infty)$. For a sequence $\tau =\{t_k\}_{k=1}^\infty$, $t_k \in (0,1]$ and a parameter $b\in (0,1)$, we define for $f\in A_1(\cD)$ the following sequences
$\{f_m\}_{m=0}^\infty$, $\{\ff_m\}_{m=1}^\infty$, $\{c_m\}_{m=1}^\infty$ inductively. Let $f_0:=f$. If $f_{m-1}=0$ for some $m\ge 1$, then we set $f_j=0$ for $j\ge m$ and stop. If $f_{m-1}\neq 0$ then we conduct the following three steps.

(1) Take any $\ff_m \in \cD^{\pm}$ such that
$$
F_{f_{m-1}}(\ff_m) \ge t_m\|f_{m-1}\|\left(1+\sum_{j=1}^{m-1}c_j\right)^{-1}. 
$$

(2) Choose $c_m>0$ from the equation
$$
 \mu(c_m/\|f_{m-1}\|) = \frac{t_mb}{2}c_m \left(1+\sum_{j=1}^{m-1}c_j\right)^{-1}. 
$$

(3) Define
$$
f_m:=f_{m-1}-c_m\ff_m. 
$$
 
We have the following rate of convergence result for this algorithm. 

\begin{Theorem}[{\cite{VT111}}]\label{A1T1} Let $\tau :=\{t_k\}_{k=1}^\infty$ be a nonincreasing sequence $1\ge t_1\ge t_2 \dots >0$ and $b\in (0,1)$. Assume that $X$ has a modulus of smoothness $\rho(u)\le \gamma u^q$, $q\in (1,2]$. Denote $\mu(u) := \gamma u^q$. Then, for any dictionary $\cD$ and any $f\in A_1(\cD)$, the rate of convergence of the MDGA$(\tau,b,\mu)$ is given by 
$$
\|f_m\|\le C(b,\gamma,q)\left(1+\sum_{k=1}^mt_k^p\right)^{-\frac{t_m(1-b)}{p(1+t_m(1-b))}}, \quad p:= \frac{q}{q-1}.
$$
\end{Theorem}

 Clearly, this modification is more ready for practical implementation than the DGA$(\tau,b,\mu)$. 
 
 We now consider two algorithms defined and studied in \cite{VT115} with a different type of thresholding. These algorithms work for any $f\in X$. We begin with the Dual Greedy Algorithm with Relaxation and Thresholding (DGART).

 {\bf DGART.} 
  We define $f_0   :=f$ and $G_0  := 0$. Then for  a given parameter $\de\in(0,1/2]$    we have the following inductive definition for $m\ge1$

(1) $\varphi_m   \in \cD^{\pm}$ is any element satisfying
\be\label{6.2}
F_{f_{m-1}}(\varphi_m  ) \ge  \de.  
\ee
If there is no $\varphi_m   \in \cD^{\pm}$ satisfying (\ref{6.2}) then we stop.

(2) Find $w_m$ and $ \lambda_m $ such that
$$
\|f-((1-w_m)G_{m-1} + \la_m\varphi_m)\| = \inf_{ \la,w}\|f-((1-w)G_{m-1} + \la\varphi_m)\|
$$
and define
$$
G_m:=   (1-w_m)G_{m-1} + \la_m\varphi_m.
$$

(3) Let
$$
f_m   := f-G_m.
$$
If $\|f_m\|\le \de\|f\|$ then we stop, otherwise we proceed to the $(m+1)$th iteration.

The following algorithm is a thresholding-type modification of the WCGA. This modification can be applied to any $f\in X$ as well.  

 {\bf Chebyshev Greedy Algorithm with Thresholding $\de$  (CGAT($\de$)).}  For a given parameter $\de\in(0,1/2]$, we conduct instead of the greedy step  of the WCGA the following thresholding step: Find $\ff_m\in\cD^{\pm}$ such that $F_{f_{m-1}}(\ff_m)\ge \de$. 
Choosing such a $\ff_m$, if one exists, we   apply the approximation step  of the WCGA.
If such $\ff_m$ does not exist, then we stop. We also stop if $\|f_m\|\le \de\|f\|$.

\begin{Theorem}[{\cite{VT115}}] \label{A1T2}Let $X$ be a uniformly smooth Banach space with modulus of smoothness $\rho(u)\le \gamma u^q$, $1<q\le 2$. Then, for any $f\in  A_1(\cD)$ the DGART (CGAT($\de$))  will stop after $m \le C(\gamma) \de^{-p}\ln(1/\de)$, $p:=q/(q-1)$, iterations with
$$
\|f_m\|\le \de.
$$
\end{Theorem}

 \subsection{Algorithms providing approximants from $A_1(\cD)$}
 
 Sometimes, it is convenient to have an approximant of a special form. Clearly, all our greedy algorithms provide after $m$ iterations approximants, which are sparse $m$-term approximants. Here we discuss some additional properties in a style that the approximant resembles the function under approximation. Namely, we begin with the requirement that the approximant comes from  $A_1(\cD)$ provided that $f\in A_1(\cD)$. We proceed to  one more thresholding-type algorithm (see \cite{VT94}).  Keeping in mind possible applications of this algorithm we  consider instead of $A_1(\cD)$  the closure of the convex hull of $\cD$, which we denote conv($\cD$) or  $A_1(\cD^+)$. Clearly, $A_1(\cD) = \conv(\cD^\pm)$, where $\cD^\pm := \{\pm g\,:\, g\in \cD\}$ is the symmetrized version of dictionary $\cD$. 
 Let $\e=\{\e_n\}_{n=1}^\infty $, $\e_n> 0$, $n=1,2,\dots$ . 

 {\bf Incremental Algorithm with schedule $\e$ (IA($\e$)).} 
Let $f\in \conv(\cD)$. Denote $f_0^{i,\e}:= f$ and $G_0^{i,\e} :=0$. Then, for each $m\ge 1$ we have the following inductive definition.

(1) $\ff_m^{i,\e} \in \cD$ is any element satisfying
$$
F_{f_{m-1}^{i,\e}}(\ff_m^{i,\e}-f) \ge -\e_m.
$$

(2) Define
$$
G_m^{i,\e}:= (1-1/m)G_{m-1}^{i,\e} +\ff_m^{i,\e}/m.
$$

(3) Let
$$
f_m^{i,\e} := f- G_m^{i,\e}.
$$

We note that it is known that  for any bounded linear functional $F$ and any $\cD$
$$
\sup_{g\in\cD}F(g) = \sup_{f\in  \conv(\cD)} F(f). 
$$
Therefore, for any $F$ and any $f\in  \conv(\cD)$ 
$$
\sup_{g\in\cD}F(g) \ge F(f).
$$
This guarantees existence of $\ff_m^{i,\e}$.

\begin{Theorem}[{\cite{VT94}}]\label{A1T3} Let $X$ be a uniformly smooth Banach space with  modulus of smoothness $\rho(u)\le \gamma u^q$, $1<q\le 2$. Define
$$
\e_n := K_1\gamma ^{1/q}n^{-1/p},\qquad p=\frac{q}{q-1},\quad n=1,2,\dots .
$$
Then, for any $f\in \conv(\cD)$ we have
$$
\|f_m^{i,\e}\| \le C(K_1) \gamma^{1/q}m^{-1/p},\qquad m=1,2\dots.
$$
\end{Theorem}

We now present a relaxed type of algorithm from \cite{VT80} (see also \cite{VTbook}, p.347).

 {\bf Weak Relaxed Greedy Algorithm (WRGA($\tau$)).} 
We define $f_0 :=  f$ and $G_0  := 0$. Then, for each $m\ge 1$ we have the following inductive definition.

(1) $\varphi_m \in \cD^\pm$ is any element satisfying
$$
F_{f_{m-1}}(\varphi_m - G_{m-1}) \ge t_m \sup_{g\in \cD^\pm} F_{f_{m-1}}(g - G_{m-1}).
$$

(2) Find $0\le \lambda_m \le 1$ such that
$$
\|f-((1-\la_m)G_{m-1} + \la_m\varphi_m)\| = \inf_{0\le \la\le 1}\|f-((1-\la)G_{m-1} + \la\varphi_m)\|
$$
and define
$$
G_m:=   (1-\la_m)G_{m-1} + \la_m\varphi_m.
$$

(3) Let
$$
f_m   := f-G_m.
$$

\begin{Theorem}[{\cite{VT80}}]\label{A1T4} Let $X$ be a uniformly smooth Banach space with modulus of smoothness $\rho(u) \le \gamma u^q$, $1<q\le 2$. Then, for a weakness sequence $\tau$ we have for any $f\in A_1(\cD$ that 
$$
\|f_m\| \le C_1(q,\gamma)\left(1+\sum_{k=1}^m t_k^p\right)^{-1/p},\quad p:= \frac{q}{q-1},
$$
with a constant $C_1(q,\gamma)$ which may depend only on $q$ and $\gamma$.
\end{Theorem}
 
 \section{Unified way of analyzing some greedy algorithms} 
 \label{U}
 
We have discussed above some important greedy-type algorithms, which are useful from applied point of view and interesting from theoretical point of view. 
One of the main goals of this section is to present a unified way of analyzing different
greedy-type algorithms in Banach spaces. This section is based on the paper \cite{VT165}. In that paper 
we defined a class of Weak Biorthogonal Greedy Algorithms
and proved convergence and rate of convergence results for algorithms from this class. In particular, the following well known algorithms~--- Weak Chebyshev Greedy Algorithm and Weak Greedy Algorithm with Free Relaxation (which we have discussed above)~--- 
belong to this class. We also introduced and studied in \cite{VT165} one more algorithm~--- Rescaled Weak Relaxed Greedy Algorithm~--- from the above class. 

 {\bf Weak Biorthogonal Greedy Algorithms (WBGA($\tau$)).} 
 Let $\tau:=\{t_m\}_{m=1}^\infty$, $t_m\in[0,1]$, be a weakness  sequence. We define $f_0   :=f$ and $G_0  := 0$. Suppose that for each $m\ge 1$ the algorithm has the following properties.

{\bf (1) Greedy selection.} At the $m$th iteration the algorithm selects a $\varphi_m   \in \cD^\pm$, which satisfies 
$$
F_{f_{m-1}}(\varphi_m) \ge t_m  \|F_{f_{m-1}}\|_\cD.
$$

{\bf (2) Biorthogonality.} At the $m$th iteration the algorithm constructs an approximant 
$G_m \in \sp(\ff_1,\dots,\ff_m)$ such that
$$
F_{f_m}(G_m) =0,\quad \text{where}\quad f_m:= f-G_m.
$$

{\bf (3) Error reduction.} We have
$$
\|f_m\| \le \inf_{\la\ge 0}\|f_{m-1} -\la \ff_m\|.
$$

Note that the step (1) of $m$th iteration is very similar to the weak dual greedy step with parameter $t_m$.
We stress that the WBGA is not an algorithm -- it is a class (collection) of algorithms, which have properties (1)--(3) listed above. 

\begin{Remark}\label{UR1} The greedy selection of $\ff_m$ at step (1) implies that 
$$
\inf_{\la\ge 0}\|f_{m-1} -\la \ff_m\| = \inf_{\la}\|f_{m-1} -\la \ff_m\|.
$$
\end{Remark}
Indeed, for $\la<0$ we have
$$
\|f_{m-1} -\la \ff_m\| \ge F_{f_{m-1}}(f_{m-1} -\la \ff_m) \ge \|f_{m-1}\|.
$$

We begin with a convergence result.

\begin{Theorem}[{\cite{VT165}}]\label{UT1} Let $X$ be a uniformly smooth Banach space with modulus of smoothness $\rho(u)$. Assume that a weakness sequence $\tau :=\{t_k\}_{k=1}^\infty$ satisfies the condition: For any $\theta >0$ we have
$$
\sum_{m=1}^\infty t_m \xi_m(\rho,\tau,\theta) =\infty.
$$
Suppose that an algorithm $\cA$ belongs to the class WBGA.  
 Then, for any $f\in X$ and any $\cD$ we have 
$$
\lim_{m\to \infty} \|f_m\| =0.
$$
\end{Theorem}

Here is a direct corollary of the above theorem. 

\begin{Theorem}[{\cite{VT165}}]\label{UT2} Let a Banach space $X$ have modulus of smoothness $\rho(u)$ of power type $1<q\le 2$, that is, $\rho(u) \le \gamma u^q$. Assume that the weakness sequence $\tau$ satisfies the condition
\be\label{1.2c}
\sum_{m=1}^\infty t_m^p =\infty, \quad p:=\frac{q}{q-1}. 
\ee
Then, the WBGA($\tau$)-type algorithm converges in $ X$.
\end{Theorem}

We now present a rate of convergence result. 

\begin{Theorem}[{\cite{VT165}}]\label{UT3} Let $X$ be a uniformly smooth Banach space with modulus of smoothness $\rho(u)\le \gamma u^q$, $1<q\le 2$.  
Then, for the WBGA($\tau$)-type algorithm we have for $f\in A_1(\cD)$
$$
\|f_m\| \le    C(q,\gamma)  \left(1+\sum_{k=1}^mt_k^p\right)^{-1/p},
\quad p:=\frac{q}{q-1},
$$
with $C(q,\gamma)= 4(2\gamma)^{1/q}$.
\end{Theorem}

Previous results (see Theorems \ref{rcT1} and \ref{rcT2}) show that we can simplify the WCGA to the WGAFR, which involves two parameters optimization at each iteration, and keep the result of the rate of convergence in the whole generality. We can further simplify the WCGA to the GAWR($\tau,\br$) with $\br:=\{2/(k+2)\}_{k=1}^\infty$, which involves one parameter optimization at each iteration, but in this case we only know the rate of convergence result for special weakness sequence $\tau = \{t\}$, $t\in (0,1]$. One of the goals of this section is to present 
a new greedy-type algorithm~--- the Rescaled Weak Relaxed Greedy Algorithm~--- that does two one parameter optimizations at each iteration and provides the rate of convergence result in the whole generality. One can view this new algorithm as a version of the WGAFR. 

{\bf Rescaled Weak Relaxed Greedy Algorithm (RWRGA($\tau$)).}
Let $\tau:=\{t_m\}_{m=1}^\infty$, $t_m\in[0,1]$, be a weakness  sequence. We define $f_0   :=f$ and $G_0  := 0$. Then, for each $m\ge 1$ we inductively define

1). $\varphi_m   \in \cD^\pm$ is any satisfying
$$
F_{f_{m-1}}(\varphi_m) \ge t_m \|F_{f_{m-1}}\|_\cD.
$$

2). Find $ \lambda_m \ge 0$ such that
$$
\|f_{m-1} - \la_m\varphi_m\| = \inf_{ \la\ge 0}\|f_{m-1} - \la\varphi_m\|.
$$

3). Find $ \mu_m$ such that
$$
\|f - \mu_m(G_{m-1} + \la_m\varphi_m)\| = \inf_{ \mu}\|f - \mu(G_{m-1} + \la_m\varphi_m)\|
$$
and define
$$
G_m:=   \mu_m(G_{m-1} + \la_m\varphi_m).
$$

4). Denote
$$
f_m   := f-G_m.
$$

We note that the above algorithm does not use any information of the Banach space $X$. It makes this algorithm universal alike the WCGA and the WGAFR.
A variant of the above algorithm was studied in~\cite{GG}. In~\cite{GG} the $\la_m$ was not chosen as a solution of the line search, it was specified in terms of 
$f_{m-1}$ and parameters $\gamma$ and $q$ characterizing smoothness of the 
Banach space $X$  (see further discussion at the end of Section 8 of \cite{VT165}). 

\begin{Theorem}[{\cite{VT165}}]\label{UT4} Let $X$ be a uniformly smooth Banach space with modulus of smoothness $\rho(u)\le \gamma u^q$, $1<q\le 2$.  
Then, for the RWRGA($\tau$) we have for any $f\in A_1(\cD)$
$$
\|f_m\| \le  C(q,\gamma)\left(1+\sum_{k=1}^mt_k^p\right)^{-1/p},
\quad p:=\frac{q}{q-1}.
$$
\end{Theorem}

Theorems~\ref{rcT1} and~\ref{UT4}  show that three algorithms~---
the WCGA($\tau$), the WGAFR($\tau$), and the RWRGA($\tau$)~--- provide the same upper bounds for approximation in smooth Banach spaces.  

Algorithms from the class WBGA, in particular, the WCGA, the WGAFR, and the RWRGA, have a greedy step (greedy selection), which is based on the norming functional $F_{f_{m-1}}$. As we mentioned above, for this reason, this type of algorithms is called dual greedy algorithms. 
There exists other natural class of greedy algorithms, which we discussed above -- the $X$-greedy type algorithms. These algorithms do not use the norming functional and, therefore, might be more suitable for implementation than the dual greedy algorithms.  
There are $X$-greedy companions for the dual greedy algorithms: 
the WCGA, the WGAFR, and the RWRGA. We build an $X$-greedy companion of a dual greedy algorithm by replacing 
the greedy step~--- the first step of the $m$th iteration~--- by minimizing over all possible choices of a newly added element from the dictionary. We follow this principle in the construction of the $X$-greedy companion of the RWRGA. 
 
 {\bf Rescaled Relaxed $X$-Greedy Algorithm (RRXGA).}
 We define $f_0   :=f$ and $G_0  := 0$. Then for each $m\ge 1$ we inductively define

1). Find $ \lambda_m \ge 0$ and $\ff_m\in \cD^\pm$ (we assume existence) such that
$$
\|f_{m-1}- \la_m\varphi_m\| = \inf_{ \la\ge 0; g\in \cD^\pm}\|f_{m-1}-  \la g\|.
$$

2). Find $ \mu_m$ such that
$$
\|f-\mu_m(G_{m-1} + \la_m\varphi_m)\| = \inf_{ \mu}\|f-\mu(G_{m-1} + \la_m\varphi_m)\|
$$
and define
$$
G_m:=   \mu_m(G_{m-1} + \la_m\varphi_m).
$$

3). Denote
$$
f_m   := f-G_m.
$$

It is known (see~\cite{VTbook}) that results on the $X$-greedy companions of dual relaxed greedy algorithms can be derived from the proofs of the corresponding results for the dual greedy algorithms. 
This was also illustrated on the RWRGA and its companion RRXGA in \cite{VT165} in the proof of the following theorem.

\begin{Theorem}[{\cite{VT165}}]\label{UT5} Let $X$ be a uniformly smooth Banach space with modulus of smoothness $\rho(u)\le \gamma u^q$, $1<q\le 2$. Then, for the RRXGA we have for any $f\in A_1(\cD)$
$$
\|f_m\| \le  C(q,\gamma) (1+m)^{-1/p},\quad p:=q/(q-1).
$$
\end{Theorem}

\section{Stability}
\label{S}

In the previous sections we discussed two important properties of an algorithm -- convergence and rate of convergence. 
In this section we present some results on one more important property of an algorithm -- stability. Stability can be understood in different ways. Clearly, stability means that small perturbations do not result in a large change in the outcome of the algorithm. In this section we discuss two kinds of perturbations -- noisy data and errors made in the process of realization of the algorithm. 

\subsection{Noisy data} 

Usually, in the greedy algorithms literature the noisy data is understood in the following deterministic way. Take a number $\e\ge 0$ and two elements $f$, $f^\e$ from $X$ such that
\be\label{S1}
\|f-f^\e\| \le \e,\quad
f^\e/A(\e) \in A_1(\cD),
\ee
with some number $A(\e)>0$.
Then we interpret $f$ as a noisy version of our original signal $f^\e$, for which we know that it has some good properties formulated in terms of $A_1(\cD)$. The first results on approximation of noisy data (in the sense of (\ref{S1})) were obtained in \cite{VT115} for the WCGA and WGAFR. Later, in \cite{VT165} we proved Theorem \ref{ST1} (see below), which 
covers all algorithms from the collection WBGA($\tau$). Thus, Theorem \ref{ST1} covers the known results from  \cite{VT115} on WCGA and WGAFR and also covers the corresponding result on the RWRGA.

\begin{Theorem}[{\cite{VT165}}]\label{ST1} Let $X$ be a uniformly smooth Banach space with modulus of smoothness $\rho(u)\le \gamma u^q$, $1<q\le 2$. Assume that $f$ and $f^\e$ satisfy (\ref{S1}).
Then, for any algorithm from the collection  WBGA($\tau$) applied to $f$ we have
$$
\|f_m\| \le  \max\left\{2\e,\, C(q,\gamma)(A(\e)+\e) \Big(1+\sum_{k=1}^mt_k^p\Big)^{-1/p}\right\},
\quad p:=\frac{q}{q-1},
$$
with $C(q,\gamma)= 4(2\gamma)^{1/q}$.
\end{Theorem}

\subsection{Approximate greedy algorithms}

The first results on the approximate versions of the algorithms WCGA and WRGA were obtained in \cite{VT94}. Further results were obtained in \cite{De}, \cite{DeDis}, and \cite{VT165}. In this subsection we formulate some results from the paper \cite{VT165}. 
We begin with the definitions of the approximate versions AWCGA,  AWGAFR, and  ARWRGA  of the previously mentioned algorithms (namely, the WCGA, the WGAFR, and the RWRGA). Then we give the definition of the class of Approximate Weak Biorthogonal Greedy Algorithms (AWBGA).  It was shown in \cite{VT165} that the algorithms AWCGA,  AWGAFR, and  ARWRGA  belong to the class of the Approximate Weak Biorthogonal Greedy Algorithms.  

{\bf Approximate norming functional.} For each algorithm let $\{F_m\}_{m=0}^\infty$ denote a sequence of functionals such that for any $m \geq 0$
\be\label{S2}
\|F_m\| \leq 1
\ \text{ and }\ 
F_m(f_m) \geq (1-\delta_m)\|f_m\|,
\ee
where $\{f_m\}_{m=0}^\infty$ is the sequence of remainders (residuals) produced by the corresponding algorithm.

We start with an approximate version of the WCGA, which was studied in~\cite{VT94} and~\cite{De}.

{\bf Approximate Weak Chebyshev Greedy Algorithm (AWCGA($\tau$))}
We denote $f^c_0 := f^{c,\tau}_0 :=f$. Then for each $m\ge 1$ we inductively define\\
1). $\varphi^{c}_m :=\varphi^{c,\tau}_m \in \cD^\pm$ is any satisfying
\[
F_{m-1}(\varphi^{c}_m) \geq t_m \| F_{m-1}\|_\cD. 
\]
2). Define
\[
\Phi^c_m := \Phi^{c,\tau}_m := \sp \{\varphi^{c}_j\}_{j=1}^m,
\]
and define $G_m^c := G_m^{c,\tau}$ to be any such element from $\Phi^c_m
$ that
\[
\|f - G_m^c\| \leq (1+\eta_m) \inf_{G \in \Phi_m^c} \|f - G\|.
\]
3). Denote
\[
f^{c}_m := f^{c,\tau}_m := f-G^c_m.
\]

\noindent
Next, we present an approximate version of the WGAFR, which was studied in \cite{VT165}.

{\bf Approximate Weak Greedy Algorithm with Free Relaxation (AWGAFR($\tau$))}
We denote $f^{fr}_0 := f^{fr,\tau}_0 :=f$. Then for each $m\ge 1$ we inductively define\\
1). $\varphi^{fr}_m :=\varphi^{fr,\tau}_m \in \cD^\pm$ is any satisfying
\[
F_{m-1}(\varphi^{fr}_m) \geq t_m \| F_{m-1}\|_\cD. 
\]
2). Find $w_m$ and $\lambda_m \ge 0$ such that 
\[
\|f - ((1-w_m)G^{fr}_{m-1} + \lambda_m \varphi^{fr}_m)\| 
\le (1+\eta_m) \inf_{\lambda\ge 0, w}\|f - ((1-w)G^{fr}_{m-1} + \lambda\varphi^{fr}_m)\|
\]
and define $G_m^{fr} := G_m^{fr,\tau} := (1-w_m)G^{fr}_{m-1} + \lambda_m \varphi^{fr}_m$.\\
3). Denote
\[
f^{fr}_m := f^{fr,\tau}_m := f-G^{fr}_m.
\]

\noindent
Lastly, we introduce an approximate version of the RWRGA, which was studied in \cite{VT165}.

{\bf Approximate Rescaled Weak Relaxed Greedy Algorithm \newline(ARWRGA($\tau$))}
We denote $f^r_0 := f^{r,\tau}_0 :=f$. Then for each $m\ge 1$ we inductively define\\
1). $\varphi^{r}_m :=\varphi^{r,\tau}_m \in \cD^\pm$ is any satisfying
\[
F_{m-1}(\varphi^{r}_m) \geq t_m \| F_{m-1}\|_\cD. 
\]
2). Find $\lambda_m \ge 0$ such that 
\[
\|f - (G^r_{m-1} + \lambda_m \varphi^r_m)\| 
\le (1+\eta_m) \inf_{\lambda\ge 0}\|f - (G^r_{m-1} + \lambda\varphi^r_m)\|.
\]
3). Find $\mu_m \ge 0$ such that 
\[
\|f - \mu_m(G^r_{m-1} + \lambda_m \varphi^r_m)\| 
\le (1+\eta_m) \inf_{\mu\ge 0}\|f - \mu(G^r_{m-1} + \lambda_m\varphi^r_m)\|
\]
and define $G_m^r := G_m^{r,\tau} := \mu_m(G^r_{m-1} + \lambda_m \varphi^r_m)$.\\
4). Denote
\[
f^{r}_m := f^{r,\tau}_m := f-G^r_m.
\]

Finally, we define the class of Approximate Weak Biorthogonal Greedy Algorithms.

{\bf Approximate Weak Biorthogonal Greedy Algorithms (AWBGA).} 
We define $f_0 := f$ and $G_0 := 0$. 
Suppose that for each $m\ge 1$ the algorithm has the following properties.
\\\indent
{\bf (1) Greedy selection.} 
At the $m$th iteration the algorithm selects a $\varphi_m \in \cD^\pm$, which satisfies 
\[
F_{m-1}(\varphi_m) \ge t_m \|F_{m-1}\|_\cD.
\]
\\\indent
{\bf (2) Biorthogonality.} 
At the $m$th iteration the algorithm constructs an approximant $G_m \in \sp(\ff_1,\dots,\ff_m)$ such that
\[
|F_m(G_m)| \leq \e_m.
\]
\indent
{\bf (3) Error reduction.} 
We have
\[
\|f_m\| \leq (1+\eta_m) \inf_{\la\ge 0}\|f_{m-1} - \la \ff_m\|.
\]

\begin{Proposition}[{\cite{VT165}}]\label{SP1}
The AWCGA($\tau$), the AWGAFR($\tau$), and \newline the ARWRGA($\tau$) belong to the class   AWBGA($\tau$) with 
\[
\e_m = \inf_{\lambda>0} \frac{1}{\lambda} (\delta_m + \eta_m + 2\rho(\lambda\|G_m\|)).
\]
\end{Proposition}

We now formulate the convergence and the rate of convergence results for the class AWBGA($\tau$). The corresponding 
corollaries for the algorithms AWCGA($\tau$), AWGAFR($\tau$), and  ARWRGA($\tau$) the reader can find in \cite{VT165}.

\begin{Theorem}[{\cite{VT165}}]\label{ST2}
Let $X$ be a uniformly smooth Banach space with modulus of smoothness $\rho(u) \leq \gamma u^q$, $1 < q \le 2$. 
Let sequences $\{t_m\}_{m=1}^\infty$, $\{\delta_m\}_{m=0}^\infty$, $\{\eta_m\}_{m=1}^\infty$, $\{\e_m\}_{m=1}^\infty$ be such that
\be\label{S3}
\sum_{k=1}^\infty t_k^p = \infty,
\ee
\be\label{S4}
\delta_{m-1} + \eta_m = o(t_m^p),
\quad
\e_m = o(1).
\ee
Then each algorithm from AWBGA($\tau$) converges in $ X$.
\end{Theorem}

\begin{Theorem}[{\cite{VT165}}]\label{ST3}
Let $X$ be a uniformly smooth Banach space with modulus of smoothness $\rho(u)\le \gamma u^q$, $1<q\le 2$. 
Take a number $\e\ge 0$ and two elements $f$, $f^\e$ from $X$ such that
\[
\|f-f^\e\| \le \e,\quad
f^\e/A(\e) \in A_1(\cD),
\]
with some number $A(\e) > 0$.
Then an algorithm from AWBGA($\tau$) with error parameters $\{t_m\}_{m=1}^\infty$, $\{\delta_m\}_{m=0}^\infty$, $\{\eta_m\}_{m=1}^\infty$, $\{\e_m\}_{m=1}^\infty$, satisfying 
\begin{align}
\label{S5}
&\delta_m + \e_m/\|f_m\| \leq 1/4,
\\
\label{S6}
&\delta_m + \eta_{m+1} \leq \frac{1}{2} C(q,\gamma)^{-p} A(\e)^{-p} t_{m+1}^p \|f_m\|^p
\end{align}
for any $m\ge 0$, provides
\[
\|f_m\| \le \max\left\{4\e,\, C(q,\gamma) (A(\e)+\e) \Big(1 + \sum_{k=1}^m t_k^p\Big)^{-1/p}\right\},
\]
where $C(q,\gamma) = 4q (2\gamma)^q \Big(\frac{2}{q-1}\Big)^{1/p}$ and $p = q/(q-1)$.
\end{Theorem}

\begin{Corollary}[{\cite{VT165}}]\label{SC1}
Let $X$ be a uniformly smooth Banach space with modulus of smoothness $\rho(u) \le \gamma u^q$, $1<q\le 2$. 
Then for any $f \in A_1(\cD)$ an algorithm from  AWBGA($\tau$) with error parameters $\{t_m\}_{m=1}^\infty$, $\{\delta_m\}_{m=0}^\infty$, $\{\eta_m\}_{m=1}^\infty$, $\{\e_m\}_{m=1}^\infty$, satisfying 
\begin{align*}
&\delta_m + \e_m/\|f_m\| \leq 1/4,
\\
&\delta_m + \eta_{m+1} \leq \frac{1}{2} C(q,\gamma)^{-p} t_{m+1}^p \|f_m\|^p
\end{align*}
for any $m\ge 0$, provides
\[
\|f_m\| \le C(q,\gamma)\Big(1+\sum_{k=1}^m t_k^p\Big)^{-1/p},
\]
where $C(q,\gamma) = 4q(2\gamma)^q \Big(\frac{2}{q-1}\Big)^{1/p}$ and $p = q/(q-1)$.
\end{Corollary}

\newpage

 \medskip
 {\bf \Large Chapter II : Greedy approximation with respect to dictionaries with special properties.}
  \medskip

In Chapter I we discussed different greedy algorithms with respect to arbitrary dictionaries. We discussed their convergence and rate of convergence properties. Clearly, in the case of rate of convergence property the class 
of elements, for which we guarantee that rate of convergence, depends on the given dictionary. In our case it is 
the $A_1(\cD)$ class. It is a surprising fact that the researchers managed to obtain convergence and rate of convergence 
results, which hold for any dictionary and are determined by the properties of the Banach space $X$. Moreover, it turns out that the property of $X$, which plays the fundamental role in those results, is very simple -- it is the modulus of smoothness $\rho(u,X)$ of the Banach space $X$. Certainly, the results of Chapter I can be applied to any dictionary 
$\cD$. In this chapter we discuss the following general question. How specific properties of a given dictionary can improve 
the convergence and rate of convergence results for greedy approximation? The central topic here will be the concept of 
Lebesgue-type inequalities. Very briefly it means that we want to make the following step in evaluation of quality of an algorithm. In Chapter I the quality of a greedy algorithm was expressed in terms of the rate of convergence of this algorithm for the whole class $A_1(\cD)$ -- the guarantied rate for all elements of $A_1(\cD)$. Here, with the concept of 
Lebesgue-type inequalities, we would like to guarantee the optimal error of approximation (understood in different ways) 
for all individual elements. 

\section{Lebesgue-type inequalities}
\label{Leb}

\subsection{General setting} 

In this subsection we follow Section 8.7 of \cite{VTbookMA}. 
In a general setting we are working in a Banach space $X$ with a redundant system of elements $\cD$ (dictionary $\cD$).   
An element (function, signal) $f\in X$ is said to be $m$-sparse with respect to $\cD$ if
it has a representation $h=\sum_{i=1}^mx_ig_i$,   $g_i\in \cD$, $i=1,\dots,m$. The set of all $m$-sparse elements is denoted by $\Sigma_m(\cD)$. For a given element $f$ we introduce the error of best $m$-term approximation
$
\sigma_m(f,\cD) := \inf_{h\in\Sigma_m(\cD)} \|f-h\|.
$
We are interested in the following fundamental problem of sparse approximation. 

{\bf Problem.} How to design a practical algorithm that builds sparse approximations, which provide errors comparable with errors of the best $m$-term approximations? 

In a general setting we study an algorithm (approximation method) $\cA = \{A_m(\cdot,\cD)\}_{m=1}^\infty$ with respect to a given dictionary $\cD$. The sequence of mappings $A_m(\cdot,\cD)$ defined on $X$ satisfies the condition: for any $f\in X$, $A_m(f,\cD)\in \Sigma_m(\cD)$. In other words, $A_m$ provides an $m$-term approximant with respect to $\cD$. It is clear that for any $f\in X$ and any $m$ we have
$
\|f-A_m(f,\cD)\| \ge \sigma_m(f,\cD).
$
We are interested in such pairs $(\cD,\cA)$ for which the algorithm $\cA$ provides for any $f\in X$ approximation close to best $m$-term approximation.

\begin{Remark}\label{integer} In the formulations of results of this section we use for convenience the following agreement. 
For a positive number $S$ the expression "$S$ iterations" means "$\lceil S\rceil$ iterations". For example, $A_S(f,\cD):= A_{\lceil S\rceil}(f,\cD)$ and $f_S := f_{\lceil S\rceil}$.  
\end{Remark}

We introduce the corresponding definitions.
\begin{Definition}\label{LebD1} We say that $\cD$ is an almost greedy dictionary with respect to $\cA$ if there exist two constants $C_1$ and $C_2$ such that for any $f\in X$ we have
\begin{equation}\label{L1}
\|f-A_{C_1m}(f,\cD)\| \le C_2\sigma_m(f,\cD),\quad m=1,2,\dots.
\end{equation}
\end{Definition}
If $\cD$ is an almost greedy dictionary with respect to $\cA$ then $\cA$ provides almost ideal sparse approximation. It provides $C_1m$-term approximant as good (up to a constant $C_2$) as the ideal $m$-term approximant for every $f\in X$. In the case $C_1=1$ we call $\cD$ a greedy dictionary.
We also need a more general definition. Let $\phi(u)$ be a  function such that 
$\phi \, :\, \bbN \to \bbR_+$. 
\begin{Definition}\label{LebD2} We say that $\cD$ is a $\phi$-greedy dictionary with respect to $\cA$ if there exists a constant $C_3$ such that for any $f\in X$ we have
\begin{equation}\label{L2}
\|f-A_{\phi(m)m}(f,\cD)\| \le C_3\sigma_m(f,\cD), \quad m=1,2,\dots.
\end{equation}
\end{Definition}

 Inequalities of the form (\ref{L1}) and (\ref{L2}) are called the Lebesgue-type inequalities. If $\cD=\Psi$ is a basis then in the above definitions we replace dictionary by basis. 
 
 \begin{Remark}\label{LebR1} Assume that $\cD$ is a $\phi$-greedy dictionary with respect to $\cA$. Let $f\in X$ be an $m$-sparse with respect to $\cD$ element: $f\in \Sigma_m(\cD)$. Then, clearly $\sigma_m(f,\cD)=0$, and, therefore, by (\ref{L2}) the algorithm $\cA$ recovers $f$ exactly ($A_{\phi(m)m}(f,\cD)=f$) after $\phi(m)m$ iterations of the algorithm.
 \end{Remark}

\subsection{Some general results for the WCGA}

 A very important advantage of the WCGA  is its convergence and rate of convergence properties. The WCGA is well defined for all iterations $m$. Moreover, it is known (see Theorem \ref{conT1} above) that the WCGA with weakness parameter $t\in(0,1]$ converges for all $f$ in all uniformly smooth Banach spaces with respect to any dictionary.  We discuss here the Lebesgue-type inequalities for the WCGA($t$) with weakness parameter $t\in(0,1]$. This discussion is based on papers \cite{LivTem} and \cite{VT144} (see also \cite{VTbookMA}, Section 8.7.2). For notational convenience we consider here a countable dictionary $\cD=\{g_i\}_{i=1}^\infty$. The following assumptions {\bf A1} and {\bf A2} were used in \cite{LivTem}.
 For a given $f_0$ let sparse element (signal)
 $$
 f:=f^\e=\sum_{i\in T}x_ig_i,\quad g_i\in\cD,
 $$
 be such that $\|f_0-f^\e\|\le \e$ and $|T|=K$. For $A\subset T$ denote
 $$
 f_A:=f_A^\e := \sum_{i\in A}x_ig_i.
 $$
 
  {\bf A1.} We say that $f=\sum_{i\in T}x_ig_i$ satisfies the Nikol'skii-type $\ell_1X$ inequality with parameter $r$ if for any $A\subset T$
 \begin{equation}\label{zC1}
 \sum_{i\in A} |x_i| \le C_1|A|^{r}\|f_A\|. 
 \end{equation}
 We say that a dictionary $\cD$ has the Nikol'skii-type $\ell_1X$ property with parameters $K$, $r$   if any $K$-sparse element satisfies the Nikol'skii-type
 $\ell_1X$ inequality with parameter $r$.

{\bf A2.}  We say that $f=\sum_{i\in T}x_ig_i$ has incoherence property with parameters $D$ and $U$ if for any $A\subset T$ and any $\Lambda$ such that $A\cap \Lambda =\emptyset$, $|A|+|\Lambda| \le D$ we have for any $\{c_i\}$
\begin{equation}\label{zC2}
\|f_A-\sum_{i\in\Lambda}c_ig_i\|\ge U^{-1}\|f_A\|.
\end{equation}
We say that a dictionary $\cD$ is $(K,D)$-unconditional with a constant $U$ if for any $f=\sum_{i\in T}x_ig_i$ with
$|T|\le K$ inequality (\ref{zC2}) holds.

The term {\it unconditional} in {\bf A2} is justified by the following remark. The above definition of $(K,D)$-unconditional dictionary is equivalent to the following definition. Let $\cD$ be such that any subsystem of $D$ distinct elements $e_1,\dots,e_D$ from $\cD$ is linearly independent and for any $A\subset [1,D]$ with $|A|\le K$ and any coefficients $\{c_i\}$ we have
$$
\|\sum_{i\in A}c_ie_i\| \le U\|\sum_{i=1}^Dc_ie_i\|.
$$

It is convenient for us to use the following assumption {\bf A3} introduced in \cite{VT144}, which is a corollary of assumptions {\bf A1} and {\bf A2}. 

{\bf A3.} We say that $f=\sum_{i\in T}x_ig_i$ has $\ell_1$ incoherence property with parameters $D$, $V$, and $r$ if for any $A\subset T$ and any $\Lambda$ such that $A\cap \Lambda =\emptyset$, $|A|+|\Lambda| \le D$ we have for any $\{c_i\}_{i\in\Lambda}$
\begin{equation}\label{zC3}
\sum_{i\in A}|x_i| \le V|A|^r\|f_A-\sum_{i\in\Lambda}c_ig_i\|.
\end{equation}
A dictionary $\cD$ has $\ell_1$ incoherence property with parameters $K$, $D$, $V$, and $r$ if for any $A\subset B$, $|A|\le K$, $|B|\le D$ we have for any $\{c_i\}_{i\in B}$
$$
\sum_{i\in A} |c_i| \le V|A|^r\|\sum_{i\in B} c_ig_i\|.
$$

It is clear that {\bf A1} and {\bf A2} imply {\bf A3} with $V=C_1U$. Also, {\bf A3} implies {\bf A1} with $C_1=V$ and {\bf A2} with $U=VK^r$. Obviously, we can restrict ourselves to $r\le 1$.

We now proceed to the main results of \cite{LivTem} and \cite{VT144} on the WCGA with respect to redundant dictionaries. The following Theorem \ref{zT2.1} from \cite{VT144} in the case $q=2$ was proved in \cite{LivTem}. We use the standard notation $q':= q/(q-1)$. 

 \begin{Theorem}\label{zT2.1} Let $X$ be a Banach space with $\rho(u)\le \gamma u^q$, $1<q\le 2$. Suppose $K$-sparse $f^\e$ satisfies {\bf A1}, {\bf A2} and $\|f_0-f^\e\|\le \e$. Assume that $rq'\ge 1$. Then there exists a positive constant $C(t,\gamma,C_1)$ such that the WCGA($t$) with weakness parameter $t$ applied to $f_0$ provides after 
 $$S:= \lceil C(t,\gamma,C_1)U^{q'}\ln (U+1) K^{rq'}\rceil$$
  iterations 
$$
\|f_{S}\| \le C\e\quad\text{for}\quad K+S\le D
$$
with an absolute constant $C$.
\end{Theorem}

 It was pointed out in \cite{LivTem} that Theorem \ref{zT2.1} provides a corollary for Hilbert spaces that gives sufficient conditions somewhat weaker than the known RIP (see below) conditions on $\cD$ for the Lebesgue-type inequality to hold. We formulate the corresponding definitions and results. 
  Let $\cD$ be the Riesz dictionary with depth $D$ and parameter $\delta\in (0,1)$. This class of dictionaries is a generalization of the class of classical Riesz bases. We give a definition in a general Hilbert space (see \cite{VTbook}, p.306).
\begin{Definition}\label{zD3.1} A dictionary $\cD$ is called the Riesz dictionary with depth $D$ and parameter $\delta \in (0,1)$ if, for any $D$ distinct elements $e_1,\dots,e_D$ of the dictionary and any coefficients $a=(a_1,\dots,a_D)$, we have
\begin{equation}\label{z3.3}
(1-\de)\|a\|_2^2 \le \|\sum_{i=1}^D a_ie_i\|^2\le(1+\de)\|a\|_2^2.
\end{equation}
We denote the class of Riesz dictionaries with depth $D$ and parameter $\delta \in (0,1)$ by $R(D,\de)$.
\end{Definition}
The term Riesz dictionary with depth $D$ and parameter $\delta \in (0,1)$ is another name for a dictionary satisfying the Restricted Isometry Property (RIP) with parameters $D$ and $\de$. The following simple lemma holds.
\begin{Lemma}\label{zL3.1} Let $\cD\in R(D,\de)$ and let $e_j\in \cD$, $j=1,\dots, s$. For $f=\sum_{i=1}^s a_ie_i$ and $A \subset \{1,\dots,s\}$ denote
$$
S_A(f) := \sum_{i\in A} a_ie_i.
$$
If $s\le D$ then
$$
\|S_A(f)\|^2 \le (1+\de)(1-\de)^{-1} \|f\|^2.
$$
\end{Lemma}

Lemma \ref{zL3.1} implies that if $\cD\in R(D,\de)$ then it is $(D,D)$-unconditional with a constant $U=(1+\de)^{1/2}(1-\de)^{-1/2}$.

In the case of a Hilbert space the WCGA($\tau$) is called the Weak Orthogonal Greedy Algorithm (WOGA($\tau$)) and 
the Weak Orthogonal Matching Pursuit (WOMP($\tau$)). 
 \begin{Theorem}\label{zT2.2} Let $X$ be a Hilbert space. Suppose $K$-sparse $f^\e$ satisfies  {\bf A2} and $\|f_0-f^\e\|\le \e$. Then the WOGA($t$) with weakness parameter $t$ applied to $f_0$ provides
$$
\|f_{C(t,U) K}\| \le C\e\quad\text{for}\quad  K+\lceil  C(t,U) K\rceil \le D
$$
with an absolute constant $C$.
\end{Theorem}
Theorem \ref{zT2.2} implies the following corollaries.
 \begin{Corollary}\label{zC2.1} Let $X$ be a Hilbert space. Suppose any $K$-sparse $f$ satisfies   {\bf A2}.   Then the WOGA($t$) with weakness parameter $t$ applied to $f_0$ provides
$$
\|f_{C(t,U) K}\| \le C\sigma_K(f_0,\cD)\quad\text{for}\quad K+\lceil C(t,U) K\rceil \le D
$$
with an absolute constant $C$.
\end{Corollary}

 \begin{Corollary}\label{zC2.2} Let $X$ be a Hilbert space. Suppose $\cD\in R(D,\de)$.     Then the WOGA($t$) with weakness parameter $t$ applied to $f_0$ provides
$$
\|f_{C(t,\de) K}\| \le C\sigma_K(f_0,\cD)\quad\text{for}\quad K+\lceil  C(t,\de) K\rceil \le D
$$
with an absolute constant $C$.
\end{Corollary}

  We   emphasized in \cite{LivTem} that in Theorem \ref{zT2.1} (with $q=2$) we impose our conditions on an individual function $f^\e$. It may happen that the dictionary does not have the Nikol'skii  $\ell_1X$ property and $(K,D)$-unconditionality but the given $f_0$ can be approximated by $f^\e$ which does satisfy assumptions {\bf A1} and {\bf A2}. Even in the case of a Hilbert space the above results from \cite{LivTem} add something new to the study based on the RIP property of a dictionary. First of all, Theorem \ref{zT2.2} shows that it is sufficient to impose assumption {\bf A2} on $f^\e$ in order to obtain exact recovery and the Lebesgue-type inequality results. Second, Corollary \ref{zC2.1} shows that the condition {\bf A2}, which is weaker than the RIP condition, is sufficient for exact recovery and the Lebesgue-type inequality results. Third, Corollary \ref{zC2.2} shows that even if we impose our assumptions in terms of RIP we do not need to assume that $\de < \de_0$. In fact, the result works for all $\de<1$ with parameters depending on $\de$.
  
Theorem \ref{zT2.1} follows from the combination of Theorems \ref{zT2.3} and \ref{zT2.4}.
In case $q=2$ these theorems were proved in \cite{LivTem} and in general case $q\in(1,2]$ -- in \cite{VT144}. 

\begin{Theorem}\label{zT2.3} Let $X$ be a Banach space with $\rho(u)\le \gamma u^q$, $1<q\le 2$. Suppose for a given $f_0$ we have $\|f_0-f^\e\|\le \e$ with $K$-sparse $f:=f^\e$ satisfying {\bf A3}. Then for any $k\ge 0$ we have for $K+m \le D$, $m\ge k$
$$
\|f_m\| \le \|f_k\|\exp\left(-\frac{c_1(m-k)}{K^{rq'}}\right) +2\e, \quad q':=\frac{q}{q-1},
$$
where $c_1:= \frac{t^{q'}}{2(16\gamma)^{\frac{1}{q-1}} V^{q'}}$.
\end{Theorem}
In all theorems that follow we assume $rq'\ge 1$.
\begin{Theorem}\label{zT2.4} Let $X$ be a Banach space with $\rho(u)\le \gamma u^q$, $1<q\le 2$. Suppose $K$-sparse $f^\e$ satisfies {\bf A1}, {\bf A2} and $\|f_0-f^\e\|\le \e$. Then the WCGA($t$) with weakness parameter $t$ applied to $f_0$ provides after $$S := \lceil C'U^{q'}\ln (U+1) K^{rq'} \rceil $$ iterations 
$$
\|f_{S}\| \le CU\e\quad\text{for}\quad K+S \le D
$$
with an absolute constant $C$ and $C' = C_2(q)\gamma^{\frac{1}{q-1}} C_1^{q'} t^{-q'}$.
\end{Theorem}
We formulate an immediate corollary of Theorem \ref{zT2.4} with $\e=0$.
\begin{Corollary}\label{zC2.3} Let $X$ be a Banach space with $\rho(u)\le \gamma u^q$. Suppose $K$-sparse $f$ satisfies {\bf A1}, {\bf A2}. Then the WCGA($t$) with weakness parameter $t$ applied to $f$ recovers it exactly after $S := \lceil C'U^{q'}\ln (U+1) K^{rq'} \rceil $ iterations  under condition $K+S \le D$.
\end{Corollary}

We formulate versions of Theorem \ref{zT2.4} with assumptions {\bf A1}, {\bf A2} replaced by a single assumption {\bf A3} and replaced by two assumptions {\bf A2} and {\bf A3} (see \cite{VTbookMA}, pp.430-431).  

\begin{Theorem}\label{zT2.5} Let $X$ be a Banach space with $\rho(u)\le \gamma u^q$, $1<q\le 2$. Suppose $K$-sparse $f^\e$ satisfies {\bf A3}  and $\|f_0-f^\e\|\le \e$. Then the WCGA($t$) with weakness parameter $t$ applied to $f_0$ provides after $$S:=\lceil C(t,\gamma,q)V^{q'}\ln (VK) K^{rq'}\rceil$$ iterations
$$
\|f_{S}\| \le CVK^r\e\quad\text{for}\quad K+S\le D
$$
with an absolute constant $C$ and $C(t,\gamma,q) = C_2(q)\gamma^{\frac{1}{q-1}}  t^{-q'}$.
\end{Theorem}

\begin{Theorem}\label{zT2.6} Let $X$ be a Banach space with $\rho(u)\le \gamma u^q$, $1<q\le 2$. Suppose $K$-sparse $f^\e$ satisfies {\bf A2}, {\bf A3}  and $\|f_0-f^\e\|\le \e$. Then the WCGA($t$) with weakness parameter $t$ applied to $f_0$ provides after $$S:=\lceil C(t,\gamma,q)V^{q'}\ln (U+1) K^{rq'}\rceil$$ iterations
$$
\|f_{S}\| \le CU\e\quad\text{for}\quad K+S\le D
$$
with an absolute constant $C$ and $C(t,\gamma,q) = C_2(q)\gamma^{\frac{1}{q-1}}  t^{-q'}$.
\end{Theorem}

Theorems \ref{zT2.5} and \ref{zT2.3} imply the following analog of Theorem \ref{zT2.1}.
  
  \begin{Theorem}\label{zT2.7} Let $X$ be a Banach space with $\rho(u)\le \gamma u^q$, $1<q\le 2$. Suppose $K$-sparse $f^\e$ satisfies {\bf A3}  and $\|f_0-f^\e\|\le \e$. Then the WCGA($t$) with weakness parameter $t$ applied to $f_0$ provides after $$S:=\lceil C(t,\gamma,q)V^{q'}\ln (VK) K^{rq'}\rceil$$ iterations
$$
\|f_{S}\| \le C\e\quad\text{for}\quad K+S \le D
$$
with an absolute constant $C$ and $C(t,\gamma,q) = C_2(q)\gamma^{\frac{1}{q-1}}  t^{-q'}$.
\end{Theorem}

The following edition of Theorems \ref{zT2.1} and \ref{zT2.7} is also useful in applications. It follows from Theorems \ref{zT2.6} and \ref{zT2.3}.

 \begin{Theorem}\label{zT2.8} Let $X$ be a Banach space with $\rho(u)\le \gamma u^q$, $1<q\le 2$. Suppose $K$-sparse $f^\e$ satisfies {\bf A2}, {\bf A3}  and $\|f_0-f^\e\|\le \e$. Then the WCGA($t$) with weakness parameter $t$ applied to $f_0$ provides after $$S:= \lceil C(t,\gamma,q)V^{q'}\ln (U+1) K^{rq'}\rceil$$ iterations
$$
\|f_{S}\| \le C\e\quad\text{for}\quad K+S\le D
$$
with an absolute constant $C$ and $C(t,\gamma,q) = C_2(q)\gamma^{\frac{1}{q-1}}  t^{-q'}$.
\end{Theorem}

\subsection{Some specific examples for the WCGA}

In this subsection, following \cite{VT144} (see also \cite{VTbookMA}, Section 8.7.4), we discuss applications of Theorems from the previous subsection for specific dictionaries $\cD$. Mostly, $\cD$ will be a basis $\Psi$ for $X$. Because of that we use $m$ instead of $K$ in the notation of sparse approximation.  In some of our examples we take $X=L_p$, $2\le p<\infty$. Then it is known that $\rho(u,L_p) \le \gamma u^2$ with $\gamma = (p-1)/2$. In some other examples we take $X=L_p$, $1<p\le 2$. 
Then it is known that $\rho(u,L_p) \le \gamma u^{p}$,  with $\gamma = 1/p$.
 
In this subsection $L_p(\Omega)$, $\Omega$ is a bounded domain in $\bbR^d$, stands for the $L_p(\Omega,\mu)$ with $\mu$ being the normalized Lebesgue measure on $\Omega$. All the statements of this subsection hold for all $m\in \bbN$. 

\begin{Proposition}\label{zExample 1.} Let $\Psi$ be a uniformly bounded orthogonal system normalized in $L_p(\Omega)$, $2\le p<\infty$, $\Omega$ is a bounded domain. Then for any $f_0\in L_p(\Omega)$ we have 
 \begin{equation}\label{z4.4}
\|f_{C(t,p,\Omega)m\ln (m+1)}\|_p \le C\sigma_m(f_0,\Psi)_p .
\end{equation}
\end{Proposition}
 
The proof of Proposition \ref{zExample 1.} is based on   Theorem \ref{zT2.7}.
 
\begin{Corollary}\label{zExample 2.} Let $\Psi$ be the normalized in $L_p$, $2\le p<\infty$, real $d$-variate trigonometric
system. Then Proposition \ref{zExample 1.} applies and gives  
for any $f_0\in L_p$
\begin{equation}\label{z4.1}
\|f_{C(t,p,d)m\ln (m+1)}\|_p \le C\sigma_m(f_0,\Psi)_p .
\end{equation}
\end{Corollary}

\begin{Proposition}\label{zExample 1q.} Let $\Psi$ be a uniformly bounded orthogonal system normalized in $L_p(\Omega)$, $1< p\le 2$, $\Omega$ is a bounded domain. Then for any $f_0\in L_p(\Omega)$ we have
 \begin{equation}\label{z4.4q}
\|f_{C(t,p,\Omega)m^{p'-1}\ln (m+1)}\|_p \le C\sigma_m(f_0,\Psi)_p .
\end{equation}
\end{Proposition}
 
The proof of Proposition \ref{zExample 1q.} is based on   Theorem \ref{zT2.7}.  

\begin{Corollary}\label{zExample 2q.} Let $\Psi$ be the normalized in $L_p$, $1<p\le 2$, real $d$-variate trigonometric
system. Then Proposition \ref{zExample 1q.} applies and gives  
for any $f_0\in L_p$
\begin{equation}\label{z4.1q}
\|f_{C(t,p,d)m^{p'-1}\ln (m+1)}\|_p \le C\sigma_m(f_0,\Psi)_p .
\end{equation}
\end{Corollary}

\begin{Proposition}\label{zExample 4.} Let $\Psi$ be the normalized in $L_p$, $2\le p<\infty$, multivariate Haar basis 
${\mathcal H}^d_p={\mathcal H}_p\times\cdots\times {\mathcal H}_p$. 
Then 
\begin{equation}\label{z4.2}
\|f_{C(t,p,d)m^{2/p'}}\|_p \le C\sigma_m(f_0,{\mathcal H}^d_p)_p .
\end{equation}
\end{Proposition}
 
The proof of Proposition \ref{zExample 4.} is based on Theorem \ref{zT2.4}.

 \begin{Proposition}\label{zExample 4q.} Let $\Psi$ be the normalized in $L_p$, $1<p\le 2$, univariate Haar basis 
${\mathcal H}_p$.  Then    
\begin{equation}\label{z4.2q}
\|f_{C(t,p)m}\|_p \le C\sigma_m(f_0,{\mathcal H}_p)_p .
\end{equation}
\end{Proposition}
 
The proof of Proposition \ref{zExample 4q.} is based on Theorem \ref{zT2.8}.
  
 \begin{Proposition}\label{zExample 5.} Let $X$ be a Banach space with $\rho(u)\le \gamma u^2$. Assume that $\Psi$ is a normalized Schauder basis for $X$. Then     
 \begin{equation}\label{z4.5}
\|f_{C(t,X,\Psi)m^2\ln (m+1)}\| \le C\sigma_m(f_0,\Psi) .
\end{equation}
\end{Proposition}
 
The proof of Proposition \ref{zExample 5.} is based on Theorem \ref{zT2.7}.
 
We note that the above bound (\ref{z4.5}) still works if we replace the assumption that $\Psi$ is a Schauder basis by the assumption that a dictionary $\cD$ is $(1,D)$-unconditional with constant $U$. Then we obtain
$$
\|f_{C(t,\gamma,U)K^2\ln (K+1)}\| \le C\sigma_K(f_0,\Psi),\quad\text{for}\quad K+\lceil C(t,\gamma,U)K^2\ln K\rceil \le D .
$$

\begin{Proposition}\label{zExample 5q.} Let $X$ be a Banach space with $\rho(u)\le \gamma u^q$, $1<q\le 2$. Assume that $\Psi$ is a normalized Schauder basis for $X$. Then    
 \begin{equation}\label{z4.5q}
\|f_{C(t,X,\Psi)m^{q'}\ln (m+1)}\| \le C\sigma_m(f_0,\Psi) .
\end{equation}
\end{Proposition}
 
The proof of Proposition \ref{zExample 5q.} is based on Theorem \ref{zT2.7}.
 
We note that the above bound (\ref{z4.5q}) still works if we replace the assumption that $\Psi$ is a Schauder basis by the assumption that a dictionary $\cD$ is $(1,D)$-unconditional with constant $U$. Then we obtain
$$
\|f_{C(t,\gamma,q,U)K^{q'}\ln (K+1)}\| \le C\sigma_K(f_0,\cD),\quad\text{for}\quad K+\lceil C(t,\gamma,q,U)K^{q'}\ln K\rceil\le D .
$$

\section{Some other algorithms for recovery of sparse elements}
\label{QO}

As we pointed out in Remark \ref{LebR1} the Lebesgue-type inequalities guarantee exact recovery of sparse elements. 
In this section we discuss greedy type algorithms, which are based on other ideas than dual greedy and $X$-greedy algorithms are based on, but provide good results for exact recovery and Lebesgue-type inequalities for special dictionaries. Results of this section are based on the paper \cite{ST}, which is a follow up to the dissertation \cite{S}. 

In the paper \cite{ST} we proved the Lebesgue-type inequalities for greedy approximation in Banach spaces. Let $X$ be a Banach space with norm $\|\cdot\|:=\|\cdot\|_X$.  We introduce a new norm, associated with a dictionary $\cD$,  by the formula
$$
\|f\|_\cD:=\sup_{g\in\cD}\sup_{F_g}|F_g(f)|,\quad f\in X.
$$
Please, do not confuse $\|f\|_\cD$ with the $\|F\|_\cD$ (see Definition \ref{Dnorm}), which is defined for $F\in X^*$. 
  The following concept of the $M$-coherent dictionary in the case of Banach spaces (see \cite{VT107}) is a generalization of the well known concept in the case of Hilbert spaces. 

Let $\cD$ be a dictionary in a Banach space $X$. We define the coherence parameter of this dictionary in the following way
$$
M(\cD):= \sup_{g\neq h;g,h\in\cD}\sup_{F_g}|F_g(h)|.
$$
We note that, in general, a norming functional $F_g$ is not unique. This is why we take $\sup_{F_g}$ over all norming functionals of $g$ in the definitions of $\|f\|_\cD$ and $M(\cD)$. We do not need $\sup_{F_g}$ in those definitions   if for each 
$g\in\cD$ there is a unique norming functional $F_g\in X^*$. Then we define $\cD^*:=\{F_g,g\in\cD\}$ and call $\cD^*$ a {\it dual dictionary} to a dictionary $\cD$. 
 It is known that the uniqueness of the norming functional $F_g$ is equivalent to the property that $g$ is a point of Gateaux smoothness:
 $$
 \lim_{u\to 0}(\|g+uy\|+\|g-uy\|-2\|g\|)/u =0
 $$
 for any $y\in X$. In particular, if $X$ is uniformly smooth then $F_f$ is unique for any $f\neq 0$. We considered in \cite{VT107} the following greedy algorithm, which generalizes the Weak Orthogonal Greedy Algorithm to a smooth Banach space setting.
 
 {\bf Weak Quasi-Orthogonal Greedy Algorithm (WQOGA).} 
 Let $t\in (0,1]$. Denote $f_0:=f_0^{q,t}:=f$ (here and below index $q$ stands for {\it quasi-orthogonal}) and find $\varphi_1:=\varphi_1^{q,t}\in\cD$ such that
 $$
 |F_{\varphi_1}(f_0)| \ge t\sup_{g\in \cD}|F_g(f_0)|.
 $$
 Next, we find $c_1$ satisfying
 $$
 F_{\varphi_1}(f-c_1\varphi_1)=0.
 $$ 
 Denote $f_1:=f_1^{q,t}:=f-c_1\varphi_1$. 
 
 We continue this construction in an inductive way. Assume that we have already constructed residuals $f_0,f_1,\dots,f_{m-1}$ and dictionary elements $\varphi_1,\dots,\varphi_{m-1}$. Now, we pick an element
 $\varphi_m:=\varphi_m^{q,t}\in\cD$ such that   
 $$
 |F_{\varphi_m}(f_{m-1})| \ge t\sup_{g\in \cD}|F_g(f_{m-1})|.
 $$
 Next, we look for $c_1^m,\dots,c_m^m$ satisfying
 \begin{equation}\label{QO9.1}
 F_{\varphi_j}(f-\sum_{i=1}^mc_i^m\varphi_i)=0,\quad j=1,\dots,m. 
 \end{equation}
If there is no solution to (\ref{QO9.1}) then we stop, otherwise we denote $G_m:=G_m^{q,t} := \sum_{i=1}^mc_i^m\varphi_i$ and $f_m:=f_m^{q,t}:=f-G_m $ with $c_1^m,\dots,c_m^m$ satisfying (\ref{QO9.1}).

\begin{Remark} We note that (\ref{QO9.1}) has a unique solution if \newline 
$\det [F_{\varphi_j}(\varphi_i)]_{i,j=1}^m \neq 0$. We apply the WQOGA in the case of a dictionary with the coherence parameter $M:=M(\cD)$. Then, by a simple well known argument on the linear independence of the rows of the matrix $[F_{\varphi_j}(\varphi_i)]_{i,j=1}^m$, we conclude that (\ref{QO9.1}) has a unique solution for any $m<1+1/M$.   Thus, in the case of an $M$-coherent 
dictionary $\cD$, we can run the WQOGA for at least $[1/M]$ iterations.
\end{Remark} 

In the case $t=1$ we call the WQOGA the Quasi-Orthogonal Greedy Algorithm (QOGA). In the case of QOGA we need to make an extra assumption that the corresponding maximizer $\ff_m\in\cD$ exists. Clearly, it is the case when $\cD$ is finite.  

It was proved in \cite{VT107}  (see also \cite{VTbook}, p.382) that the WQOGA is as good as the WOGA in the sense of exact recovery of sparse signals with respect to incoherent dictionaries. The following result was obtained in \cite{VT107} (see Theorem 11.14 there).
 \begin{Theorem}[{\cite{VT107}}]\label{QOT1} Let $t\in (0,1]$. Assume that $\cD$ has coherence parameter $M$. Let $S<\frac{t}{1+t}(1+1/M)$. Then for any
 $f$ of the form
 \begin{equation}\label{QO1.1}
 f=\sum_{i=1}^Sa_i\psi_i,
 \end{equation}
 where $\psi_i$ are distinct elements of $\cD$, the WQOGA recovers it exactly after $S$ iterations. In other words we have that $f^{q,t}_S=0$.
 \end{Theorem}
It is known (see, for instance, \cite{VTbook}, pp.303--305) that the bound $S<\frac{1}{2}(1+1/M)$ is sharp for exact recovery by the OGA. 

 As above,  we define the best $m$-term approximation in the norm $Y$ as follows
$$
\sigma_m(f)_Y := \inf_{g\in\Sigma_m(\cD)}\|f-g\|_Y.
$$
In this section the norm $Y$ will be either the norm $X$ of our Banach space or the norm $\|\cdot\|_\cD$ defined above.
In \cite{ST} (see Theorems 1.4, 1.5 and Corollary 1.2 there) we proved the following two Lebesgue-type inequalities.
 \begin{Theorem}[{\cite{ST}}]\label{QOT0.1} Assume that $\cD$ is an $M$-coherent dictionary. Then for $m\le 1/(3M)$ we have for the QOGA
 \begin{equation}\label{QO0.2}
 \|f_m\|_\cD \le 13.5\sigma_m(f)_\cD.
 \end{equation}
 \end{Theorem}

 \begin{Theorem}[{\cite{ST}}]\label{QOT0.2} Assume that $\cD$ is an $M$-coherent dictionary in a Banach space $X$. There exists an absolute constant $C$   such that for $m\le 1/(3M)$ we have for the QOGA
$$
\|f_m\|_X \le C \inf_{g\in\Sigma_m(\cD)}(\|f-g\|_X + m\|f-g\|_\cD).
$$
\end{Theorem} 
\begin{Corollary}[{\cite{ST}}]\label{QOC0.2} Using the inequality $\|g\|_\cD \le \|g\|_X$ we obtain from Theorem \ref{QOT0.2}
$$
\|f_m\|_X \le C(1+m) \sigma_m(f)_X.
$$
\end{Corollary}

Inequality (\ref{QO0.2}) is a perfect (up to a constant 13.5) Lebesgue-type inequality. It indicates that the norm $\|\cdot\|_\cD$ used in that inequality is a suitable norm for analyzing performance of the QOGA. Corollary \ref{QOC0.2} shows that the Lebesgue-type inequality (\ref{QO0.2}) in the norm $\|\cdot\|_\cD$ implies the Lebesgue-type inequality in the norm $\|\cdot\|_X$. 

We do not know if the Lebesgue-type inequality in Corollary  \ref{QOC0.2}
 is sharp. We know that it can be improved  (see Theorem
2.6 in \cite{ST}) in a Hilbert space. In this case, an extra factor $(1+m)$ can
be replaced by $(1 + m^{1/2})$, and it is known that this cannot
be replaced by a slower growing in $m$ factor, see \cite{GMS}. Thus, in
the case of Hilbert spaces, the technique from \cite{ST} based on the general
inequality (\ref{QO0.2}) provides sharp Lebesque-type inequalities.

\newpage
 \medskip
 {\bf \Large Chapter III : Greedy approximation with respect to special dictionaries.}
  \medskip
  
  \section{Brief Introduction}
  \label{BI}
  
  Probably, the simplest example of an infinite dictionary would be an orthonormal basis of a Hilbert space. The next natural 
  simple dictionary would be a Schauder basis in a Banach space. Theory of greedy approximation with respect to bases is 
  well developed. The reader can find the corresponding results in the books \cite{VTbook}, \cite{VTsparse} and in the very recent survey paper \cite{AAT} (also, see references therein). We will not discuss this direction in detail here but only give a very brief description of the problem setting. 
  
  Let a Banach space $X$, with a basis $\Psi =\{\psi_k\}_{k=1}^\infty$,    be given. We assume that $\|\psi_k\|\ge C>0$, $k=1,2,\dots$,  and consider the following theoretical greedy algorithm. For a given element $f\in X$ we consider the expansion
\begin{equation}\label{BIn1}
f = \sum_{k=1}^\infty c_k(f,\Psi)\psi_k. 
\end{equation}
For an element $f\in X$ we say that a permutation $\rho$ of the positive integers   is  decreasing  if 
\begin{equation}\label{BIn2}
|c_{k_1}(f,\Psi) |\ge |c_{k_2}(f,\Psi) | \ge \dots ,  
\end{equation}
 where $\rho(j)=k_j$, $j=1,2,\dots$, and write $\rho \in D(f)$.
If the   inequalities are strict in (\ref{BIn2}), then $D(f)$ consists of only one permutation. We define the $m$th greedy approximant of $f$, with regard to the basis $\Psi$ corresponding to a permutation $\rho \in D(f)$, by the formula
$$
G_m(f):=G_m(f,\Psi)  :=G_m(f,\Psi,\rho) := \sum_{j=1}^m c_{k_j}(f,\Psi)\psi_{k_j}.
$$
We note that there is another natural greedy-type algorithm based on ordering $\|c_k(f,\Psi)\psi_k\|$ instead of ordering absolute values of coefficients. In this case we do not need the restriction $\|\psi_k\|\ge C>0$, $k=1,2,\dots$. Let $\Lambda_m(f)$ be a set of indices such that
$$
\min_{k\in \Lambda_m(f)}\|c_k(f,\Psi)\psi_k\| \ge \max_{k\notin \Lambda_m(f)}\|c_k(f,\Psi)\psi_k\|.
$$
We define $G^X_m(f,\Psi)$ by the formula
$$
G^X_m(f,\Psi) := S_{\Lambda_m(f)}(f,\Psi), \quad \text{where}\quad S_E(f):=S_E(f,\Psi) := \sum_{k\in E} c_k(f,\Psi)\psi_k.
$$
It is clear that for a normalized basis ($\|\psi_k\|=1,$ $k=1,2,\dots$) the above two greedy algorithms coincide.
It is also clear that the above greedy algorithm $G_m^X(\cdot,\Psi)$ can be considered as a greedy algorithm $G_m(\cdot,\Psi')$, with $\Psi':=\{\psi_k/\|\psi_k\|\}_{k=1}^\infty$ being a normalized version of the $\Psi$. Thus, usually the researchers concentrate on studying the algorithm $G_m(\cdot,\Psi)$. In the above definition of $G_m(\cdot,\Psi)$ we impose an extra condition on a basis $\Psi$: $\inf_k\|\psi_k\|>0$. This restriction allows us to define $G_m(f,\Psi)$ for all $f\in X$.

The above algorithm $G_m(\cdot,\Psi)$ is a simple algorithm, which describes the theoretical scheme  for $m$-term approximation of an element $f$. This algorithm is called  Thresholding Greedy Algorithm (TGA). In order to understand the efficiency of this algorithm we compare its accuracy with the best-possible accuracy when an approximant is a linear combination of $m$ terms from $\Psi$ (see Section \ref{Leb} above). We define the best $m$-term approximation with regard to $\Psi$ as follows:
$$
\sigma_m(f):=\sigma_m(f,\Psi)_X := \inf_{c_k, \La} \|f -\sum_{k\in \La} c_k\psi_k\|_X,
$$
where the infimum is taken over coefficients $c_k$ and sets of indices $\La$ with cardinality $|\La|=m$. The best we can achieve with the algorithm $G_m$ is
$$
\|f-G_m(f,\Psi,\rho)\|_X = \sigma_m(f,\Psi)_X,
$$
or the slightly weaker
\begin{equation}\label{BIn3}
\|f-G_m(f,\Psi,\rho)\|_X \le G\sigma_m(f,\Psi)_X,  
\end{equation}
for all elements $f\in X$, and with a constant $G=C(X,\Psi)$ independent of $f$ and $m$.
It is clear that, when $X=H$ is a Hilbert space and $\mathcal B$ is an orthonormal basis, we have
$$
\|f-G_m(f,\mathcal B,\rho)\|_H = \sigma_m(f,\mathcal B)_H.
$$

 The following concept of a greedy basis has been introduced in \cite{KonT}.

\begin{Definition}\label{BID1}
We call a basis $\Psi$ a greedy basis if for
every $f\in X$ there exists a permutation $\rho \in D(f)$ such
that
$$
\|f-G_m(f,\Psi,\rho)\|_X\leq C\sigma_m (f,\Psi)_X
$$
with a constant $C$ independent of $f$ and $m$.
\end{Definition}

Later, other important concepts of greedy type bases -- quasi-greedy, almost greedy, and others -- were introduced and studied in detail. For instance, characterization theorems for greedy and almost greedy bases were proved. The reader can find those results, for instance, in \cite{VTsparse}. We point out one important feature of bases, which distinguishes them from the redundant dictionaries. If $\Psi$ is a basis then we have the expansion  (\ref{BIn1}) and, therefore, we can identify 
the function $f$ with its sequence of coefficients $\{c_k(f,\Psi)\}$. This, in turn, allows us to apply the very simple greedy type algorithm -- the Thresholding Greedy Algorithm. In this paper we mostly discuss redundant dictionaries and need 
other, more complicated, greedy type algorithms.  

\section{Bilinear approximation}
\label{Bi}

The first example of sparse approximation with respect to redundant dictionaries 
 was considered by E. Schmidt in \cite{Sch}, who studied the
approximation of functions $f(x,y)$ of two variables by bilinear forms,
$$
\sum_{i=1}^mu_i(x)v_i(y),
$$ 
in $L_2([0,1]^2)$. Thus, in this case we use the following dictionary (bilinear dictionary)
\be\label{Bi1}
\Pi := \Pi(1) := \{u(x)v(y)\,:\, \|u\|_{L_2}=\|v\|_{L_2}=1\},
\ee
where the parameter $1$ in $\Pi(1)$ indicates that the functions $u$ and $v$ are functions of a single variable. 
This problem is closely connected with properties of the integral operator
$$
J_f(g) := \int_0^1 f(x,y)g(y) dy
$$
 with kernel $f(x,y)$. E. Schmidt (\cite{Sch}) gave an expansion (known as the Schmidt expansion)
\be\label{Bi2}
f(x,y) = \sum_{j=1}^\infty s_j(J_f) \phi_j(x)\psi_j(y),
\ee
where $\{s_j(J_f)\}$ is a nonincreasing sequence of singular numbers of $J_f$, i.e. $s_j(J_f) := \lambda_j(J^*_fJ_f)^{1/2}$, where $\{\lambda_j(A)\}$ is a sequence of eigenvalues of an operator $A$, and $J^*_f$ is the adjoint operator to $J_f$. The two sequences $\{\phi_j(x)\}$ and $\{\psi_j(y)\}$ form orthonormal sequences of eigenfunctions of the operators $J_fJ_f^*$ and $J_f^*J_f$, respectively. He also proved that 
$$
\left\|f(x,y) -\sum_{j=1}^m s_j(J_f) \phi_j(x)\psi_j(y)\right\|_{L_2} 
$$
\be\label{Bi3}
=
\inf_{u_j,v_j\in L_2, \quad j=1,\dots,m}\left\|f(x,y) -\sum_{j=1}^m  u_j(x)v_j(y)\right\|_{L_2}.
\ee
It was understood later that the above best bilinear approximation can be realized by the following greedy algorithm. Assume $c_j$, $u_j(x)$, $v_j(y)$, $\|u_j\|_{L_2}=\|v_j\|_{L_2}=1$, $j=1,\dots,m-1$, have been constructed after $m-1$ steps of the algorithm. At the $m$th step we choose $c_m$, $u_m(x)$, $v_m(y)$, $\|u_m\|_{L_2}=\|v_m\|_{L_2}=1$, to minimize 
$$
\left\|f(x,y)-\sum_{j=1}^m c_j u_j(x)v_j(y)\right\|_{L_2}.
$$
The above algorithm is the Pure Greedy Algorithm (PGA) with respect to $\Pi$. Therefore, for the PGA with respect to $\Pi$ we have for any $f\in L_2([0,1]^2)$
\be\label{Bi4}
\|f-G_m(f,\Pi)\|_2 = \sigma_m(f,\Pi)_2 = \left(\sum_{j=m+1}^\infty s_j(J_f)^2\right)^{1/2}.
\ee
Relation (\ref{Bi4}) shows that in the case of dictionary $\Pi$ the problems of approximation by the PGA (OGA as well),
of best $m$-term approximation, and of decay of singular numbers $s_j(J_f)$ are closely related. The problem of 
estimating the singular numbers $s_j(J_f)$ for functions $f$ coming from different smoothness classes has a long history.
We give a brief description of it (see also \cite{VT83}).  

\begin{Remark}\label{BiR1a} In this section we discuss the best $m$-term bilinear approximations of functions from different classes. We use the notation 
$$
\sigma_m(\bF,\Pi)_p := \sup_{f\in \bF}\sigma_m(f,\Pi)_p.
$$
Note, that in a number of papers on this topic the following notation is used as well
$$
\tau_m(\bF)_p := \sigma_m(\bF,\Pi)_p.
$$
\end{Remark}

\subsection{H{\"o}lder-Nikol'skii classes}

We begin with presentation of some results for standard
   periodic H{\"o}lder-Nikol'skii classes $NH^{R_1,R_2}_{q_1,q_2} := NH^\bR_\bq$, $\bR=(R_1,R_2)$, $\bq=(q_1,q_2)$, of functions of two variables. We define these classes in the following way. 
First of all we define the vector $L_{\bq}$-norm, $\bq=(q_1,\dots,q_d)$, of functions of $d$ variables $\bx=(x_1,\dots,x_d)$ as
$$
\|f(\bx)\|_{\bq} := \|\cdots\|f(\cdot,x_2,\dots,x_d)\|_{q_1}\cdots\|_{q_d}.
$$

The class $NH^{\bR}_{\bq}$,  is the set of periodic functions $f\in L_{\bq}(\bbT^d)$ such that for each 
$l_j = [R_j]+1$, $j=1,\dots,d$, the following relations hold
$$
\|f\|_{\bq} \le 1,\qquad \|\Delta^{l_j,j}_tf\|_{\bq} \le |t|^{R_j}, \quad j=1,\dots,d,
$$
where $\Delta^{l,j}_t$ is the $l$-th difference with step $t$ in the variable $x_j$. In the case of functions of a single variable $NH^R_q$ coincides with the standard H{\"o}lder class $H^R_q$.
 
We now give some historical remarks on estimating eigenvalues and
singular numbers of integral operators. We begin with the
following theorem that is a corollary of the Weyl Majorant Theorem
(see \cite{GoKr}, p.41).
 \begin{Theorem}\label{BiT5.8} Let $A$ be a compact
(completely continuous) operator in a Hilbert space $H$. Suppose
that 
$$ s_n(A) \ll n^{-r}, \quad r>0. 
$$ 
Then 
$$ |\lambda_n(A)|
\ll n^{-r}.
$$
\end{Theorem}
I. Fredholm \cite{F} proved that if the kernel $f(x,y)$ is a continuous
function and satisfies the condition 
$$
\sup_{x,y}|f(x,y+t)-f(x,y)|
\le C|t|^\alpha,\quad 0<\alpha\le 1,
$$
 then for an arbitrary $\rho >
2/(2\alpha +1)$ the series 
$$
 \sum_{j=1}^\infty
|\lambda_j(J_f)|^\rho <\infty
$$
 converges.

Starting with that article, smoothness conditions with respect to
one variable were imposed on the kernel. H. Weyl \cite{We} proved the
estimate 
$$ 
\lambda_n(J_f) = o(n^{-R-1/2}) 
$$ 
under the condition
that the kernel $f(x,y)$ is symmetric and continuous and that
$\partial^R{f}/\partial{x}^R$ is continuous. 

Let us introduce some more notation.
Define $NH^{(R,0)}_{q_1,q_2}$ as follows: $f(x,y)$ belongs to this
class if for all $y\in \bbT$ the function $f(\cdot,y)$ of $x$
belongs to the class $H^R_{q_1}B(y)$, and $B(y)$ is such that
$\|B(y)\|_{q_2}\le 1$. We use here the following notation. For a function class 
$F$ and a number $B>0$ we define $FB:= \{f:f/B \in F\}$.

E. Hille and J.D. Tamarkin \cite{HT}  
proved, in particular, that for $1<q\le 2$ and $R\ge 1$
$$
\sup_{f\in NH^{(R,0)}_{q,q'}}|\lambda_n(J_f)| \ll n^{-R-1+1/q}
(\log n)^R, \quad q'=q/(q-1), 
$$ 
and they conjectured that the
extra logarithmic factor can be removed or even replaced by a
logarithmic factor with a negative power.

The next important step was taken by F. Smithies \cite{Sm}. He proved the
estimate 
\be\label{Bi5.5}
\sup_{f\in NH^{(R,0)}_{q,2}}s_n(J_f) \ll
n^{-R-1+1/q},\quad 1<q\le 2,\quad R>1/q-1/2.  
\ee 
Of later results
we mention those of A.O. Gel'fond and M.G. Krein (see \cite{GoKr}, Ch.III,
S9.4), M.Sh. Birman and M.Z. Solomyak \cite{BS}, and J.A. Cochran \cite{Co}.

We proved in \cite{VT39} the   estimate (\ref{Bi5.5}) for $1\le q\le 2$,
which implies that (\ref{Bi5.5}) holds also for $q=1$. This bound was derived 
from the corresponding results on the best bilinear approximations. It was proved 
in \cite{VT39} (see Theorem 2.5 there) that 
$$
\sigma_m(NH^{(R,0)}_{q_1,q_2},\Pi)_{p_1,p_2}  \asymp
m^{-R+(1/q_1-\max(1/2,1/p_1))_+},\quad (a)_+ := \max(a,0),
$$ 
for $1\le q_1\le p_1 \le \infty$, $1\le q_2=p_2\le \infty$ and 
 $R>r(q_1,p_1)$. We denote here $r(q,p):=(1/q-1/p)_+$ for $1\le q\le p \le 2$ or $1\le p\le q\le \infty$ and $r(q,p):= \max(1/2,1/q)$ otherwise.  
 
 Further, the bilinear approximation was extended to the case of vectors $\bx =(x_1,\dots,x_d)$, $\by =(y_1,\dots,y_d)$ instead of scalars $x$ and $y$. We formulate only some of those results. For $d\in \bbN$ consider the dictionary
\be\label{Bi6}
 \Pi(d) := \{u(\bx)v(\by)\,:\, \|u\|_{L_2(\bbT^d)}=\|v\|_{L_2(\bbT^d)}=1\},
\ee
where the parameter $d$ in $ \Pi(d)$  indicates that the functions $u$ and $v$ are functions of $d$ variables. 
Then, as above the problem of best bilinear approximation of $f(\bx,\by)$ in the $L_2(\bbT^{2d})$ with respect to $\Pi(d)$   is closely connected with the properties of the integral operator
$$
J_f(g) := \int_{\bbT^d} f(\bx,\by)g(\by) d\mu(\by)
$$
 with the kernel $f(\bx,\by)$ and the normalized on $\bbT^d$ Lebesgue measure $\mu$.
 
  \begin{Remark}\label{BIR0} Note that in a number of papers the following notation is used for the best $m$-term bilinear approximation
 \be\label{tau}
 \tau_m(f)_\bp := \sigma_m(f,\Pi(d))_\bp.
 \ee
 \end{Remark}

 Let $\bR:= (R_1,\dots,R_n)\in \bbR^n_+$,   be a vector with positive coordinates and $\ba:= (a_1,\dots,a_{n})$, $1\le a_j \le \infty$, $j=1,\dots,n$. Define the H{\"o}lder-Nikol'skii class $NH^{\bR}_\ba$ (see above) as the set of functions $f(\bz)$, $\bz \in \bbT^n$ such that for each 
$l_j := [R_j]+1$, $j=1,\dots,n$,   the following relations hold for  $f(\bz)$
$$
\|f\|_{\ba} \le 1,\qquad \|\Delta^{l_j,j}_tf\|_{\ba} \le |t|^{R_j}, \quad j=1,\dots,n
$$
where $\Delta^{l,j}_t$ is the $l$-th difference with step $t$ in the variable $z_j$. 
 We consider a special case of the above classes, when $n=2d$, $d\in \bbN$, $\bR = (\bR^1,\bR^2)$, with $\bR^i :=(R^i_1,\dots,R^i_d) \in \bbR^d_+$, $i=1,2$, and $\ba$ is such that $a_j =q_1$ for $j=1,\dots,d$ and $a_j =q_2$ for $j=d+1,\dots,2d$. We use the following notation $NH^{\bR^1,\bR^2}_{\bq}$, $\bq :=(q_1,\dots,q_1,q_2,\dots,q_2)$ for this class. 
 
  Let $\bR^i :=(R^i_1,\dots,R^i_d) \in \bbR^d_+$, $i=1,2$, be vectors with positive coordinates. Denote
$$
g(\bR^i) := \left(\sum_{j=1}^d (R^i_j)^{-1}\right)^{-1},\quad i=1,2. 
$$
 
 In addition, we use the following notations $\eta_i := (1/q_i -1/2)_+$, $i=1,2$ and
 $$
 \rho_1(\bq,\bR) := \eta_1(1-\eta_2/g(\bR^2))^{-1},\qquad \rho_2(\bq,\bR) := \eta_2(1-\eta_1/g(\bR^1))^{-1}.
 $$
 
  The following results are from \cite{VT47} (see Theorems 4.1, 4.1', 4.2, 4.2' there). 
 
 \begin{Theorem}[{\cite{VT47}}]\label{BiT1} Let $\bR =(\bR^1,\bR^2)$ be such that $g(\bR^1) \le g(\bR^2)$ and $g(\bR^1) > \eta_1$,
 $g(\bR^2) > \rho_2(\bq,\bR)$. Then for all $1\le \bq \le \infty$ we have
 $$
  \sigma_m(NH^{\bR^1,\bR^2}_{\bq},\Pi(d))_2 \asymp m^{-g(\bR^2) +\eta_1 g(\bR^2)/g(\bR^1)}.
 $$
 \end{Theorem}
  
   \begin{Corollary}[{\cite{VT47}}]\label{BiC1} Let $\bR =(\bR^1,\bR^2)$ be as in Theorem \ref{BiT1}. Then for all $1\le \bq \le \infty$ we have
 $$
 \sup_{f\in NH^{\bR^1,\bR^2}_{\bq}} s_m(J_f) \asymp m^{-g(\bR^2) +\eta_1 g(\bR^2)/g(\bR^1)-1/2}.
 $$
 \end{Corollary}

 \begin{Theorem}[{\cite{VT47}}]\label{BiT2} Let $\bR =(\bR^1,\bR^2)$ be such that $g(\bR^1) \ge g(\bR^2)$ and $g(\bR^2) > \eta_2$,
 $g(\bR^1) > \rho_1(\bq,\bR)$. Then for all $1\le \bq \le \infty$ we have
 $$
   \sigma_m(NH^{\bR^1,\bR^2}_{\bq},\Pi(d))_2 \asymp m^{-g(\bR^1) +\eta_2 g(\bR^1)/g(\bR^2)+\eta_1-\eta_2}.
 $$
 \end{Theorem}
  
   \begin{Corollary}[{\cite{VT47}}]\label{BiC2} Let $\bR =(\bR^1,\bR^2)$ be as in Theorem \ref{BiT2}. Then for all $1\le \bq \le \infty$ we have
 $$
 \sup_{f\in NH^{\bR^1,\bR^2}_{\bq}} s_m(J_f) \asymp m^{-g(\bR^1) +\eta_2 g(\bR^1)/g(\bR^2)+\eta_1-\eta_2-1/2}.
 $$
 \end{Corollary}

Theorems \ref{BiT1} and \ref{BiT2} solve the problem of the behavior (in the sense of order) of the best $m$-term bilinear approximations in the $L_2$ norm for a big variety of the H{\"o}lder-Nikol'skii classes $NH^{\bR^1,\bR^2}_{\bq}$. The reader can find some results on the best $m$-term bilinear approximations in the $L_\bp$ norm for classes $NH^{\bR^1,\bR^2}_{\bq}$ in \cite{VT39}. 

The bilinear approximation turns out to be useful in studying not only the sequences $s_m(J_f)$ of singular numbers but also some of their natural generalizations. We now present the appropriate definitions. We begin with a classical concept of the Kolmogorov width. Let $X$ be a Banach space and $\bF\subset X$ be a  compact subset of $X$. The quantities  
$$
d_n (\bF, X) :=  \inf_{\{u_i\}_{i=1}^n\subset X}
\sup_{f\in \bF}
\inf_{c_i} \left \| f - \sum_{i=1}^{n}
c_i u_i \right\|_X, \quad n = 1, 2, \dots,
$$
are called the {\it Kolmogorov widths} of $\bF$ in $X$. 

For a function $f\in L_1(\bbT^{2d})$ written in the form $f(\bx,\by)$, $\bx,\by \in \bbT^d$ define a function class
$$
\bW^f_q := \left\{g\,:\, g(\bx) = \int_{\bbT^d} f(\bx,\by)\ff(\by) d\mu(\by),\quad \|\ff\|_q \le 1\right\}.
$$
It is well known (see, for instance, \cite{GoKr}) that for $f\in L_2(\bbT^{2d})$ we have
\be\label{Bi7}
s_{m+1}(J_f) = d_m(\bW^f_2,L_2). 
\ee
The following result from \cite{VT47} (see Theorem 4.3 there) is an extension  of Corollary \ref{BiC2}.

\begin{Theorem}[{\cite{VT47}}]\label{BiT3} Let $\bR =(\bR^1,\bR^2)$ be such that $g(\bR^1) \ge g(\bR^2)$ and $g(\bR^2) > \eta_2$.  Then for all $1\le \bq \le \infty$ and all $1\le b \le \infty$ we have
 $$
 \sup_{f\in NH^{\bR^1,\bR^2}_{\bq}} d_m((\bW^f_2,L_b)) \asymp m^{-g(\bR^1) +\eta_2 g(\bR^1)/g(\bR^2)+\eta_1-\eta_2-1/2}
 $$
 under extra conditions: $g(\bR^1) > \rho_1(\bq,\bR)$ for $b\in [1,2]$ and $g(\bR^1) > \rho_1(\bq,\bR)+1/2$ for $b>2$.
 \end{Theorem}

\subsection{Mixed smoothness classes}

In this section we discuss bilinear approximation of the multivariate functions, which have mixed smoothness. 
 We begin with the definition of classes $\bW^\ba_\bq$ (see, for instance, \cite{VTmon}, p.31).
\begin{Definition}\label{BiD1}
In the univariate case, for $a>0$, let
\be\label{Bi8}
F_a(x):= 1+2\sum_{k=1}^\infty k^{-a}\cos (kx-a\pi/2)
\ee
and in the multivariate case, for $\ba=(a_1,\dots,a_n) \in \bbR^n_+$, $\bx=(x_1,\dots,x_n)\in \bbT^n$, let
$$
F_\ba(\bx) := \prod_{j=1}^n F_{a_j}(x_j).
$$
Denote for $\mathbf{1}\le \bq\le \infty$ (we understand the vector inequality coordinate wise)
$$
\bW^\ba_\bq := \{f:f=\varphi\ast F_\ba,\quad \|\varphi\|_\bq \le 1\},
$$
where
$$
(\varphi \ast F_\ba)(\bx):= (2\pi)^{-d}\int_{\bbT^n} F_\ba(\bx-\by) \ff(\by)d\by.
$$
\end{Definition}
The classes $\bW^\ba_\bq$ are classical classes of functions with {\it dominated mixed derivative}
(Sobolev-type classes of functions with mixed smoothness).
 
We now proceed to the definition of the classes $\bH^\ba_\bq$, which is based on the mixed differences (see, for instance, \cite{VTmon}, p.31).  
 
\begin{Definition}\label{BiD2}
Let  $\btt =(t_1,\dots,t_n)$ and $\Delta_{\btt}^l f(\bx)$
be the mixed $l$-th difference with
step $t_j$ in the variable $x_j$, that is
$$
\Delta_{\btt}^l f(\bx) :=\Delta_{t_d,d}^l\cdots\Delta_{t_1,1}^l
f(x_1,\dots ,x_d ) .
$$
Let $e$ be a subset of natural numbers in $[1,d ]$. We denote
$$
\Delta_{\btt}^l (e) :=\prod_{j\in e}\Delta_{t_j,j}^l,\qquad
\Delta_{\btt}^l (\varnothing) := Id \,-\, \text{identity operator}.
$$
We define the class $\bH_{\bq,l}^\ba B$, $l > \|\ba\|_\infty$,   as the set of
$f\in L_\bq(\bbT^d)$ such that for any $e$
\be\label{Bi9}
\bigl\|\Delta_{\btt}^l(e)f(\bx)\bigr\|_\bq\le B
\prod_{j\in e} |t_j |^{a_j} .
\ee
In the case $B=1$ we omit it. It is known (see, for instance, \cite{VTmon}, p.32, for the scalar $q$ and \cite{VT32} for the vector $\bq$) that the classes $\bH^\ba_{\bq,l}$ with different $l>\|\ba\|_\infty$ are equivalent. So, for convenience we omit $l$ from the notation. 
\end{Definition}

The reader can find results on approximation properties of these classes in the books \cite{VTmon}, \cite{VTbookMA}, and \cite{DTU}.
In this section we consider the case, when $n=2d$, $d\in \bbN$, $\mathbf{1}\le \bq \le \infty$, and $\ba$ has a special form: $a_j = r_1$, $a_{d+j}=r_2$ for $j=1,\dots,d$. In this case we write $\bW^\br_\bq$ and $\bH^\br_\bq$ with $\br=(r_1,\dots,r_1,r_2,\dots,r_2)$. 

 We begin with a result in the case of functions of two variables ($n=2$, i.e. $d=1$). The following result is from \cite{VT32} (see Theorems 1, 2.1, and 3.2 there). We need the following notation for $1\le q,p\le\infty$
 \be\label{ksi}
 \xi(q,p):= \left(\frac{1}{q} - \max\left(\frac{1}{2},\frac{1}{p}\right)\right)_+, \quad  (a)_+ :=\max(a,0).
\ee
 
 \begin{Theorem}[{\cite{VT32}}]\label{BiT4} Let $d=1$ and $\bF^\br_\bq$ denote one of the classes $\bW^\br_\bq$ or $\bH^\br_\bq$. Then for $\br > \mathbf{1}$ and $1\le q_1\le p_1 \le \infty$, $1\le q_2,p_2 \le \infty$ we have 
 $$
  \sigma_m(\bF^\br_\bq,\Pi)_\bp \asymp m^{-r_1-r_2 + \xi(q_1,p_1)}.  
 $$
 \end{Theorem}

\begin{Remark}\label{BiLB} Note that Theorem \ref{BiT4} is proved in \cite{VT32} under weaker conditions on $\br$ than above. That restriction on $\br$ is needed for the proof of the upper bounds. For the lower bounds it is sufficient to assume that $\br > (1/q_1-1/p_1, (1/q_2-1/p_2)_+)$.
\end{Remark}

 \begin{Corollary}[{\cite{VT32}}]\label{BiC3} Under conditions of Theorem \ref{BiT4} we have 
 $$
 \sup_{f\in \bF^\br_\bq} s_m(J_f) \asymp m^{-r_1-r_2 +  \max\left(\frac{1}{2},\frac{1}{q_1}\right)-1}.
   $$
 \end{Corollary} 
 
 The following analog of Theorem \ref{BiT3} is proved in \cite{VT32} (see Theorems 4.1 and 4.2 there).
 
  \begin{Theorem}[{\cite{VT32}}]\label{BiT5} Let $d=1$ and $\bF^\br_1$ denote one of the classes $\bW^\br_1$ or $\bH^\br_1$. Then for $1\le q,p \le \infty$ and $\br > (1,1+\max(1/2,1/q))$   we have 
 $$
\sup_{f\in \bF^\br_1}  d_m(\bW^f_q)_p \asymp m^{-r_1-r_2 + \xi(q,p)} 
 $$
 with $\xi(q,p)$ defined in (\ref{ksi}).
 \end{Theorem}
 
 In the above Theorem \ref{BiT5} we consider the case of classes $\bF^\br_1$. Some results on the classes $\bF^\br_\bq$
 are obtained in \cite{VT47} (see Theorem 3.1' there). 
We refer the reader to the paper \cite{VT32} for further results and historical comments on bilinear approximation of functions on two variables with mixed smoothness. 

In the case $d>1$ results are not as complete as in the case $d=1$ discussed above. We now present some results in that direction from \cite{VT47} (see Theorems 2.1, 2.2, and 3.2 there). 

 \begin{Theorem}[{\cite{VT47}}]\label{BiT6} Let $n=2d$, $d\in \bbN$, and let $\br=(r_1,\dots,r_1,r_2,\dots,r_2)$ be such that 
 $r_1 >1/2$ and $r_2 >1/2$. Then for any $\mathbf{2} \le \bq \le \infty$ and $\mathbf{2} \le \bp <\infty$
  we have 
 $$
   \sigma_m(\bW^\br_\bq,\Pi(d))_\bp \asymp m^{-r_1-r_2} (\log m)^{(r_1+r_2)(d-1)}. 
 $$
 \end{Theorem}
 
 \begin{Theorem}[{\cite{VT47}}]\label{BiT7} Let $n=2d$, $d\in \bbN$, and let $\br=(r_1,\dots,r_1,r_2,\dots,r_2)$ be such that 
 $r_1 >1/2$ and $r_2 >1/2$. Then for any $\mathbf{2} \le \bq \le \infty$ and $\mathbf{2} \le \bp <\infty$
  we have 
 $$
   \sigma_m( \bH^\br_\bq,\Pi(d))_\bp \asymp m^{-r_1-r_2} (\log m)^{(r_1+r_2+1)(d-1)}. 
 $$
 \end{Theorem}

\begin{Remark}\label{BiR1} In Theorems \ref{BiT6} and \ref{BiT7} it is sufficient to require $r_i >0$, $i=1,2$ in the case $\bq\ge \bp$.
\end{Remark}

Here is an analog of Theorem \ref{BiT5}, which holds for $d=1$, in the case $d>1$. 

\begin{Theorem}[{\cite{VT32}}]\label{BiT5a} Let $d\in\bbN$ and $\mathbf{2} \le \bq \le \infty$, $2\le a <\infty$, $1<b<\infty$. 
Assume that in the case $b\in (1,2]$ we have $r_i>0$, $i=1,2$, and in the case $b\in (2,\infty)$ we have $r_1>1/2$, $r_2>0$. Then
 $$
\sup_{f\in \bW^\br_\bq}  d_m(\bW^f_a)_b \asymp (m^{-1}(\log m)^{d-1})^{r_1+r_2}m^{-1/2} .
 $$
 \end{Theorem}

\section{Multilinear approximation}
\label{Mu}

In this section we study multilinear approximation (nonlinear tensor product approximation) of multivariate functions.
As above, we consider periodic functions defined on $\bbT^d$ equipped with the normalized Lebesgue measure. 
In other words, say in the case of $L_2(\bbT^d)$, we are interested in studying $m$-term approximations of functions with respect to the dictionary
$$
\Pi^d := \{g(x_1,\dots,x_d): g(x_1,\dots,x_d)=\prod_{i=1}^d u^i(x_i),\quad \|u^i\|_2 =1, \, i=1,\dots,d\},
$$
where $u^i(x_i)$, $i=1,\dots,d$, are arbitrary normalized univariate functions from $L_2(\bbT)$. In the case $d=2$ we have $\Pi^2 = \Pi = \Pi(1)$.
We can give the following interpretation to the dictionary $\Pi^d$ and its $L_p$ analogs. For a Banach space $X$ denote 
by $S(X)$ its unit sphere $\{f\in X\,:\, \|f\|_X=1\}$. Then we can write
$$
\Pi^d = S(L_2(\bbT)) \times \cdots \times S(L_2(\bbT)) \quad (d\,\text{times}).
$$
In words -- $\Pi^d$ is a tensor (direct) product of $d$ copies of the unit spheres of the univariate $L_2(\bbT)$ spaces. For $1\le p\le \infty$ 
we define
$$
\Pi^d_p := S(L_p(\bbT)) \times \cdots \times S(L_p(\bbT)) \quad (d\,\text{times}).
$$

\begin{Remark}\label{MuR1} In this section we discuss the best $m$-term approximations of functions from different classes. We use the notation 
$$
\sigma_m(\bF,\Pi^d)_p := \sup_{f\in \bF}\sigma_m(f,\Pi^d)_p.
$$
Note, that in a number of papers on this topic the following notation is used as well
$$
\Theta_m(\bF)_p := \sigma_m(\bF,\Pi^d)_p.
$$
\end{Remark}

The dictionary $\Pi^2$ is the bilinear dictionary $\Pi$, which was discussed in detail above. We already pointed out (see (\ref{Bi4})) that the dictionary $\Pi$ has a very interesting and important property: The PGA and the OGA realize the best 
$m$-term approximations. Unfortunately, the dictionary $\Pi^d$ with $d>2$ does not have this special property. This makes 
the study of its approximation properties much more difficult than the study of the bilinear dictionary $\Pi$. In this section we present some results on sparse approximation with respect to $\Pi^d$. Theorem \ref{BiT4} gives rather complete 
results on the bilinear dictionary $\Pi^2$. The following   upper estimate from \cite{VT35} (see Theorem 4.1 there) in the case $q=p=2$
\begin{equation}\label{Mu1}
\sigma_m(\bW^r_2,\Pi^d)_2 \ll m^{-rd/(d-1)} ,\quad r>0,  
\end{equation}
is an extension of Theorem \ref{BiT4} in the case $q=p=2$. Further progress was obtained in \cite{VT147} (see Theorem 1.1 there). Let $2\le p<\infty$ and $r> (d-1)/d$. Then
\begin{equation}\label{Mu2}
\sigma_m(\bW^r_2,\Pi^d)_p \ll  \left(\frac{m}{(\log m)^{d-1}}\right)^{-\frac{rd}{d-1}}. 
\end{equation}

{\bf Comment \ref{Mu}.1.} The proof of the bound \ref{Mu2} in \cite{VT147} is not constructive. It goes by induction and uses 
a nonconstructive bound in the case $d=2$, which is obtained in \cite{VT35}. The corresponding proof from \cite{VT35} uses the bounds for the Kolmogorov width $d_n(W^r_2, L_\infty)$ of the class $W^r_2$ of univariate functions, proved by Kashin \cite{Ka}. Kashin's proof is a probabilistic one, which provides existence of a good linear subspace for approximation, but there is no known explicit constructions of such subspaces. This problem is related to a problem from compressed sensing  on construction of good matrices with Restricted Isometry Property (see, for instance, \cite{VTbook}, Ch.5 and \cite{FR}). It is an outstanding difficult open problem.  
In \cite{VT147} we discuss constructive ways of building good multilinear approximations. The simplest way would be to use known results about $m$-term approximation with respect to special systems with tensor product structure. We illustrate this idea in \cite{VT147} on the example of Thresholding Greedy Algorithm with respect to a special basis. 

We now discuss some known lower bounds. In the case $d=2$ the lower bound 
\begin{equation}\label{Mu3}
\sigma_m(\bW^r_p)_p \gg m^{-2r},\qquad 1\le p\le \infty ,  
\end{equation}
follows from more general results in \cite{VT32} (see Theorem \ref{BiT4} and Remark \ref{BiLB} above). A stronger result 
\begin{equation}\label{Mu4}
\sigma_m(\bW^r_\infty)_1 \gg m^{-2r}  
\end{equation}
follows from Theorem 1.1 in \cite{VT46}. We formulate that important result here.  Let $\cT^d := \{e^{i(\bk,\bx)}\}_{\bk\in \bbZ^d}$ be the trigonometric system. Denote for $\bN = (N_1,\dots,N_d)$, $N_j\in \bbN_0$, $j=1,\dots,d$,
$$
\cT(\bN,d) := \left\{f=\sum_{\bk: |k_j|\le N_j, j=1,\dots,d} c_j e^{i(\bk,\bx)}\right\},\quad \vartheta(\bN) :=\prod_{j=1}^d (2N_j+1). 
$$
Let $\cT(\bN,d)_\infty$ denote the unit $L_\infty$-ball of the subspace $\cT(\bN,d)$.

\begin{Theorem}[{\cite{VT46}}]\label{MuT1} Let $a,d \in \bbN$ be such that $a\in [1,d)$. There exists a positive constant $C(d)$ with the following property.  For $\bN = (N_1,\dots,N_d)$
denote $\bN^1 := (N_1,\dots,N_a)$ and $\bN^2 := (N_{a+1},\dots,N_d)$. Then for any 
$$
m\le \frac{1}{5} \min (\vartheta(\bN^1),\vartheta(\bN^2))
$$
there is a function $t(\bx^1,\bx^2)$, $\bx^1:=(x_1,\dots,x_a)$,  $\bx^2:=(x_{a+1},\dots,x_d)$ in the $\cT(\bN,d)_\infty$ such that for any integrable functions $u_i(\bx^1)$, $v_i(\bx^2)$, $i=1,\dots,m$, we have 
$$
\left\|t(\bx^1,\bx^2) - \sum_{i=1}^m u_i(\bx^1)v_i(\bx^2)\right\|_1 \ge C(d) .
$$
\end{Theorem}

Clearly, in the case $d=2$, $a=1$ Theorem \ref{MuT1} provides the lower bound for the bilinear approximations
\be\label{Mu4}
\sigma_m(\cT(\bN,2)_\infty,\Pi)_1 \ge C(2), \quad\text{for}\quad  m\le c\min(N_1,N_2). 
\ee

A very interesting development of Theorem \ref{MuT1} was obtained recently in the paper \cite{Ma} (see Statement 3.3 there). 

\begin{Theorem}[{\cite{Ma}}]\label{MuT2} There exist two positive constant $C(d)$ and $c(d)$ such that
$$
\sigma_m(\cT(\bN,d)_\infty,\Pi^d)_1 \ge C(d), \quad \text{for}\quad m\le c(d)\frac{N_1\cdots N_d}{\max_j N_j}. 
$$
\end{Theorem}

Theorem \ref{MuT2} was proved by a method, which is based on the ideas substantially different from ideas, used in the proof of Theorem \ref{MuT1}. The algebraic method of Alon-Frankl-R{\"o}dl, developed in the theory of matrix and tensor rigidity, was used in \cite{Ma}. As a corollary of Theorem \ref{MuT2}, which provides the lower bounds, and relation (\ref{Mu1}), which provides the upper bounds we obtain for $r>0$ and $1\le p\le 2\le q \le \infty$ 
\be\label{Mu5}
\sigma_m(\bW^r_q,\Pi^d)_p \asymp m^{-rd/(d-1)}. 
\ee

\section{Brief comments on some other dictionaries}
\label{OD} 

In this section we only mention some of the well known redundant dictionaries and will not discuss any results on them. 

{\bf Ridge dictionary.}  The following problem, which  is well known in statistics under the name of projection pursuit regression problem,  is an example of nonlinear sparse approximation with respect to a redundant dictionary. The problem is to approximate in $L_2(\Omega,\mu)$ a given function $f\in L_2(\Omega,\mu)$ by a sum of ridge functions, i.e. by
\be\label{OD1}
\sum_{j=1}^m r_j(\omega_j\cdot \bx), \quad \bx\in \Omega\subset \bbR^d,\quad \omega_j \in \bbR^d,\, \|\omega_j\|_2=1, \quad j=1,\dots,m,
\ee
where $r_j$, $j=1,\dots,m$, are univariate functions. The following greedy-type  algorithm (projection pursuit) was proposed in \cite{FS} to solve this problem. 
Assume functions $r_1,\dots,r_{m-1}$ and vectors $\omega_1,\dots,\omega_{m-1}$
 have been determined after $m-1$ iterations of the algorithm. Choose at $m$th step a unit vector $\omega_m$ and a function $r_m$ to minimize the error
$$
\|f(\bx) -\sum_{j=1}^m r_j(\omega_j\cdot \bx)\|_{L_2(\Omega,\mu)}.
$$
This is an example of realization of the Pure Greedy Algorithm with respect to the ridge dictionary $\cR_2^d$ with
 $$
 \cR^d_p:= \{r(\omega\cdot \bx), \, \bx \in\Omega\subset \bbR^d, \, \omega \in \bbR^d, \, \|\omega\|_2=1, \,  \|r(\omega\cdot \bx)\|_{L_p(\Omega,\mu)}=1\},
$$
where $r$ runs over all univariate functions. 

{\bf Neural networks.}  We fix a univariate function  $\sigma(t)$,  $t\in \bbR$. Usually, this function 
takes values in  $(0,1)$ and increases. Then as a system (dictionary) we consider
$$
 \{g(\bx)\,: \, g(\bx)= \sigma(\omega\cdot \bx+b),\,  \, \bx \in\Omega\subset \bbR^d, \, \omega \in \bbR^d, \, \|\omega\|_2=1, \, b\in \bbR  \}.
$$
Constructions based on this dictionary are called {\it shallow neural networks} or {\it neural networks 
with one layer}. 

We build approximating manifolds inductively. Let  $\bx \in \Omega \subset \bbR^d$. Fix a univariate function  $h(t)$. A popular one is the  $ReLU$ function:  $ReLU(t)=0$ for $t<0$ and
 $ReLu(t)=t$ for  $t\ge 0$.  ReLU = Rectified Linear Unit.

Take  numbers  $s,n\in\bbN$ and build  $s$-term approximants of depth  $n$ (neural network with  $n$ layers). In the capacity of parameters take  $n$ matrices:  $A_1$ of 
size  $s\times d$,  $A_2,\dots,A_n$ of size  $s\times s$ and vectors  $\bb^1,\dots,\bb^n, \bc$ from  $\bbR^s$. 

At the first step define  $\by^1\in\bbR^s$
$$
 \by^1 := h(A_1\bx + \bb^1):= (h((A_1\bx)_1 + b^1_1),\dots,h((A_1\bx)_s + b^1_s))^T.
$$
Note that  $\by^1$ is a function on  $\bx$. 

At the  $k$th step ($k=2,\dots,n$) define
$$
 \by^k := h(A_k\by^{k-1} + \bb^k)
$$
$$
 := (h((A_k\by^{k-1})_1 + b^k_1),\dots,h((A_k\by^{k-1})_s + b^k_s))^T.
$$
Finally, after the  $n$th step we define
$$
 g_n(\bx) := \<\bc,\by^n(\bx)\> = \sum_{j=1}^s c_jy^n_j(\bx).
$$
Thus we build a manifold $\cM:= \cM(d,s,n)$, which is described by the following parameters:
 $n$ matrices:  $A_1$ of size  $s\times d$,  $A_2,\dots,A_n$ of size $s\times s$
and vectors:  $\bb^1,\dots,\bb^n, \bc$ from  $\bbR^s$. 
Then the problem of best approximation, say in the $L_2(\Omega,\mu)$ norm, of a given function $f \in L_2(\Omega,\mu)$ 
by elements of $\cM$ is formulated as follows
$$
\inf_{g_n \in \cM} \|f-g_n\|_{L_2(\Omega,\mu)}.
$$

\newpage
 \medskip
 {\bf \Large Chapter IV : Some typical proofs}
  \medskip

A typical way of studying convergence and rate of convergence of greedy algorithms consists of two steps. At the first step we obtain an inequality for the norm of residual $\|f_m\|$ in terms of the $\|f_{m-1}\|$. At the second step we analyze the corresponding recursive inequalities and derive bounds on the sequence $\{\|f_m\|\}$. We begin with results, which are useful at the second step.

\section{Some lemmas about numerical sequences}
\label{Le}

The following Lemma \ref{LeL1} is from \cite{VT196} (see, for instance, \cite{VTbook}, p.91, in case $C_1=C_2$). 

\begin{Lemma}\label{LeL1} Let $\{a_m\}_{m=0}^\infty$
be a sequence of non-negative
 numbers satisfying the inequalities
$$
a_0 \le C_1, \quad a_{m+1} \le a_m(1 - a_mC_2) , \quad m = 0,1,2, \dots,\quad C_1,C_2>0 .
$$
Then we have for each $m$
$$
a_m \le (C_1^{-1}+C_2m)^{-1} .
$$
\end{Lemma}
\begin{proof} The proof is by induction on $m$. For $m = 0$ the statement
is true by assumption. We assume $a_m \le (C_1^{-1}+C_2m)^{-1}$ and prove that $a_{m+1} \le
(C_1^{-1}+C_2(m+1))^{-1}$. If $a_{m+1} = 0$ this statement is obvious. Assume therefore
that $a_{m+1} > 0$. Then we have 
$$
a_{m+1}^{-1} \ge a_m^{-1}(1 - a_mC_2)^{-1} \ge a_m^{-1}(1 + a_mC_2) =
a_m^{-1} + C_2 \ge C_1^{-1}+(m+1)C_2 ,
$$
which implies $a_{m+1} \le (C_1^{-1}+C_2(m+1))^{-1}$ . 
\end{proof}

We shall need the
following simple generalization of  Lemma \ref{LeL1}, different versions of which are well known (see, for example, \cite{VTbook}, p.91 and \cite{VT209}).

\begin{Lemma}\label{HL1} Let a number $C_1>0$ and a sequence $\{y_k\}_{k=1}^\infty$, $y_k \ge 0$, $k=1,2,\dots$, be given.
Assume that $\{a_m\}_{m=0}^\infty$
is a sequence of non-negative
 numbers satisfying the inequalities
$$
a_0 \le C_1, \quad a_{m+1} \le a_m(1 - a_m y_{m+1}) , \quad m = 0,1,2, \dots,\quad C_1>0 .
$$
Then we have for each $m$
$$
a_m \le \left(C_1^{-1}+\sum_{k=1}^{m} y_k\right)^{-1} .
$$
\end{Lemma}
\begin{proof} The proof is by induction on $m$. For $m = 0$ the statement
is true by assumption. We prove that
 $$
 a_m \le \left(C_1^{-1}+\sum_{k=1}^{m} y_k\right)^{-1}\quad \text{implies} \quad   a_{m+1} \le
\left(C_1^{-1}+\sum_{k=1}^{m+1} y_k\right)^{-1}.
$$
 If $a_{m+1} = 0$ this statement is obvious. Assume therefore
that $a_{m+1} > 0$. Then we have 
$$
a_{m+1}^{-1} \ge a_m^{-1}(1 - a_m y_{m+1})^{-1} \ge a_m^{-1}(1 + a_m y_{m+1}) =
a_m^{-1} + y_{m+1} \ge C_1^{-1}+\sum_{k=1}^{m+1} y_k ,
$$
which proves the required inequality. 
\end{proof}

The following Lemma \ref{LeL2} is from \cite{DT1} (see also \cite{VTbook}, p.89). 

\begin{Lemma}\label{LeL2}  If $A>0$ and $\{a_n\}_{n=1}^\infty$ is a sequence of non-negative numbers
satisfying  $a_1\le A$ and   
\be\label{Le3.1}
a_m\le a_{m-1}-\frac{2 }{m} a_{m-1}+ \frac{A}{m^2},\quad m=2,3,\dots, 
\ee
then
\be\label{Le3.2}
a_m\le \frac{A}{m}. 
\ee
\end{Lemma}
\begin{proof} The proof is by induction.  Suppose we have
$$
a_{m-1}\le \frac{A}{m-1}
$$
for some $m\ge 2$.  Then, from our assumption (\ref{Le3.1}), we have
$$
a_m\le
\frac{A}{m-1}\left(1-\frac{2}{m}\right)+\frac{A}{m^2}=A\left(\frac{1}{m}-\frac{1}{(m-1)m}+
\frac{1}{m^2}\right)
 \le \frac{A}{m}.  
$$
\end{proof} 

The following Lemma \ref{LeL3} is from \cite{VT68} (see also \cite{VTbook}, p.102). 

\begin{Lemma}\label{LeL3} Let three positive numbers $\al < \ga \le 1$, $A > 1$ be given and let a 
sequence of positive numbers $1\ge a_1 \ge a_2 \ge \dots $ satisfy the condition:
If, for some $\nu \in \bbN$ we have
$$
a_\nu \ge A\nu^{-\al},
$$
then
\begin{equation}\label{Le24.3}
a_{\nu + 1} \le a_\nu (1- \ga/\nu). 
\end{equation}
Then there exists $B = B(A,\al , \ga )$ such that for all $n=1,2,\dots $ we have
$$
a_n \le Bn^{-\al} .
$$
\end{Lemma}
\begin{proof} We have $a_1 \le 1 < A$ which implies that the set
$$
V:= \{\nu : a_\nu \ge A\nu^{-\al} \}
$$
does not contain $\nu =1$. We now prove  that for any segment $[n,n+k] \subset V$ we have
$k \le C(\al,\ga)n$. Indeed, let $n \ge 2$ be such that $n-1 \notin V$, which means
\begin{equation}\label{Le24.4}
a_{n-1} < A(n-1)^{-\al} ,  
\end{equation}
and $[n,n+k] \subset V$, which in turn means 
\begin{equation}\label{Le24.5}
a_{n+j} \ge A(n+j)^{-\al}, \quad j = 0,1,\dots , k .  
\end{equation}
Then by the condition (\ref{Le24.3}) of the lemma we get
\begin{equation}\label{Le24.6}
a_{n+k} \le a_n \prod ^{n+k-1}_{\nu =n} (1-\ga/\nu) \le a_{n-1} \prod ^{n+k-1}_{\nu =n} (1-\ga/\nu) . 
\end{equation}
Combining (\ref{Le24.4}) -- (\ref{Le24.6}) we obtain
\begin{equation}\label{Le24.7}
(n+k)^{-\al} \le (n-1)^{-\al}  \prod ^{n+k-1}_{\nu =n} (1-\ga/\nu).  
\end{equation}
Taking logarithms and using the inequalities 
$$
\ln (1-x) \le -x, \quad x \in [0,1) ;
$$
$$
\sum^{m-1}_{\nu=n} \nu^{-1} \ge \int^m_n x^{-1} dx = \ln (m/n) ,
$$
we get from (\ref{Le24.7})
$$
-\al \ln \frac{n+k}{n-1} \le \sum^{n+k-1}_{\nu=n} \ln (1-\ga/\nu) \le - \sum^{n+k-1}_{\nu=n} \ga/\nu
\le -\ga \ln \frac {n+k}{n} .
$$
Hence
$$
(\ga - \al) \ln (n+k) \le (\ga - \al)\ln n + \al \ln \frac {n}{n-1} ,
$$
which implies
$$
n+k \le 2^{\frac{\al}{\ga-\al}} n
$$
and
$$
k \le C(\al,\ga) n.
$$
Let us take any $\mu \in \bbN$. If $\mu \notin V$ we have the desired inequality with $B=A$. Assume 
$\mu \in V$, and let $[n,n+k]$ be the maximal segment in $V$ containing $\mu$. Then
\begin{equation}\label{Le24.8}
a_\mu \le a_{n-1} \le A(n-1)^{-\al} = A\mu^{-\al} \biggl(\frac{n-1}{\mu}
\biggr)^{-\al} . 
\end{equation}
Using the inequality $k\le C(\al,\ga)n$ proved above we get
\begin{equation}\label{Le24.9}
\frac{\mu}{n-1} \le \frac{n+k}{n-1} \le C_1(\al,\ga) . 
\end{equation}
Substituting (\ref{Le24.9}) into (\ref{Le24.8}) we complete the proof of Lemma \ref{LeL3} with $B = AC_1(\al,\ga)^\al$.
\end{proof}

\begin{Lemma}\label{LeL4} Let $r>0$ be given. Assume that a sequence $\{a_n\}_{n=1}^\infty$ is such that $a_1\le A$ and 
\be\label{Le10} 
a_{n+1} \le a_n\left(1-\left(\frac{a_n}{A}\right)^{1/r}\right),\quad n=1,2,\dots.
\ee
Then in the case $r\in (0,1]$ we have
\be\label{Le11}
a_m \le Am^{-r}
\ee
and in the case $r\in [1,\infty)$ we have
\be\label{Le12}
a_m \le Ar^rm^{-r}.
\ee
\end{Lemma} 
\begin{proof} We rewrite (\ref{Le10}) in the form
\be\label{Le13} 
\left(\frac{a_{n+1}}{A}\right)^{1/r} \le \left(\frac{a_{n}}{A}\right)^{1/r}\left(1-\left(\frac{a_n}{A}\right)^{1/r}\right)^{1/r},\quad n=1,2,\dots.
\ee
First, consider the case $r\in (0,1]$. Then for any $x\in [0,1]$ we have $x^{1/r}\le x$ and, therefore, (\ref{Le13}) implies
\be\label{Le14} 
\left(\frac{a_{n+1}}{A}\right)^{1/r} \le \left(\frac{a_{n}}{A}\right)^{1/r}\left(1-\left(\frac{a_n}{A}\right)^{1/r}\right),\quad n=1,2,\dots.
\ee
By Lemma \ref{LeL1} with $C_1=C_2=1$ we obtain
$$
\left(\frac{a_{m}}{A}\right)^{1/r} \le \frac{1}{m},\quad \text{and} \quad a_m \le Am^{-r}.
$$
Second, consider the case $r\in [1,\infty)$. Using the inequality
$$
(1-x)^\al \le 1-\al x,\qquad x\in [0,1],\quad \al\in (0,1],
$$
we get from (\ref{Le13})
\be\label{Le15} 
\left(\frac{a_{n+1}}{A}\right)^{1/r} \le \left(\frac{a_{n}}{A}\right)^{1/r}\left(1-\frac{1}{r}\left(\frac{a_n}{A}\right)^{1/r}\right),\quad n=1,2,\dots.
\ee
By Lemma \ref{LeL1} with $C_1=1$ and $C_2=1/r$ we obtain
$$
\left(\frac{a_{m}}{A}\right)^{1/r} \le \frac{r}{m},\quad \text{and} \quad a_m \le Ar^rm^{-r}.
$$
The proof is complete.
\end{proof}

\begin{Lemma}\label{LeL5} Let $\ff\,:\, [0,1]\to[0,1]$ be an increasing convex function with $\ff(0)=0$. Assume that a sequence $A\ge a_1\ge \cdots \ge 0$ satisfies the inequalities
\be\label{Le16}
a_{n+1} \le a_n(1-\ff(a_n/A)),\quad n=1,2,\dots .
\ee
Then 
$$
a_m \le A\ff^{-1}(1/m).
$$
\end{Lemma}
\begin{proof} Using our assumption that $\ff$ is increasing, we obtain from (\ref{Le16})
\be\label{Le17}
\ff(a_{n+1}/A) \le \ff((a_n/A)(1-\ff(a_n/A))),\quad n=1,2,\dots .
\ee
The convexity assumption implies that for any $x\in [0,1]$ and $\al \in [0,1]$ we have
$$
\ff(x(1-\al))\le \al \ff(0) +(1-\al)\ff(x) = (1-\al)\ff(x).
$$
Thus, (\ref{Le17}) gives
$$
\ff(a_{n+1}/A) \le \ff((a_n/A))(1-\ff(a_n/A)),\quad n=1,2,\dots ,
$$
and by Lemma \ref{LeL1} we obtain
$$
\ff(a_m/A) \le 1/m \quad \text{and} \quad a_m \le A\ff^{-1}(1/m).
$$
\end{proof}

\begin{Lemma}\label{LeL6} Let $\ff $ be as in Lemma \ref{LeL5} with the following additional properties. It is differentiable and there is a $\bt>0$ such that for all $x\in (0,1]$ and $\theta \in (0,x]$ we have
\be\label{Le18}
x\ff'(\theta) \ge \bt\ff(x).
\ee
 Assume that a sequence $A\ge a_1\ge \cdots \ge 0$ satisfies the inequalities (\ref{Le16}).
Then 
$$
a_m \le A\ff^{-1}(1/(\bt m)).
$$
\end{Lemma}
\begin{proof} By the Mean Value Theorem (Lagrange Theorem) we have
$$
\ff(x) -\ff(x(1-\al)) = \ff'(\theta)\al x,\quad x(1-\al) <\theta < x.
$$
Thus, by (\ref{Le18}) 
$$
\ff(x(1-\al)) \le \ff(x) - \al\bt \ff(x) = \ff(x)(1-\al\bt).
$$
Setting $x:=a_n/A$ and $\al := \ff(a_n/A)$, we obtain from (\ref{Le17})
$$
\ff(a_{n+1}/A) \le \ff((a_n/A))(1-\bt\ff(a_n/A)),\quad n=1,2,\dots ,
$$
and by Lemma \ref{LeL1} we obtain
$$
\ff(a_m/A) \le 1/(\bt m) \quad \text{and} \quad a_m \le A\ff^{-1}(1/(\bt m)).
$$
\end{proof}

The following Lemma \ref{LeL8} is from \cite{DT}.
\begin{Lemma}\label{LeL8} If a nonnegative sequence $a_0,a_1,\dots,a_N$ satisfies  
\begin{equation}\label{Le3.5}
a_m\le a_{m-1} +\inf_{0\le \la\le 1}(-\la va_{m-1}+B\la^q)+\de, \quad B>0,\quad \de\in(0,1],\quad 0<v\le 1,
\end{equation}
for $m\le N:=[\de^{-1/q}]$, $q\in (1,2]$, then
\be\label{Le3.2c}
a_m\le C(q,v,B,a_0) m^{1-q},\quad m\le N, 
\ee
with $C(q,v,B,a_0) \le C'(q,B,a_0)v^{-q}.$
\end{Lemma}
\begin{proof}
By taking $\lambda=0$,  we derive from (\ref{Le3.5})  that 
\begin{equation}\label{Le3.6}
a_m\le a_{m-1}+\de,\quad m\le N.
\end{equation}
  Therefore,  for all $m\le N$ we have
$$
a_m\le a_0+N\delta \le a_0+1, \quad 0\le m\le N.
$$
Now fix any value of $m\in [1,N]$ and define $\la_1:= \left(\frac{va_{m-1}}{2B}\right)^{\frac{1}{q-1}}$, so that 
\begin{equation}\label{Le3.7}
 \la_1 v a_{m-1}=  2B\la_1^q .
\end{equation}
If $\la_1\le 1$ then 
$$
\inf_{0\le \la\le 1}(-\la va_{m-1}+ B\la^q)\le -\la_1 va_{m-1}+ B\la_1^q 
$$
$$
=-\frac{1}{2}\la_1va_{m-1} = -C_1(q,B)v^pa_{m-1}^p,\quad p:=\frac{q}{q-1}.
$$
If $\la_1> 1$ then for all $\la\le\la_1$ we have $\la v a_{m-1}>  2B\la^q$ and specifying $\la=1$ we get
$$
\inf_{0\le \la\le 1}(-\la va_{m-1}+B\la^q)\le-\frac{1}{2}va_{m-1} 
$$
$$
\le -\frac{1}{2}va_{m-1}^p(a_0+1)^{1-p} = -C_1(q,a_0)va_{m-1}^p.
$$
Thus, in any case, setting $C_2:=C_2(q,v,B,a_0):=\min(C_1(q,B)v^p,C_1(q,a_0)v)$ we obtain
from (\ref{Le3.5})
\begin{equation}\label{Le3.8}
a_m\le a_{m-1}- C_2a_{m-1}^p +\de,\quad C_2 \ge C_2'(q,B,a_0)v^p, 
\end{equation}
holds for all $1\le m\le N$.

Now to establish  (\ref{Le3.2c}), we  let $n\in [1,N]$ be the smallest integer such that 
\be\label{smallest}
C_2a_{n-1}^p \le 2\de.
\ee
If there is no such $n$,
we set $n=N$.  In view of (\ref{Le3.8}),
we have
\be
\label{have1}
a_m\le a_{m-1}-(C_2/2)a_{m-1}^p, \quad 1\le m\le n.
\ee
If we modify the sequence $a_m$ by defining it to be zero if $m>n$, then this modified sequence satisfies (\ref{have1})
for all $m$ and Lemma \ref{LeL4} with $r=q-1$ gives
\be\label{gives1}
a_m\le C_3 m^{1-q},\quad 1\le m\le n, \quad C_3\le C_3'(q,B,a_0)v^{-q}.
\ee

 If $n=N$, we have finished the proof.
If $n<N$,  then, by (\ref{Le3.6}), we obtain for $m\in[n,N]$
$$
a_m\le a_{n-1} +(m-n+1)\de \le a_{n-1}+N\delta\le a_{n-1}+ NN^{-q}\le \left(\frac{2\delta}{C_2}\right)^{1/p}+C_4N^{1-q},
$$
where we have used the definition of $n$.
Since $\delta^{1/p}\le N^{-q/p}=N^{-q+1}$,  we have 
$$
a_m \le C_5N^{1-q}\le C_5m^{1-q},\quad n\le m\le N,\quad C_5\le C_5' v^{-1},
$$
where $C_5'$ depends only on $q,B,a_0$.
This completes the proof of the lemma.  
\end{proof}

The following Lemma \ref{LeL9} is from \cite{VT148}. 
\begin{Lemma}\label{LeL9} Let $\rho(u)$ be a nonnegative convex on $[0,1]$ function with the property 
$\rho(u)/u\to0$ as $u\to 0$. Assume that a nonnegative sequence $\{\de_k\}$ is such that $\de_k\to0$ as $k\to\infty$. Suppose that a nonnegative sequence $\{a_k\}_{k=0}^\infty$ satisfies  the inequalities
\be\label{Le27}
a_m\le a_{m-1} +\inf_{0\le\la\le1}(-\la va_{m-1}+B\rho(\la)) +\de_{m-1},\quad m=1,2,\dots,
\ee
with positive numbers $v$ and $B$. Then
$$
\lim_{m\to\infty} a_m =0.
$$
\end{Lemma}
\begin{proof} We carry out the proof under assumption that $\rho(u)>0$ for $u>0$. Otherwise, if $\rho(u)=0$ for $u\in (0,u_0]$ then
$$
a_m\le (a_0+\de_0)(1-u_0v)^{m-1}+\de_1(1-u_0v)^{m-2}+\cdots+\de_{m-1} \to 0\quad \text{as}\quad m\to\infty.
$$
Denote
$$
\bt_{m-1}:=-\inf_{0\le\la\le1}(-\la va_{m-1}+B\rho(\la)).
$$
It is clear that $\bt_{m-1}\ge 0$. We divide the set of natural numbers into two sets:
$$
\cM_1:=\{m:\bt_{m-1}\le 2\de_{m-1}\};\qquad \cM_2:=\{m:\bt_{m-1}>2\de_{m-1}\}.
$$
The set $\cM_1$ can be either finite or infinite. First, consider the case of infinite $\cM_1$. 
Let
$$
\cM_1=\{m_k\}_{k=1}^\infty,\quad m_1<m_2<\dots .
$$
For any $m\in\cM_2$ we have
$$
a_m\le a_{m-1} -\bt_{m-1}+\de_{m-1} < a_{m-1}-\de_{m-1} \le a_{m-1}. 
$$
Thus, the sequence $\{a_m\}$ is monotone decreasing on $(m_{k-1},m_k)$. Also, we have
\begin{equation}\label{Led3.1}
a_{m_{k-1}} \le a_{m_{k-1}-1}+\de_{m_{k-1}-1}.
\end{equation}
It is clear from (\ref{Led3.1}), monotonicity of $\{a_m\}$ on $(m_{k-1},m_k)$ and the property 
$\de_k\to0$ as $k\to\infty$ that it is sufficient to prove that 
$$
\lim_{k\to\infty} a_{m_{k}-1}=0.
$$
For $m\in\cM_1$ we have $\bt_{m-1}\le 2\de_{m-1}$. Let $\la_1(m)$ be a nonzero solution to the equation 
$$
\la v a_{m-1} = 2B\rho(\la) \quad \text{or}\quad \frac{\rho(\la)}{\la} = \frac{va_{m-1}}{2B}.
$$
If $\la_1(m)\le 1$ then
$$
-\bt_{m-1} \le -\la_1(m)va_{m-1}+B\rho(\la_1(m))=-B\rho(\la_1(m)).
$$
Thus, 
$$
B\rho(\la_1(m)) \le \bt_{m-1} \le 2\de_{m-1}.
$$
Therefore, using that $\rho(u)>0$ for $u>0$, we obtain $\la_1(m)\to 0$ as $m\to\infty$. Next,
$$
a_{m-1} = \frac{2B}{v}\frac{\rho(\la_1(m))}{\la_1(m)} \to 0 \quad \text{as}\quad m\to\infty.
$$
If $\la_1(m)>1$ then by monotonicity of $\rho(u)/u$ for all $\la\le \la_1(m)$ we have $\la va_{m-1}\ge 2B\rho(\la)$. Specifying $\la=1$ we get
$$
-\bt_{m-1} \le -\frac{1}{2}va_{m-1}.
$$
Therefore, $a_{m-1}\le (2/v)\bt_{m-1} \le (4/v)\de_{m-1} \to 0$ as $m\to\infty$.

Second, consider the case of finite $\cM_1$. Then there exists $m_0$ such that for all $m\ge m_0$ we have
\begin{equation}\label{Led3.2}
a_m\le a_{m-1} -\bt_{m-1} +\de_{m-1} \le a_{m-1}-\frac{1}{2} \bt_{m-1}.
\end{equation} 
The sequence $\{a_m\}_{m\ge m_0}$ is monotone decreasing and therefore it has a limit $\alpha\ge 0$. We prove that $\alpha=0$ by contradiction. Suppose $\alpha>0$. Then $a_{m-1}\ge \alpha$ for $m\ge m_0$. It is clear that for $m\ge m_0$ we have $\bt_{m-1} \ge c_0>0$. This together with (\ref{Led3.2}) contradict to our assumption that $a_m\ge\alpha$, $m\ge m_0$. 
\end{proof}

We now prove a lemma from \cite{VT148} that gives the rate of decay of a sequence satisfying (\ref{Le27}) with a special $\rho(u) = \ga u^q$.
\begin{Lemma}\label{LeL10} Suppose a nonnegative sequence $a_0,a_1,\dots$ satisfies the inequalities for $m=1,2,\dots$
\begin{equation}\label{Le3.5a}
a_m\le a_{m-1} +\inf_{0\le \la\le 1}(-\la va_{m-1}+B\la^q)+\de_{m-1},  \quad \de_{m-1}\le cm^{-q},
\end{equation}
where   $q\in (1,2]$, $v\in(0,1]$, and $B>0$. Then
$$
a_m\le C(q,v,B,a_0,c) m^{1-q},\qquad C(q,v,B,a_0,c) \le C'(q,B,a_0,c)v^{-q} . 
$$
\end{Lemma}
\begin{proof}
In particular, (\ref{Le3.5a}) implies that 
\begin{equation}\label{Le3.6a}
a_m\le a_{m-1}+\de_{m-1}.
\end{equation}
  Then for all $m$ we have
$$
a_m\le a_0+C_1(q,c),\quad C_1(q,c):= c\sum_{k=0}^\infty (k+1)^{-q}.
$$
Denote $\la_1$ a nonzero solution of the equation
\begin{equation}\label{Le3.7a}
 \la v a_{m-1}=  2B\la^q, \quad  \la_1= \left(\frac{va_{m-1}}{2B}\right)^{\frac{1}{q-1}}.
\end{equation}
If $\la_1\le 1$ then 
$$
\inf_{0\le \la\le 1}(-\la va_{m-1}+ B\la^q)\le -\la_1 va_{m-1}+ B\la_1^q 
$$
$$
=-\frac{1}{2}\la_1va_{m-1} = -C_1'(q,B)v^pa_{m-1}^p,\quad p:=\frac{q}{q-1}.
$$
If $\la_1> 1$ then for all $\la\le\la_1$ we have $\la v a_{m-1}\ge  2B\la^q$ and specifying $\la=1$ we get
$$
\inf_{0\le \la\le 1}(-\la va_{m-1}+B\la^q)\le-\frac{1}{2}va_{m-1} 
$$
$$
\le -\frac{1}{2}va_{m-1}^p(a_0+C_1(q,c))^{1-p} = -C_1(q,a_0,c)va_{m-1}^p.
$$
Setting $C_2:=C_2(q,v,B,a_0,c):=\min(C_1'(q,B)v^p,C_1(q,a_0,c)v)$ we obtain
from (\ref{Le3.5a})
\begin{equation}\label{Le3.8a}
a_m\le a_{m-1}- C_2a_{m-1}^p +\de_{m-1},\quad C_2\ge C_2'(q,B,a_0,c)v^p.
\end{equation}

We now need one more technical lemma (see \cite{VT148}). 
\begin{Lemma}\label{LeL11} Let $q\in (1,2]$, $p:=\frac{q}{q-1}$. Assume that a sequence 
$\{\de_k\}_{k=0}^\infty$ is such that $\de_k\ge 0$ and $\de_k\le c(k+1)^{-q}$. Suppose a nonnegative sequence $\{a_k\}_{k=0}^\infty$ satisfies the inequalities
\begin{equation}\label{Led3.3}
a_m\le a_{m-1}-wa_{m-1}^p +\de_{m-1},\qquad m=1,2,\dots,
\end{equation}
with a positive number $w\in(0,1]$. Then
$$
a_m\le C(q,c,w,a_0)m^{1-q},\quad m=1,2,\dots,\quad C(q,c,w,a_0)\le C'(q,c,a_0)w^{-\frac{1}{p-1}}.
$$
\end{Lemma}
\begin{proof} Lemma \ref{LeL11} is a simple corollary of the following known lemma, which is a more general version of Lemma \ref{LeL3}.
\begin{Lemma}\label{LeL12}  Let three positive numbers $\al < \beta $, $A$   be given and let a sequence 
$\{a_n\}_{n=0}^\infty$ have the following properties:  $ a_0<A$ and  we have for all $n\ge 1$
 \begin{equation}\label{Le3.20}
 a_n\le a_{n-1}+An^{-\al};  
\end{equation}
 if for some $\nu $ we have
$$
a_\nu \ge A\nu^{-\al}
$$
then
\begin{equation}\label{Le3.21}
a_{\nu + 1} \le a_\nu (1- \beta/\nu). 
\end{equation}
Then there exists a constant $C=C(\al , \beta)$ such that for all $n=1,2,\dots $ we have
$$
a_n \le C A n^{-\al} .
$$
 \end{Lemma}
\begin{Remark}\label{LeR1} If conditions (\ref{Le3.20}) and (\ref{Le3.21}) are satisfied for $n\le N$ and $\nu\le N$ then the statement of Lemma \ref{LeL12} holds for $n\le N$. 
\end{Remark} 
 Suppose that
 $$
 a_\nu \ge A\nu^{1-q}.
 $$
 Then by (\ref{Led3.3})
 $$
 a_{\nu+1}\le a_\nu(1-wa_\nu^{p-1})+c\nu^{-q} \le a_\nu(1-wa_\nu^{p-1}+(c/A)/\nu).
 $$
 Making $A$ large enough $A=C(a_0,c)w^{-\frac{1}{p-1}}$ we get $a_0<A$ and
 $$
 -wA^{p-1}+c/A \le -2.
 $$
 We now apply Lemma \ref{LeL12}, which completes the proof of Lemma \ref{LeL11} and hence of Lemma \ref{LeL10} as well.
 \end{proof}
\end{proof}

 \section{Some results on greedy algorithms}
 
 In this section we present some typical proofs for convergence and rate of convergence of some greedy algorithms. 
 We present these proofs in the general situation, when we are in a complex Banach space. We follow the presentation from \cite{VT209}. This presentation goes along the lines of proofs in the case of real Banach spaces with some necessary modifications. The reader can find the real case discussion in the book \cite{VTbook} and detailed discussion of the complex case in \cite{DGHKT} and \cite{VT209}. 
   
 \subsection{Definitions and Lemmas}


We note that from the definition of modulus of smoothness we get the following inequality (in the real case see, for instance, \cite{VTbook}, p.336, and in the complex case see \cite{VT209}).
\begin{Lemma}\label{LL0} Let $x\neq0$. Then
\be\label{In1}
0\le \|x+uy\|-\|x\|-Re(uF_x(y))\le 2\|x\|\rho(u\|y\|/\|x\|). 
\ee
\end{Lemma}

 Lemma \ref{LL0} can be proved in a way similar to the proof of Lemma \ref{LL1} below. 
 
  The following  simple and well-known Lemma \ref{LL1} in the case of real Banach spaces is proved, for instance, in \cite{VTbook}, Ch.6, pp.342-343. In the case of complex Banach spaces see \cite{VT209}. 
\begin{Lemma}\label{LL1} Let $X$ be a uniformly smooth Banach space and $L$ be a finite-dimensional subspace of $X$. For any $f\in X\setminus L$   let $f_L$ denote the best approximant of $f$ from $L$. Then we have 
$$
F_{f-f_L}(\phi) =0
$$
for any $\phi \in L$.
\end{Lemma}
\begin{proof} Let us assume the contrary: there is a $\phi \in L$ such that $\|\phi\|=1$ and
$$
|F_{f-f_L}(\phi)| =\bt >0.
$$
Denote by $\nu$ the complex conjugate of $\sign(F_{f-f_L}(\phi))$, where $\sign z := z/|z|$ for $z\neq 0$. Then 
$\nu F_{f-f_L}(\phi) = |F_{f-f_L}(\phi)|$.
For any $\la\ge 0$ we have from the definition of $\rho(u)$ that 
\be\label{2.3}
\|f-f_L-\la \nu\phi\| +\|f-f_L+\la \nu\phi\| \le 2\|f-f_L\|\left(1+\rho\left(\frac{\la}{\|f-f_L\|}\right)\right).
\ee
Next
\be\label{2.4}
 \|f-f_L+\la \nu\phi\| \ge |F_{f-f_L}(f-f_L+\la\nu\phi)| =\|f-f_L\|+\la\bt.
\ee
Combining (\ref{2.3}) and (\ref{2.4}) we get
\be\label{2.5}
\|f-f_L-\la\nu\phi\|\le \|f-f_L\|\left(1-\frac{\la\bt}{\|f-f_L\|} +2\rho\left(\frac{\la}{\|f-f_L\|}\right)\right).
\ee
Taking into account that $\rho(u) =o(u)$, we find $\la'>0$ such that
$$
\left(1-\frac{\la'\bt}{\|f-f_L\|} +2\rho\left(\frac{\la'}{\|f-f_L\|}\right)\right) <1.
$$
Then (\ref{2.5}) gives
$$ 
\|f-f_L-\la'\nu\phi\| < \|f-f_L\|,
$$
which contradicts the assumption that $f_L \in L$ is the best approximant of $f$.
\end{proof}

\begin{Remark}\label{SC1} The condition $F_{f-f_L}(\phi) =0$ for any $\phi \in L$ is also a sufficient condition for $f_L\in L$ 
to be a best approximant of $f$ from $L$.
\end{Remark}
\begin{proof} Indeed, for any $g \in L$ we have
$$
\|f-f_L\| = F_{f-f_L}(f-f_L) =  F_{f-f_L}(f-g) \le \|f-g\|.
$$
\end{proof}

The following  simple and well-known Lemma \ref{LL2} in the case of real Banach spaces is proved, for instance, in \cite{VTbook}, Ch.6, p.343. In the case of complex Banach spaces see \cite{VT209}. 
\begin{Lemma}\label{LL2} For any bounded linear functional $F$ and any dictionary $\cD$, we have
$$
\|F\|_\cD:=\sup_{g\in \cD}|F(g)| = \sup_{f\in  A_1(\cD)} |F(f)|.
$$
\end{Lemma}
 The following Lemma \ref{LL3} is from \cite{VT209}. The proofs of Lemmas \ref{LL2} and \ref{LL3} are similar. We only present the proof of Lemma \ref{LL3} here. 
\begin{Lemma}\label{LL3} For any bounded linear functional $F$ and any dictionary $\cD$, we have
$$
 \sup_{g\in \cD}Re(F(g)) = \sup_{f\in  \conv(\cD)} Re(F(f)).
$$
\end{Lemma}
\begin{proof} The inequality 
$$
\sup_{g\in \cD}Re(F(g)) \le \sup_{f\in \conv(\cD)} Re(F(f))
$$
is obvious because $\cD \subset \conv(\cD)$. We prove the opposite inequality. Take any $f\in \conv(\cD)$. Then for any $\e >0$ there exist $g_1^\e,\dots,g_N^\e \in \cD$ and nonnegative numbers $a_1^\e,\dots,a_N^\e$ such that $a_1^\e+\cdots+a_N^\e = 1$ and 
$$
\left\|f-\sum_{j=i}^Na_i^\e g_i^\e\right\| \le \e.
$$
Thus
$$
Re(F(f)) \le \|F\|\e + Re(F\left(\sum_{i=1}^Na_i^\e g_i^\e\right)) \le \e \|F\| +\sup_{g\in \cD} Re(F(g)),
$$
which proves Lemma \ref{LL3}.
\end{proof}
 
 In the case of real Banach spaces the algorithms WCGA, WGAFR, GAWR, {\bf IA}($\e$) and the corresponding results on their convergence  and rate of convergence are known (see, for instance, \cite{VTbook}, Ch.6). We showed in \cite{VT209}  how to modify those algorithms to the case of complex Banach spaces.   Note, that results on the WCGA in the case of complex Banach spaces can be found in \cite{DGHKT}. For illustration we only present some results from \cite{VT209}. 
 Namely, we discuss the Incremental Algorithms (see Subsection 5.2 above). 
 
\subsection{Incremental Algorithm}
\label{MSub3}

Let $\e=\{\e_n\}_{n=1}^\infty $, $\e_n> 0$, $n=1,2,\dots$. The Incremental Algorithms {\bf IA}($\e$) (see Subsection 5.2 above) was introduced and 
studied in \cite{VT94} (see also, \cite{VTbook}, pp.361--363) in the case of real Banach spaces.  
  We now modify the definition of {\bf IA($\e$)} to make it suitable for the complex Banach spaces (see \cite{VT209}). 
 
 {\bf Incremental Algorithm (complex) with schedule $\e$ (IAc($\e$)).} 
Let $f\in A_1(\cD)$. Denote $f_0^{c,\e}:= f$ and $G_0^{c,\e} :=0$. Then, for each $m\ge 1$ we have the following inductive definition.

(1) $\ff_m^{c,\e} \in \cD^\circ$ is any element satisfying
$$
Re(F_{f_{m-1}^{c,\e}}(\ff_m^{c,\e}-f)) \ge -\e_m.
$$
Denote by $\nu_m$ the complex conjugate of $\sign F_{f_{m-1}^{c,\e}}(\ff_m^{c,\e})$, where $\sign z := z/|z|$ for $z\neq 0$.

(2) Define
$$
G_m^{c,\e}:= (1-1/m)G_{m-1}^{c,\e} +\nu_m\ff_m^{c,\e}/m.
$$

(3) Let
$$
f_m^{c,\e} := f- G_m^{c,\e}.
$$

It follows from the definition of the {\bf IAc($\e$)} that 
\be\label{M15}
G_m^{c,\e} = \frac{1}{m} \sum_{j=1}^m \nu_j\ff_j^{c,\e},\quad |\nu_j|=1, \quad j=1,2,\dots. 
\ee

Note that by the definition of $\nu_m$ we have $F_{f_{m-1}^{c,\e}}(\nu_m\ff_m^{c,\e}) = |F_{f_{m-1}^{c,\e}}(\ff_m^{c,\e})|$ 
and, therefore, by Lemma \ref{LL2}
$$
\sup_{\phi \in \cD^\circ}Re(F_{f_{m-1}^{c,\e}}(\phi))= \sup_{g \in \cD}|F_{f_{m-1}^{c,\e}}(g)| 
$$
$$
= \sup_{\phi \in A_1(\cD)}|F_{f_{m-1}^{c,\e}}(\phi)| \ge |F_{f_{m-1}^{c,\e}}(f)| \ge Re(F_{f_{m-1}^{c,\e}}(f)).
$$
This means that we can always run the IAc($\e$) for $f\in A_1(\cD)$.
 
\begin{Theorem}\label{MT5} Let $X$ be a uniformly smooth complex Banach space with  modulus of smoothness $\rho(u)\le \gamma u^q$, $1<q\le 2$. Define
$$
\e_n := K_1\gamma ^{1/q}n^{-1/p},\qquad p=\frac{q}{q-1},\quad n=1,2,\dots .
$$
Then, for any $f\in  A_1(\cD)$ we have
$$
\|f_m^{c,\e}\| \le C(K_1) \gamma^{1/q}m^{-1/p},\qquad m=1,2\dots.
$$
\end{Theorem}
\begin{proof} We will use the abbreviated notation $f_m:=f_m^{c,\e}$, $\ff_m:=\nu_m\ff_m^{c,\e}$, $G_m:=G_m^{c,\e}$. Writing
$$
f_m = f_{m-1}-(\ff_m-G_{m-1})/m,
$$
we  immediately obtain the trivial estimate
\be\label{M6.6}
\|f_m\|\le \|f_{m-1}\| +2/m.  
\ee
Since
\be\label{M6.7}
f_m=\left(1-\frac{1}{m}\right)f_{m-1}-\frac{\ff_m-f}{m} 
 = \left(1-\frac{1}{m}\right)\left(f_{m-1}-\frac{\ff_m-f}{m-1}\right) 
\ee
we obtain by Lemma \ref{LL0} with $y= \ff_m-f$ and $u= -1/(m-1)$
$$
\left\|f_{m-1} -\frac{\ff_m-f}{m-1}\right\| 
$$
\be\label{M6.8}
\le \|f_{m-1}\|\left(1+2\rho\left(\frac{2}{(m-1)\|f_{m-1}\|}\right)\right) +\frac{\e_m}{m-1}.  
\ee
 Using the definition of $\e_m$ and the assumption $\rho(u) \le \gamma u^q$, we make the following observation. There exists a constant $C(K_1)$ such that, if
\be\label{M6.9}
\|f_{m-1}\|\ge C(K_1) \gamma ^{1/q}(m-1)^{-1/p}  
\ee
then
\be\label{M6.10}
2\rho(2((m-1)\|f_{m-1}\|)^{-1}) + \e_m ((m-1)\|f_{m-1}\|)^{-1} \le 1/(4m),  
\ee
and therefore, by (\ref{M6.7}) and (\ref{M6.8})
\be\label{M6.11}
\|f_m\| \le \left(1-\frac{3}{4m}\right)\|f_{m-1}\|.  
\ee

We now need the following   technical lemma (see, for instance, \cite{VTbook}, p.357, and Lemma \ref{LeL12} above).
\begin{Lemma}\label{ML4}  Let a sequence $\{a_n\}_{n=1}^\infty$ have the following property. For given positive
 numbers $\alpha < \ga \le 1$, $A > a_1$,  we have for all $n\ge2$
 \be\label{M5.6}
 a_n\le a_{n-1}+A(n-1)^{-\alpha}.  
 \ee
 If for some $v \ge2$ we have
$$
a_v \ge Av^{-\alpha}
$$
then
 \be\label{M5.7}
a_{v + 1} \le a_v (1- \ga/v). 
\ee
Then there exists a constant $C(\alpha , \ga )$ such that for all $n=1,2,\dots $ we have
$$
a_n \le C(\alpha,\ga)An^{-\alpha} .
$$
 \end{Lemma}

Taking into account (\ref{M6.6}) we apply Lemma \ref{ML4}  to the sequence $a_n=\|f_n\|$, $n=1,2,\dots$ with $\alpha=1/p$, $\beta =3/4$ and complete the proof of Theorem \ref{MT5}.
\end{proof}

We now consider one more modification  of {\bf IA($\e$)} suitable for the complex Banach spaces and providing the convex combination of the dictionary elements. In the case of real Banach spaces it coincides with the {\bf IA($\e$)}.
 
 {\bf Incremental Algorithm (complex and convex) with schedule $\e$ (IAcc($\e$)).} 
Let $f\in \conv(\cD)$. Denote $f_0^{cc,\e}:= f$ and $G_0^{cc,\e} :=0$. Then, for each $m\ge 1$ we have the following inductive definition.

(1) $\ff_m^{cc,\e} \in \cD$ is any element satisfying
$$
Re(F_{f_{m-1}^{cc,\e}}(\ff_m^{cc,\e}-f)) \ge -\e_m.
$$
 
(2) Define
$$
G_m^{cc,\e}:= (1-1/m)G_{m-1}^{cc,\e} + \ff_m^{cc,\e}/m.
$$

(3) Let
$$
f_m^{cc,\e} := f- G_m^{cc,\e}.
$$

It follows from the definition of the {\bf IAcc($\e$)} that 
\be\label{M15}
G_m^{cc,\e} = \frac{1}{m} \sum_{j=1}^m \ff_j^{cc,\e} . 
\ee

Note that by   Lemma \ref{LL3}
$$
 \sup_{g \in \cD}Re(F_{f_{m-1}^{cc,\e}}(g))
= \sup_{\phi \in \conv(\cD)}Re(F_{f_{m-1}^{cc,\e}}(\phi))   \ge Re(F_{f_{m-1}^{cc,\e}}(f)).
$$
This means that we can always run the {\bf IAcc($\e$)} for $f\in \conv(\cD)$.
 
In the same way as Theorem \ref{MT5} was proved above one can prove the following Theorem \ref{MT6}.

\begin{Theorem}\label{MT6} Let $X$ be a uniformly smooth Banach space with  modulus of smoothness $\rho(u)\le \gamma u^q$, $1<q\le 2$. Define
$$
\e_n := K_1\gamma ^{1/q}n^{-1/p},\qquad p=\frac{q}{q-1},\quad n=1,2,\dots .
$$
Then, for any $f\in  \conv(\cD)$ we have
$$
\|f_m^{cc,\e}\| \le C(K_1) \gamma^{1/q}m^{-1/p},\qquad m=1,2\dots.
$$
\end{Theorem}
 
 \newpage
 \medskip
 {\bf \Large Chapter V }
  \medskip
  
  \section{Some open problems}
  \label{OP}

In this chapter we list some known open problems in the theory of greedy approximation. We begin with the most famous problem on the rate of convergence of the PGA.  The first upper bound, obtained in \cite{DT1} (see, for instance, Theorem 2.18 from \cite{VTbook}), states that for a general dictionary $\cD$ the Pure Greedy Algorithm provides the estimate: For any $f \in A_1(\cD)$
\be\label{OP1}
\|f-G_m(f,\cD)\| \le   m^{-1/6}. 
\ee
 The above estimate was improved a little in \cite{KT2} to
\be\label{OP2}
\|f-G_m(f,\cD)\| \le 4 m^{-11/62}. 
\ee
  We now discuss recent progress on the following open problem (see \cite{VT83}, p.65, Open Problem 3.1). This problem is a central theoretical problem in greedy approximation in Hilbert spaces.  
 
{\bf Open problem 1.} Find the order of decay of the sequence
$$
\gamma(m):=\sup_{f \in A_1(\cD),\cD,\{G_m\}} \|f-G_m(f,\cD)\|,
$$
where the supremum is taken over all dictionaries $\cD$, all elements
$f\in A_1(\cD) $ and all possible choices of realizations $\{G_m(f,\cD)\}$ of the PGA. In other words, we want to find a sequence $\{\ff_m\}_{m=1}^\infty$ such that there exist 
two absolute constants $0<C_1\le C_2<\infty$ with the property
\be\label{OPp}
C_1 \ff_m \le \gamma(m) \le C_2 \ff_m,\quad m=1,2,\dots .
\ee
We begin with the upper bounds. 
 The   upper bounds (\ref{OP1}) and (\ref{OP2})   were improved in \cite{Si}.   The author  proved the estimate
$$
\gamma(m) \le Cm^{-\frac{s}{2(2+s)}},
$$
where $s$ is a solution from $[1,1.5]$ of the equation
$$
(1+x)^{\frac{1}{2+x}}\left(\frac{2+x}{1+x}\right)-\frac{1+x}{x}=0.
$$
Numerical calculations of $s$ (see \cite{Si}) give
$$
  \frac{s}{2(2+s)}=0.182\dots >11/62.
  $$  
 The technique used in \cite{Si} is a further development of a method from \cite{KT2}.
 
 We now discuss some progress in the lower estimates.    The  estimate
 $$
 \gamma(m) \ge Cm^{-0.27},  
 $$
 with a positive constant $C$, was proved in  \cite{LTe2}. For previous lower estimates see \cite{VT83}, p.59.   The author of \cite{Li4}, using the technique from \cite{LTe2}, proved the following lower estimate
 \be\label{OP3}
\gamma(m) \ge Cm^{-0.1898}.  
\ee
Very recently, substantial progress has been achieved in \cite{KS}. The authors proved that for any $\de>0$ 
we have
 \be\label{OP4}
 \gamma(m) m^{\frac{s}{2(2+s)}+\de} \ge C(\de)>0,\quad m=1,2,\dots .
 \ee
The technique used in \cite{KS} is a further development of a method from \cite{LTe2} and \cite{Li4}.
 
In  Subsections \ref{OP1} and \ref{OP2} we discuss two fundamental properties of a greedy algorithm GA -- convergence and rate of convergence (see Sections \ref{con} and \ref{rc} above). We would like to formulate those properties in terms of 
some characteristics of a given Banach space and parameters of a greedy algorithm GA. We consider uniformly smooth Banach spaces. It turns out that a very natural characteristic of a Banach space $X$, which plays an important role in convergence and rate of convergence properties of a greedy algorithm GA, is its modulus of smoothness $\rho(u,X)$.
We pay special attention to the collection of Banach spaces $\cX(\ga,q)$, defined in Section \ref{rc}: For  fixed $1<q\le 2$ and $\ga >0$ define
$$
\cX(\ga,q) := \{X\,:\, \rho(u,X) \le \gamma u^q \}.
$$
In particular, this collection is of high interest because the $L_p$ spaces are from this collection (see (\ref{Lprho})).

\subsection{Convergence.} \label{OP1}
Certainly, ideally, we would like to have a criterion for convergence. For illustration we now formulate some  open problems from \cite{VT83}, p. 73.

{\bf Open problem 2.} (\cite{VT83}, 4.1). Characterize Banach spaces $X$ such that the $X$-Greedy Algorithm converges for all dictionaries $\cD$ and each element $f$.

{\bf Open problem 3.} (\cite{VT83}, 4.2). Characterize Banach spaces $X$ such that the Dual Greedy Algorithm converges for all dictionaries $\cD$ and each element $f$.

{\bf Open problem 4.} (\cite{VT83}, 4.3). (Conjecture). Prove that the Dual Greedy Algorithm converges for all dictionaries $\cD$ and each element $f\in X$ in uniformly smooth Banach spaces $X$ with modulus of smoothness of fixed power type $q$, $1<q\le 2$, ($\rho(u,X) \le \gamma u^q$).

{\bf Open problem 5.} (\cite{VT83}, 4.4). Find the necessary and sufficient conditions on a weakness sequence $\tau$ to guarantee convergence of the Weak Dual Greedy Algorithm in uniformly smooth Banach spaces $X$ with modulus of smoothness of fixed power type $q$, $1<q\le 2$, ($\rho(u,X) \le \gamma u^q$) for all dictionaries $\cD$ and each element $f\in X$.
 
 The above Open problems 2--4 are formulated for two greedy algorithms -- XGA and DGA. Certainly, similar problems are of interest for other greedy algorithms. Here are some open problems from \cite{VT118}.  
 
 {\bf Open problem 6.} (\cite{VT118}, Open Problem 1). Does the XGA converge for all dictionaries $\cD$ and each element $f\in X$ in uniformly smooth Banach spaces $X$ with modulus of smoothness of fixed power type $q$, $1<q\le 2$, ($\rho(u,X) \le \gamma u^q$)?

{\bf Open problem 7.} (\cite{VT118}, Open Problem 3). Characterize Banach spaces $X$ such that the  WDGA($t$), $t\in(0,1]$,  converges for all dictionaries $\cD$ and each element $f$.

{\bf Open problem 8.} (\cite{VT118}, Open Problem 4). (Conjecture). Prove that the WDGA($t$), $t\in(0,1]$,  converges for all dictionaries $\cD$ and each element $f\in X$ in uniformly smooth Banach spaces $X$ with modulus of smoothness of fixed power type $q$, $1<q\le 2$, ($\rho(u,X) \le \gamma u^q$).

{\bf Open problem 9.} (\cite{VT118}, Open Problem 6). Let $p\in(1,\infty)$. Find the necessary and sufficient conditions on a weakness sequence $\tau$ to guarantee convergence of the Weak Dual Greedy Algorithm in  the $L_p$ space for all dictionaries $\cD$ and each element $f\in L_p$.

{\bf Open problem 10.} (\cite{VT118}, Open Problem 7). Characterize Banach spaces $X$ such that the 
WCGA($t$), $t\in(0,1]$, converges for every dictionary $\cD$ and for every $f\in X$.

It would be interesting to understand if the answers to the Open problem 10 and the following Open problem 11 are the same. 

{\bf Open problem 11.} (\cite{VT118}, Open Problem 8). Characterize Banach spaces $X$ such that the 
 WGAFR with the weakness sequence $\tau=\{t\}$, $t\in(0,1]$, converges for every dictionary $\cD$ and for every $f\in X$.
 
{\bf Open problem 12.} (\cite{VT118}, Open Problem 9). Characterize Banach spaces $X$ such that the 
  XGAFR converges for every dictionary $\cD$ and for every $f\in X$.
  
In the above formulated open problems we discussed convergence of a greedy algorithm GA in the sense of Section \ref{con} -- convergence for any $\cD$ and each $f\in X$. Clearly, we can expect better convergence conditions for 
a given specific dictionary $\cD$. 

For illustration we formulate some open problems from \cite{VT83} (pages 78-79) in the case of bilinear dictionary
\be\label{OP5}
\Pi_p  := \{u(x)v(y)\,:\, \|u\|_{L_p}=\|v\|_{L_p}=1\},\quad 1\le p\le \infty.
\ee

{\bf Open problem 13.} (\cite{VT83}, 5.1). Find the necessary and sufficient conditions on a weakness sequence $\tau$ to guarantee convergence of the Weak Greedy Algorithm with regard to $\Pi_2$ for each $f\in L_2$.

{\bf Open problem 14.} (\cite{VT83}, 5.2). Does the $L_p$-Greedy Algorithm with respect to $\Pi_p$ converge for each $f\in L_p$, $1<p<\infty$?

{\bf Problem 15.} (\cite{VT83}, 5.3). Does the Dual Greedy Algorithm with respect to $\Pi_p$ converge for each $f\in L_p$, $1<p<\infty$?

{\bf Open problem 16.} (\cite{VT83}, 5.4).   Find the necessary and sufficient conditions on a weakness sequence $\tau$ to guarantee convergence of the Weak Dual Greedy Algorithm with respect to $\Pi_p$ for each $f\in L_p$.

{\bf Open problem 17.} (\cite{VT83}, 5.5). Find the necessary and sufficient conditions on a weakness sequence $\tau$ to guarantee convergence of the Weak Chebyshev Greedy Algorithm with respect to $\Pi_p$ for each $f\in L_p$.

We point out that the XGA and DGA are so difficult to study that only a little progress in studying the DGA was achieved and some convergence problems are still open even in the case of classical bases as dictionaries. 

{\bf Open problem 18.} (\cite{VT83}, 8.1). Does the $L_p$-Greedy Algorithm with regard to $\cT$ converge in $L_p$, $1<p<\infty$, for each $f\in L_p(\bbT)$?

{\bf Problem 19.} (\cite{VT83}, 8.2). Does the Dual Greedy Algorithm with regard to $\cT$ converge in $L_p$, $1<p<\infty$, for each $f\in L_p(\bbT)$?

{\bf Open problem 20.} (\cite{VT83}, 8.3). Does the $L_p$-Greedy Algorithm with regard to $\cH_p$ converge in $L_p$, $1<p<\infty$, for each $f\in L_p(0,1)$?

{\bf Problem 21.} (\cite{VT83}, 8.4). Does the Dual Greedy Algorithm with regard to $\cH_p$ converge in $L_p$, $1<p<\infty$, for each $f\in L_p(0,1)$?

{\bf Comment.} Problems 15, 19, and 21 are solved. The following theorem is proved in \cite{GK} (see also \cite{VTbook}, p.378, and Theorem \ref{conT3} and Proposition \ref{conP1} above). 

\begin{Theorem}\label{OPT1} Let $p\in(1,\infty)$. Then the WDGA($\tau$) with $\tau=\{t\}$, $t\in(0,1]$, converges for each dictionary and all $f\in L_p$.
\end{Theorem}

\subsection{Rate of convergence.} \label{OP2}

The problem of estimating different asymptotic characteristics of function classes is a very classical problem of approximation theory. There are myriads of papers, which belong to this area. Here we discuss some more asymptotic characteristics, which are of interest in sparse approximation with respect to dictionaries. We formulate the settings of these problems at three different levels in the case of the class $A_1(\cD)$ and a specific algorithm GA. 

{\bf Level 1. Specific $X$ and specific $\cD$.} For a specific greedy algorithm GA define 
$$
er_m(X,\cD,GA) :=   \sup_{f\in A_1(\cD)} \sup_{realizations\, of\, GA} \|f_m\|_X,
$$
where $\{f_m\}$ is a sequence of residuals of the element $f$ after application of the greedy algorithm GA.

{\bf Level 2. Specific $X$ and arbitrary $\cD$.} For a specific greedy algorithm GA define 
$$
er_m(X,GA) :=  \sup_{\cD} \sup_{f\in A_1(\cD)} \sup_{realizations\, of\, GA} \|f_m\|_X,
$$
where $\{f_m\}$ is a sequence of residuals of the element $f$ after application of the greedy algorithm GA.

{\bf Level 3. A collection $\cX$ of Banach spaces and arbitrary $\cD$.}
  Let $\cX$ be a collection of Banach spaces, for instance, for  fixed $1<q\le 2$ and $\ga >0$ define
$$
\cX(\ga,q) := \{X\,:\, \rho(u,X) \le \gamma u^q \}.
$$
Then for a specific greedy algorithm GA define 
$$
er_m(\cX,GA) := \sup_{X\in \cX} \sup_{\cD} \sup_{f\in A_1(\cD)} \sup_{realizations\, of\, GA} \|f_m\|_X.
$$

First of all, we are interested in dependence of the above characteristics on $m$. The characteristic $er_m(\cX,GA)$ can be used for evaluation of a given greedy algorithm GA -- the faster the decay of the sequence $\{er_m(\cX,GA)\}$, the better the algorithm (from the point of view of accuracy). The characteristic $er_m(X,\cD,GA)$ can be used in the following situation. Suppose that we are working on the sparse approximation problem in a given Banach space $X$ with respect to a given dictionary $\cD$. Then, as the reader can see from this paper, there are many different greedy algorithms to choose from.  The characteristic $er_m(X,\cD,GA)$ will help to compare the accuracy properties of the algorithms. 

The following general problem is important and seems to be difficult. 

{\bf Open problem 22.} For a given greedy algorithm GA find the right order of the sequence $er_m(\cX(\ga,q),GA)$, 
$1<q\le 2$ and $\ga >0$. 

Clearly, the general Open problem 22 contains a number of subproblems, when we specify the greedy algorithm GA. 
For instance, it might be the WCGA($\tau$), WGAFR($\tau$), XGAFR, GAWR($\tau,\br$), XGAR($\br$), and other greedy 
algorithms discussed above. 

Certainly, the following characteristics, which provide the benchmarks for the above characteristics are also of importance and interest.

{\bf Best $m$-term approximations.} 

For a specific $X$, specific $\cD$, and a given function class $\bF$ define (as above)
\be\label{OP6}
\sigma_m(\bF,\cD)_X :=   \sup_{f\in \bF}   \sigma_m(f,\cD)_X .
\ee
For a specific Banach space $X$ define
\be\label{OP7}
\sigma_m(X) := \sigma_m(A_1(\cD),\cD)_X .
\ee
Finally, for   a collection of Banach spaces $\cX$
  define 
\be\label{OP8}
\sigma_m(\cX) := \sup_{X\in \cX} \sigma_m(X).
\ee

Clearly, for any greedy algorithm GA we have
\be\label{OP9}
er_m(\cX,GA) \ge \sigma_m(\cX)  .
\ee
and
\be\label{OP10}
er_m(X,GA) \ge \sigma_m(X)  .
\ee

We note that in the study of $\sigma_m(\cX(\ga,q))$     we know that the answer 
is determined by three parameters $\ga$, $q$, and $m$. In the case of characteristics $er_m(\cX(\ga,q),GA)$ the answer 
is determined by three parameters $\ga$, $q$, and $m$ and by parameters of the greedy algorithm GA. 
In the study of characteristics $er_m(X,GA)$ and  $\sigma_m(X)$
we need to find an appropriate property of the Banach space $X$, which controls the above characteristics. 

Let us give a simple lower bound of the characteristic $\sigma_m(\cX(\ga,q))$. Let $1<q\le 2$. Consider $X= \ell_q$ with $\cD$ being a canonical basis $\cE:=\{e_j\}_{j=1}^\infty$. For a given $m\in \bbN$ define the function
$$
f:=  \sum_{j=1}^{2m} e_j.
$$
Then we have
$$
\sigma_m(f,\cE)_{\ell_q} \ge m^{1/q},\qquad (2m)^{-1} f \in A_1(\cE). 
$$
Relations (\ref{Lprho}) show that $\ell_q\in \cX(\ga,q)$ for any $1<q\le 2$ and $\ga \ge 1/q$. Therefore, we have 
\be\label{OP11}
\sigma_m(\cX(\ga,q)) \ge \frac{1}{2} m^{1/q-1},\quad  1<q\le 2,\quad \ga \ge 1/q. 
\ee
The matching upper bounds are given by the WCGA($t$) and WGAFR($t$) (see Theorem \ref{rcT1}). 

{\bf Open problem 23.} For a given weakness sequence $\tau$ satisfying (\ref{2.2}) find good upper and lower bounds for the quantities
$$
er_m(\cX(\ga,q),WCGA(\tau))\quad \text{and} \quad er_m(\cX(\ga,q),WGAFR(\tau)). 
$$

\subsection{Lebesgue-type inequalities}

 We now formulate some open problems on the Lebesgue-type inequalities from \cite{VTbookMA} 
 (see p.448). 
 Here we concentrate on open problems related to the multivariate approximation. 
In the majority of cases we do not know the optimal $\phi$ such that a basis from a given collection of bases is $\phi$-greedy with respect to the WCGA. We formulate some of these open problems. Here we consider the version WCGA($t$), $t\in (0,1]$, of the Weak Chebyshev Greedy Algorithm with the weakness parameter $t$. Here we use the notation $\mathcal R\mathcal T^d_p$ for the real $d$-variate trigonometric system normalized in the $L_p$. 

{\bf Open problem 24.} For a Banach space $L_p$, $1<p<\infty$, characterize almost greedy bases with respect to the WCGA. 

{\bf Open problem 25.} Is $\mathcal R\mathcal T^d_p$ an almost greedy basis with respect to the WCGA, in $L_p(\bbT^d)$, $1<p<\infty$?

{\bf Open problem 26.} Is $\mathcal H_p$ an almost greedy basis with respect to the WCGA, in $L_p$, $2<p<\infty$?

{\bf Open problem 27.} Is $\mathcal H_p^d$, $d\ge 2$, an almost greedy basis with respect to the WCGA, in $L_p$, $1<p<\infty$?

{\bf Open problem 28.} For each $L_p$, $1<p<\infty$, find the best $\phi$ such that any Schauder basis is a $\phi$-greedy basis with respect to the WCGA. 

{\bf Open problem 29.} For each $L_p$, $1<p<\infty$, find the best $\phi$ such that any unconditional basis is a $\phi$-greedy basis with respect to the WCGA. 

{\bf Open problem 30.} Is there a greedy-type algorithm $\cA$ such that the multivariate Haar system $\cH^d_p$ is an almost greedy basis of $L_p$, $1<p<\infty$, with respect to $\cA$?

 \Addresses
 
\end{document}